\documentclass[reqno,11pt]{amsart}
\usepackage[T1]{fontenc}
\usepackage{geometry}
\title[Instanton Complex and Morse-Bott inequalities on pseudomanifolds]{Witten Instanton Complex and Morse-Bott inequalities on stratified pseudomanifolds}
\date{\today}
\author{Gayana Jayasinghe} 
\email{mgsjayasinghe@gmail.com}

\author{Hadrian Quan} 
\email{hquan1@ucsc.edu}

\author{Xinran Yu}
\email{xinran4@illinois.edu}

\keywords{} 

\makeatletter
\renewcommand{\tocsection}[3]{%
  \indentlabel{\@ifnotempty{#2}{\bfseries\ignorespaces#1 #2\quad}}\bfseries#3}

\renewcommand{\tocsubsection}[3]{%
  \indentlabel{\@ifnotempty{#2}{\ignorespaces#1 #2\quad}}#3}
\let\tocsubsubsection\tocsubsection

    \renewcommand{\tocsubsubsection}[3]{%
  \indentlabel{\@ifnotempty{#2}{\ignorespaces#1 #2\quad}}#3}
  
\newcommand\@dotsep{4.5}
\def\@tocline#1#2#3#4#5#6#7{\relax
  \ifnum #1>\c@tocdepth 
  \else
    \par \addpenalty\@secpenalty\addvspace{#2}%
    \begingroup \hyphenpenalty\@M
    \@ifempty{#4}{%
      \@tempdima\csname r@tocindent\number#1\endcsname\relax
    }{%
      \@tempdima#4\relax
    }%
    \parindent\z@ \leftskip#3\relax \advance\leftskip\@tempdima\relax
    \rightskip\@pnumwidth plus1em \parfillskip-\@pnumwidth
    #5\leavevmode\hskip-\@tempdima{#6}\nobreak
    \leaders\hbox{$\m@th\mkern \@dotsep mu\hbox{.}\mkern \@dotsep mu$}\hfill
    \nobreak
    \hbox to\@pnumwidth{\@tocpagenum{\ifnum#1=1\bfseries\fi#7}}\par
    \nobreak
    \endgroup
  \fi}

\AtBeginDocument{%
\expandafter\renewcommand\csname r@tocindent0\endcsname{0pt}
}
\def\l@subsection{\@tocline{2}{0pt}{2.5pc}{5pc}{}}
\def\l@subsubsection{\@tocline{2}{0pt}{4.5pc}{5pc}{}}

\makeatother

\usepackage{amsmath}
\usepackage{amsfonts, mathrsfs} 
\usepackage{hyperref}
\usepackage{slashed} 
\usepackage{mathtools}
\usepackage{tikz-cd}
\allowdisplaybreaks

\usepackage{enumerate}

\numberwithin{equation}{section} 

\newtheorem{theorem}{Theorem}[section]
\newtheorem{proposition}[theorem]{Proposition}
\newtheorem{lemma}[theorem]{Lemma}
\newtheorem{corollary}[theorem]{Corollary}

\newtheorem{remark}[theorem]{Remark}
\newtheorem{example}[theorem]{Example}
\theoremstyle{definition}
\newtheorem{definition}[theorem]{Definition}

\usepackage[commandnameprefix=always,  defaultcolor=orange!70!black]{changes}

\usepackage{adjustbox}
\usepackage{graphicx}
\graphicspath{ {images/} }
\usepackage{caption}
\usepackage{subcaption}
\usepackage{float}
\usepackage{comment}

\begin{document}
\pagenumbering{roman} 

\maketitle
\date{\today}

\begin{abstract}

In this paper we construct Witten instanton complexes on stratified pseudomanifolds with wedge metrics, for all choices of mezzo-perversities which classify the self-adjoint extensions of the Hodge Dirac operator. In this singular setting we introduce a generalization of the Morse-Bott condition and in so doing can consider a class of functions with certain non-isolated critical point sets which arise naturally in many examples. This construction of the instanton complex extends the Morse polynomial to this setting from which we prove the corresponding Morse inequalities.

This work proceeds by constructing Hilbert complexes and normal cohomology complexes, including those corresponding to the Witten deformed complexes for such critical point sets and all mezzo-perversities; these in turn are used to express local Morse polynomials as polynomial trace formulas over their cohomology groups.
Under a technical assumption of `flatness' on our Morse-Bott functions we then construct the instanton complex by extending the local harmonic forms to global quasimodes.
We also study the Poincar\'e dual complexes and in the case of self-dual complexes extract refined Morse inequalities generalizing those in the smooth setting. We end with a guide for computing local cohomology groups and Morse polynomials.

\end{abstract}

\tableofcontents
\pagenumbering{arabic}

\section{Introduction}

Witten studied Morse theory and its relationship with physics in a number of articles in the early 80's \cite{witten1981dynamical,witten1982supersymmetry,witten1983fermion,witten1984holomorphic}, extending Morse inequalities and instanton complexes that categorify the Morse polynomials. In \cite{witten1982supersymmetry} he suggested an approach to study Morse inequalities for functions with non-isolated critical point sets, which was systematically developed by Bismut in \cite{bismut1986witten} for Morse-Bott functions. 

Morse inequalities on singular spaces have been explored for intersection cohomology \cite{goresky1988stratified,kirwan1988intersection} and for $L^2$ cohomology for certain choices of extensions of the de Rham complex in \cite{Jesus2018Wittens,Jesus2018Wittensgeneral}. Witten instanton complexes in the de Rham case were constructed under various conditions in \cite{ludwig2013analytic,ludwig2017index,ludwig2024morse,jayasinghe2023l2} where the function is assumed to have isolated critical points on the stratified pseudomanifold, and for the holomorphic case in \cite{jayasinghe2024holomorphic}.

Given a stratified pseudomanifold $\widehat{X}$ there is a corresponding resolved manifold with corners with iterated fibration structures $X$. For such spaces equipped with a wedge metric (roughly an iterated conic metric on) and a flat vector bundle $E$, the Hodge theory for twisted de Rham complexes $\mathcal{P}_W(X)=(L^2\Omega^{\cdot}(X;E), \mathcal{D}_W(d_E),d_E)$ was worked out in detail in \cite{Albin_hodge_theory_cheeger_spaces,Albin_hodge_theory_on_stratified_spaces} (c.f. \cite{albin2017novikov}) including a classification of self-adjoint extensions of the Hodge de Rham Dirac operator in terms of \textit{mezzo-perversities}.
In this article we construct the de Rham Witten instanton complexes for what we call \textit{flat stratified Morse-Bott functions} corresponding to all such self-adjoint extensions. 

\textit{As far as we are aware this article contains the first construction of Witten instanton complexes for stratified pseudomanifolds with non-isolated critical points.} One of the important contributions we make is the construction of various local Hilbert complexes for different types of critical point sets of a given \textbf{stratified Morse-Bott function} $h$. 
We consider three types of connected components of critical point sets $\widehat{F_a}$ in Subsection \ref{subsection_stratified_Morse_Bott_functions}. Those that are smoothly embedded in the regular part of $\widehat{X}$, those that are smoothly embedded in a singular stratum of $\widehat{X}$, and those that are stratified spaces themselves, with the regular part being a sub-manifold of a stratum of $\widehat{X}$, with a compatibility of the Thom-Mather data of $\widehat{F_a}$ and $\widehat{X}$. 
We call the latter two as type I and type II stratified non-degenerate critical point sets respectively.

In addition to tubular neighborhoods of $\widehat{F_a}$ in $\widehat{X}$ given by the Thom-Mather data, we assume that there are \textit{fundamental neighbourhoods} $U(\widehat{F_a})$ that depend on the gradient flow of the function $h$ in the neighborhood which refines the stratification in a neighbourhood. In general, these fundamental neighbourhoods are not actual neighbourhoods of the critical point sets (see Remark \ref{Remark_not_a_neighbourhood}), but are spaces constructed by modifying dense subsets of neighbourhoods where technically complicated measure zero sets of actual neighbourhoods of critical point sets are avoided or modified. Our philosophy for studying $L^2$ invariants is to understand the space of sections on such dense subsects of spaces, and take suitable graph closures of the operators being studied, obtaining domains suitable for the analytic constructions we do.
We emphasize that we do not demand any fiber bundle structure on $U(\widehat{F_a})$, since any such notion would carry with it additional complications since the candidate base is stratified. We follow the philosophy that for the purposes of studying $L^p$ cohomology on spaces, one only has to understand the space of sections on a dense set, and find domains for the operators by taking appropriate (ideal) boundary data. In studying local invariants corresponding to global domains, we will consider restrictions of sections in the global domains to the dense sets and define local domains using appropriate graph closures. This is done on a set which corresponds to a resolution of the above singular fundamental neighbourhoods that we define next.

We emphasize here that the nomenclature of (resolved) fundamental neighbourhood is an abuse of language since the set we define is not strictly a neighbourhood of the fixed point set $F_a$ in the case of stratified critical point sets of Type II.

In particular we demand that there is a fiber bundle over the regular part of each $\widehat{F_a}$ which is dense in the fundamental neighbourhood. If the fiber bundle over the regular part of $\widehat{F_a}$ is flat for each such connected component, we say that $h$ is a \textbf{flat stratified Morse-Bott function} and this allows us to pick a \textbf{local model metric} near $F_a$ where the fiber metrics are isometric, and extend it to a \textbf{model wedge metric} on the entire space (see Definition \ref{Definition_flatness_assumption}, Proposition \ref{Proposition_model_metric_existence} and accompanying remarks), simplifying the spectral theory. Since $L^2$ cohomology is isomorphic for all wedge metrics, this suffices for the study of many interesting problems related to $L^2$ cohomology.

The Witten instanton complex $\mathcal{P}_{W,inst}(X)$ is constructed by introducing a deformed operator $d_\varepsilon=d_E+\varepsilon dh \wedge$ and a deformed complex $\mathcal{P}_{W,\varepsilon}(X)$ in the setting described above. The associated deformed Laplace-type operator $\Delta_\varepsilon$ has a sub complex given by eigensections corresponding to \textit{small-enough} eigenvalues, earning this the alternate name of small eigenvalue complex.
We introduce \textit{local Hilbert complexes} $\mathcal{P}_{W,B}(U(F_a))=(L^2\Omega^{\cdot}(U(F_a);E), \mathcal{D}_{W,B}(P_{U(F_a)}),P_{U(F_a)})$ and \textit{normal cohomology complexes} $\mathcal{R}_{W,B}(U(F_a))$ for \textit{fundamental neighbourhoods} of connected components of critical points sets $F_a$, where $\cup_{a \in \mathcal{I}} F_a$ is the critical point set of $h$, and construct Witten deformations of these complexes as well. These normal complexes are roughly the forms on $F$ twisted by the system of local coefficients corresponding to deformed cohomology groups on the normal fibers, up to functional analytic subtleties. 
The associated deformed Laplace-type operator is denoted $\Delta_{\varepsilon}$ and the sections of the Witten instanton complex are constructed using the harmonic sections of these local complexes over all critical points. The normal cohomology complex is used to understand the spectrum of the local Hilbert complex, and we introduce these in detail in Subsection \ref{subsubsubsection_Neumann_boundary_condition}.
The following is the main result and is a version of Theorem \ref{theorem_small_eig_complex}.

\begin{theorem}[de Rham Witten instanton complex]
\label{theorem_small_eig_complex_intro_version}
Let $X$ be a resolution of a stratified pseudomanifold of dimension $n$ equipped with a model wedge metric and a flat stratified Morse-Bott function $h$. Let $E$ be a flat vector bundle on $X$, and let $\mathcal{P}_W(X)=(L^2\Omega^{\cdot}(X;E), \mathcal{D}_W(d_E),d_E)$ be a twisted de Rham complex on $X$ with a choice of mezzo-perversity $W$. 

For any integer $0 \leq k \leq n$, let $
\mathrm{K}_{W,\varepsilon, k}^{[0, c]} \subset L^2\Omega^k(X;E)$ denote the vector space generated by the eigenspaces of $\Delta_{\varepsilon}$ associated with eigenvalues in $[0, c]$. There exist constants $c>0$ and $\varepsilon_0>0$ such that $\dim \mathrm{K}_{W,\varepsilon, k}^{[0, c]} =\dim \sum_{a \in \mathcal{I}} \mathcal{H}^k(\mathcal{P}_{W,B}(U(F_a))$ when $\varepsilon>\varepsilon_0$, and together form a finite dimensional sub complex,
\begin{equation}
    \label{small_eigenvalue_complex_intro}
\mathcal{P}_{W,inst}(X):=(\mathrm{K}_{W,\varepsilon, k}^{[0, c]}, P_{\varepsilon}): 0 \longrightarrow \mathrm{K}_{W,\varepsilon, 0}^{[0, c]} \stackrel{P_{\varepsilon}}{\longrightarrow} \mathrm{K}_{W,\varepsilon, 1}^{[0, c]} \stackrel{P_{\varepsilon}}{\longrightarrow} \cdots \stackrel{P_{\varepsilon}}{\longrightarrow} \mathrm{K}_{W,\varepsilon, n}^{[0, c]} \longrightarrow 0
\end{equation}
of $\mathcal{P}_{W,\varepsilon}(X)$ which is quasi-isomorphic to it.
\end{theorem}

\begin{definition}
\label{Definition_Morse_and_Poincare_polynomials}
In the setting of Theorem \ref{theorem_small_eig_complex_intro_version} we define the \textbf{Morse polynomial} for a complex $\mathcal{P}_W(X)$ and a flat stratified Morse-Bott function $h$ by
\begin{equation}
     M(\mathcal{P}_W(X),h)(b):=\sum_{k=0}^n b^k \dim \mathrm{K}_{W,\varepsilon, k}^{[0, c]}
\end{equation}
for small $c$ and large $\varepsilon$ as in the theorem above.
Given any Fredholm complex $\mathcal{S}$, we define the \textbf{Poincar\'e polynomial} for the complex by
\begin{equation}
    P(\mathcal{S})(b):=\sum_{k=0}^n b^k \dim(\mathcal{H}^{k}(\mathcal{S})).
\end{equation}
Given a local Hilbert complex on a fundamental neighbourhood of a component of the critical point set, we define the \textbf{\textit{local} Morse polynomial} to be the Poincar\'e polynomial of the local complex.
\end{definition}

It is easy to see from Theorem \ref{theorem_small_eig_complex_intro_version} that 
\begin{equation}
    M(\mathcal{P}_W(X),h)(b)=\sum_{a \in \mathcal{I}} P(\mathcal{P}_{W,B}(U(F_a)))=\Big( \sum_{a \in \mathcal{I}}  \sum_{k=0}^n b^k \dim(\mathcal{H}^{k}(\mathcal{P}_{W,B}(U(F_a))) \Big)
\end{equation}
and we can formulate the strong version of the (polynomial) Morse inequalities as follows.

\begin{theorem}
\label{Theorem_strong_Morse_Bott_de_Rham_intro_version}
[Polynomial Morse inequalities]
In the setting of Theorem \ref{theorem_small_eig_complex_intro_version}, 
there exist non-negative integers $Q_0,..., Q_{n-1}$ such that
\begin{equation}
    \label{Morse_Bott_inequality_de_Rham_dimension_cohomology_intro_version}
    M(\mathcal{P}_W(X),h)(b)=  P(\mathcal{P}_W(X))(b) + (1+b) \sum_{k=0}^{n-1} Q_k b^k.
\end{equation}
\end{theorem}

This is a version of Theorem \ref{Theorem_strong_Morse_Bott_de_Rham} and generalizes the Morse inequalities to the introduced setting.
The vanishing of the \textbf{error polynomial} $(1+b) \sum_{k=0}^{n-1} Q_k b^k$ when we set $b=-1$ corresponds to a Lefschetz fixed point theorem for the self-map induced by the gradient flow of the function $h$.

Given a complex $\mathcal{P}_W(X)$, there is an \textbf{adjoint complex} 
$\mathcal{Q}_{W^{\perp}}(X)=\big(L^2\Omega^{\cdot}(X;E), \mathcal{D}_{W^{\perp}}(\delta_E),\delta_E\big)$ and an \textbf{adjoint Poincar\'e dual complex} $\mathcal{Q}_{\star W}(X)=(L^2\Omega^{\cdot}(X;E), \mathcal{D}_{\star W}(\delta_E),\delta_E)$, and while these are equal in the smooth setting, they are only equal to each other if the mezzo-perversity is a \textbf{self-dual mezzo perversity} ($\star W=W^{\perp}$, or $W=\star W^{\perp}$), and we introduce these notions in detail in Section \ref{section_Hilbert_complexes_and_cohomology}. This is certainly the case on \textbf{Witt spaces} where the Hodge Laplacian is essentially self-adjoint and there is a unique mezzo perversity $W$.

\begin{remark}[Notation and grading convention]
\label{Remark_Notation_and_grading_P_vs_Q}
In this article we will use the notation $\mathcal{P}$ and $\mathcal{Q}$ (with additional decorations), to denote the complexes corresponding to the operators $d$ and $\delta$ respectively. The grading for sections in the $\mathcal{P}$ complexes correspond to the form degrees while the grading for forms of degree $k$ in the $\mathcal{Q}$ complexes as Hilbert complexes correspond to $n-k$ where $n$ is the dimension of the space. \textbf{However we will use the grading given by the form degree for both complexes when constructing Morse and Poincar\'e polynomials.} This is to be consistent with topological versions of Morse inequalities appearing in the literature.
\end{remark}
With this convention it is easy to check that $b^nM(\mathcal{P}_W(X),h)(b^{-1})=M(\mathcal{Q}_{W^{\perp}}(X),-h)(b)$ (see Remark 
\ref{Remark_on_adjoint_inequalities_proof}). In addition to Morse inequalities coming from the adjoint complex, we have those corresponding to the Poincar\'e dual complex 
\begin{multline}
\label{Dual_Morse_Bott_inequality_de_Rham_dimension_cohomology_intro_version}
     M(\mathcal{Q}_{\star W}(X),h)(b)=\Big( \sum_{a \in \mathcal{I}}  \sum_{k=0}^n b^k \dim \mathcal{H}^{k}(\mathcal{Q}_{\star W, \star B}(U(F_a))) \Big) \\= \sum_{k=0}^n b^k \dim(\mathcal{H}^{k}(\mathcal{Q}_{\star W}(X))) + (1+b) \sum_{k=0}^{n-1} Q'_{k} b^{k}.
\end{multline}
It is well known in the \textbf{smooth} setting that $b^nM(\mathcal{P}_W(X),h)(b^{-1})=M(\mathcal{P}_W(X),-h)(b)$ and that this is a consequence of Poincar\'e duality and yields more information, but now this only holds for self-dual mezzo perversities (i.e. when $W=\star W^{\perp}$) as we formalize in Corollary \ref{corollary_Poincare_dual_inequalities_intro_version} (see Example \ref{example_breakdown_refined_inequalities} for a case where the relation fails when the mezzo-perversity is not self-dual). We define refined Morse polynomials as follows. 

\begin{definition}[Refined Morse polynomials]
\label{Definition_refined_Morse_polynomials}
Given two polynomials $M_1(b)=\sum_k M_{1,k}b^k$ and $M_2(b)=\sum_k M_{2,k}b^k$, where the coefficients $M_{1,k}, M_{2,k}$ are non-negative integers, we define the \textbf{minimal polynomial} $M_{\min}(b):=\sum_k M_{\min,k}b^k$ where $M_{\min,k}:=\min \{M_{1,k}, M_{2,k} \}$.

Let us assume that there is a self-dual mezzo-perversity $W$. Then we define the \textbf{\textit{refined Morse polynomial}} $M_{re}(\mathcal{P}_W(X),h,b)$ to be the minimal polynomial of $M(\mathcal{P}_W(X),h)(b)$ and $M(\mathcal{P}_W(X),-h)(b)$.
\end{definition}

The following corollary generalizes the well known result of Poincar\'e duality for Morse polynomials in the smooth setting. Now however there is a complication in its utility since the dual complex is not isomorphic to the adjoint complex in general. The following is a version of Corollary \ref{corollary_Poincare_dual_inequalities}.

\begin{corollary}[Refined inequalities for self-dual mezzo perversities]
\label{corollary_Poincare_dual_inequalities_intro_version}
In the same setting as Theorem \ref{theorem_small_eig_complex_intro_version}, 

if the mezzo-perversity $W$ is self-dual, we have that
\begin{equation}
    \label{equation_Poincare_duality_for_polynomials_intro_version}
    M(\mathcal{P}_W(X),h)(b)=b^nM(\mathcal{P}_W(X),-h)(b^{-1}), 
\end{equation}
and there is a \textbf{refined polynomial Morse inequality}
\begin{equation}
\label{equation_refined_inequalities_intro_version}
    M_{re}(\mathcal{P}_W(X),h,b)- \sum_{k=0}^n b^k \dim(\mathcal{H}^{k}(\mathcal{P}_W(X))) = \sum_{k=0}^{n} \overline{Q}_{k} b^{k}
\end{equation}
where $\overline{Q}_{k}$ are non-negative integers.
\end{corollary}
We note that there is no $(1+b)$ factor on the right hand side as in the standard Morse inequalities (see Example \ref{Example_two_horned_torus}).

In the smooth setting, the notions of perfect Morse/Morse-Bott functions are defined after fixing a ring of coefficients for the cohomology group. Now we must consider the mezzo-perversity of the complex used.
\begin{definition}[Perfect stratified Morse/Morse-Bott functions]
\label{definition_perfect_morse_functions}
    For a \textit{fixed mezzo-perversity}, we say that a stratified Morse/Morse-Bott function $h$ is \textbf{perfect at degree $k$} if the cohomology of the corresponding instanton complex at degree $k$ has the same dimension as the cohomology at degree $k$ of the $L^2$ de Rham complex.
    We say \textbf{$h$ is perfect} if it is perfect at each degree $k$, equivalently if the Morse polynomial is equal to the Poincar\'e polynomial.
\end{definition}

\subsection{Overview of sections}

The breakdown of each section is as follows. In the rest of this section we survey and discuss related results and ideas.

In Section \ref{section_background} we review stratified pseudomanifolds and their resolutions and discuss wedge metrics and related structures. We then introduce flat stratified Morse-Bott functions and fundamental neighbourhoods of critical points as well as associated resolved versions. 

In Section \ref{section_Hilbert_complexes_and_cohomology} we begin by discussing abstract Hilbert complexes before discussing twisted de Rham complexes. Here we review the choices of self-adjoint extensions of the corresponding Dirac operators following \cite{Albin_hodge_theory_cheeger_spaces,Albin_hodge_theory_on_stratified_spaces} before discussing the corresponding restricted complexes on fundamental neighbourhoods of critical points of stratified Morse-Bott functions, where we \textit{do not} assume the flatness condition in Definition \ref{Definition_flatness_assumption}.
We also introduce normal cohomology complexes on fundamental neighbourhoods before exploring Poincar\'e dual complexes and self-dual mezzo-perversities. We then discuss some results on polynomial supertraces that are used to prove Morse inequalities once the instanton complexes are constructed.

In Section \ref{section_Witten_deformation} we discuss the Witten deformed versions of the complexes introduced in Section \ref{section_Hilbert_complexes_and_cohomology}. We then develop localization results by studying Bochner type formulas for Witten deformed Laplace-type operators, and study spectral properties for local Laplace-type operators for flat stratified Morse-Bott functions.

In Section \ref{section_technical} we prove the main results that we discussed in the introduction. We give a computational guide in Section \ref{Section_computational_guide} where we study illustrative examples.

\subsection{Background and related results}
\label{subsection_History_and_motivations}

Morse theory is a rich subject with applications in various areas, where the critical points of functions and functionals on various spaces can be used to study the topology of such spaces. Morse studied the critical points of an energy functional on the path space to prove the existence of infinitely many geodesics. This follows from showing that the homology of the path space is infinite dimensional and that each homology class contains a geodesic representative. Bott extended those ideas to compute the homotopy groups of the classical Lie groups, and studied what are now known as Morse-Bott functions where the critical points are not isolated. In \cite{AtiyahBottYangMillsRiemann83} Atiyah and Bott studied the critical points of the Yang-Mills functional on Riemann surfaces and showed that the critical points are non-isolated, and developed equivariant Morse-Bott theory when there are group actions. Kirwan extended this in algebraic and symplectic settings where Hamiltonians are Morse-Bott-Kirwan \cite{Kirwancohomofquot84}, and this is still a topic of great interest (see, e.g., \cite{feehan2022bialynicki,austin1995morse}).

Witten discovered a novel perspective on Morse theory while studying supersymmetry breaking (see \cite{witten1981dynamical,witten1982constraints,witten1982supersymmetry}) which he developed further for twisted Dolbeault complexes in \cite{witten1983fermion,witten1984holomorphic} when there is a K\"ahler isometry. If the dynamical systems approach to Morse theory by Smale is viewed as a classical mechanical system, then Witten proposed quantizing the system and studying the semi-classical limit $\hbar=1/\varepsilon \rightarrow 0$. The classical ground states are the minima of the potential $|dh|^2$, i.e. the critical points of the Morse function $h$, and the quantum ground states are well approximated by Gaussians which converge to delta functions at the classical ones. Doing this procedure on the $L^2$ bounded forms by deforming the de Rham (Dolbeault) differentials in a manner that the supersymmetry is preserved one shows that the direct sum of localizing Gaussians, isomorphic to the null space of the local Hodge Laplacian with suitable boundary conditions, forms a complex that is quasi-isomorphic to the undeformed complex. There is an underlying $\mathcal{N}=2$ supersymmetric quantum mechanical system, where the SUSY operators are $Q_1=d_\varepsilon+d^*_\varepsilon$ and $Q_2=\sqrt{-1}(d_\varepsilon-d^*_\varepsilon)$ (see \cite{witten1982supersymmetry}).
Generalizing this richer description of the Morse inequalities to singular spaces comes with unique difficulties (compared to algebraic and topological approaches) since one has to study not only cohomology groups but also the spectral theory of Schr\"odinger operators and their semi-classical limits. 

The case of Morse-Bott functions on smooth manifolds was already addressed in \cite{witten1982supersymmetry}, then studied more thoroughly in \cite{bismut1986witten} by Bismut. A simpler proof technique was used for Morse-Bott-Kirwan critical point sets arising as critical points of K\"ahler Hamiltonian functions in \cite{wu1998equivariant} by Wu and Zhang for twisted Dolbeault complexes, building on work of Bismut-Lebeau in \cite{bismut1991complex} (c.f. the exposition in \cite{Zhanglectures}).
We mainly follow this approach, generalizing it to our setting.
Kirwan and Penington \cite{Kirwan_Morse_21} generalized the Witten instanton complexes to functions with arbitrary critical point sets (with finitely many connected components) on smooth manifolds. 

In the singular setting, Morse theory was studied by Goresky and MacPherson \cite{goresky1983intersection} who formulated Morse inequalities. The Morse inequalities in $L^2$ cohomology for stratified pseudomanifolds for a large class of metrics were proven in \cite{Jesus2018Wittens,Jesus2018Wittensgeneral} for the minimal and maximal domains for the de Rham complex, and the Witten instanton complexes were constructed in \cite{ludwig2013analytic,ludwig2017index}. 
Analytic techniques to study Morse theory for Hilbert complexes on stratified pseudomanifolds were developed in \cite{jayasinghe2023l2,jayasinghe2024holomorphic} for twisted de Rham and Dolbeault complexes, constructing instanton complexes for the case of wedge metrics where the complexes satisfy a Witt condition.
The latter work built techniques to study Lefschetz fixed point theorems, equivariant index theorems and Morse inequalities since these are highly intertwined.
our article must be considered in that light as an exploration of techniques that can be used to extend other related formulas used in mathematics and physics to Hilbert spaces corresponding to singular spaces.

We emphasize that in this article we are not simply considering Morse-Bott functions $h$ on a resolution of the space which correspond to lifts of functions which have isolated critical points on the stratified space, such as in \cite{ludwig2024morse}. Rather we are considering functions on stratified spaces $\widehat{X}$ which have non-isolated critical point sets on the stratified space itself, and while such settings have been studied using algebraic and topological techniques as discussed earlier, our article seems to be the first to do so using analytic means. It is also the first to construct Witten's instanton complexes for choices of mezzo-perversities other than those corresponding to the minimal and maximal domains, and for non-isolated critical point sets of any kind on stratified pseudomanifolds. A key innovation here is in defining the local Hilbert complexes and normal cohomology complexes associated to critical point sets, even when the critical point sets are stratified. These generalize local Hilbert complexes for fundamental neighborhoods in the case of isolated critical point sets used in \cite{jayasinghe2023l2,jayasinghe2024holomorphic} and we can compute local Morse polynomials in many examples with these definitions, even without the flatness condition for stratified Morse-Bott functions that we use to prove our main results. In the smooth setting it is known (see, e.g., \cite{Orientation_Morse_Bott_Rot_2016}) that the terms contributing to the Morse polynomial from a Morse-Bott critical point set is given by a polynomial supertrace over cohomology with local coefficients in the expanding (negative) normal bundles. In the singular setting both the attracting and expanding normal bundles give contributions (see Example \ref{example_suspension_suspension_torus_2}).

Witten notes that the analysis of the local eigenstates near non-isolated critical point sets in the smooth setting is similar to Born-Oppenheimer approximation, and this was studied rigorously in \cite{bismut1986witten}. 
The flat condition in Definition \ref{Definition_flatness_assumption} that we assume for the functions that we study allows for simpler techniques for the computation of the spectrum, but is also an assumption that has been used for simpler computational techniques of local cohomology groups using topological tools (see, e.g. \cite{BanaglFlatfiber2013} which includes interesting examples).

The analysis of non-isolated critical point sets in the smooth setting is related to equivariant localization, widely used in various applications of physics, for instance in computations of blackhole entropy in which these critical point sets are called bolts \cite{gibbons1979classification,Equivariant_loc_Sparks_2024}. Related invariants were extended to isolated critical point sets in \cite{jayasinghe2024holomorphic} and our study of local Hilbert complexes sets the stage to extend these results to the singular setting.

We saw in Corollary \ref{corollary_Poincare_dual_inequalities_intro_version} that the refined Morse inequalities can be constructed only when the mezzo-perversity is self-dual. The key to this breakdown is the fact that the local cohomology of attracting vs. expanding critical point sets correspond to those of the local complex and its adjoint complex and if the gradient flow of $h$ has an attracting critical point set, then it is expanding for the gradient flow of $-h$, and vice versa. This phenomenon is also captured in our construction of local Hilbert complexes near critical point sets.
Since the adjoint complex is can be identified with the Poincar\'e dual complex in the case of self-dual mezzo-perversities, we get refined polynomial Morse inequalities. Example \ref{example_breakdown_refined_inequalities} shows that this result does not hold without self-duality. 

The gradient flow for $-h$ corresponds to the time reversed gradient flow for $h$, and this lifts to a time reversal on the quantized system implemented by the Hodge star operator on forms, even before taking a semi-classical limit. In the singular setting, if a quantization of the gradient flow of $h$ corresponds to the Witten deformed complex with mezzo-perversity $W$, then the time reversed system corresponds to the complex with the dual mezzo-perversity $\star W$, and there is time reversal symmetry only if the mezzo-perversity is self-dual.

In the case of holomorphic Morse inequalities formulated by Witten in \cite{witten1984holomorphic}, the Morse polynomials are polynomials in $b$ where the coefficients are infinite series in powers of $e^{i\theta}$ (up to shifts for twisted complexes). However the refined Morse polynomials are polynomials in $b$ where the coefficients are finite series in $e^{i\theta}$. This was extended in \cite{jayasinghe2024holomorphic} to the singular setting with isolated critical points where the complex and the Serre dual complex are related by the Hodge star operator. While the Dolbeault-Dirac operators are not essentially self-adjoint, the Witt condition reduces the complications arising in the dualities. 
Witten deformation for spin-c Dirac operators on symplectic manifolds is quite similar to the Dolbeault case \cite{tian1998holomorphic} and we anticipate that there will be similar phenomena as those studied in this article when studying dualities of spin-$\mathbb{C}$ Dirac complexes on non-Witt spaces with certain choices of domains of operators. 

It is well known that Morse inequalities reduce to Lefschetz fixed point theorems on evaluating the Morse polynomials at $b=-1$. However Lefschetz fixed point theorems hold for more general self maps, not just those that arise from gradient flows of functions. When the critical points are isolated and the space is Witt, this reduces to the de Rham case of the Lefschetz fixed point theorem given in \cite{jayasinghe2023l2}.

Certain other invariants such as the equivariant signature and equivariant de Rham type Poincar\'e Hodge inequalities can be studied when the stratified Morse-Bott function is K\"ahler Hamiltonian extending Theorem 1.10 of \cite{jayasinghe2024holomorphic}, using the instanton complex we construct here, as long as the K\"ahler metric is of product type in a neighbourhood of a critical point set. This can be done without constructing a more general Dolbeault instanton complex, since the Witten deformation will then respect the $(p,q)$ decomposition. We will discuss these in a future article on holomorphic Morse inequalities in the presence of non-isolated critical points.

\textbf{Acknowledgements:} 
The authors would like to thank Pierre Albin for many helpful conversations, and organizing a reading group on topics related to this article.
GJ thanks Jes\'us \'Alvarez L\'opez for useful comments on related work, and also thanks the University of Peradeniya for its hospitality while completing this work.

\section{Stratified spaces and Morse-Bott functions}
\label{section_background}

In this section, we review stratified spaces and their resolutions, and study the structure of critical point sets of \textit{flat stratified Morse-Bott functions}.

\subsection{Stratified pseudomanifolds}

\subsubsection{Stratified pseudomanifolds}

All topological spaces we consider will be Hausdorff, locally compact topological spaces with a countable basis for their topology. Recall that a subset $W$ of a topological space $V$ is locally closed if every point $a \in W$ has a neighborhood $\mathcal{U}$ in $V$ such that $W \cap \mathcal{U}$ is closed in $\mathcal{U}$. A collection ${S}$ of subsets of $V$ is locally finite if every point $v \in V$ has a neighborhood that intersects only finitely many sets in ${S}$.

\begin{definition}
    \label{Thom-Mather stratified space}    
A Thom-Mather stratified space $\widehat{X}$ is a metrizable, locally compact, second countable space which admits a locally finite decomposition into a union of locally closed strata $\mathcal{S}(\widehat{X})=\left\{Y_{\alpha}\right\}$, where each $Y_{\alpha}$ is a smooth manifold, with dimension depending on the index $\alpha$. We require the following:
\begin{enumerate}
    \item If $Y_{\alpha}, Y_{\beta} \in \mathcal{S}(X)$ and $Y_{\alpha} \cap \overline{Y_{\beta}} \neq \emptyset$, then $Y_{\alpha} \subset \overline{Y_{\beta}}$.

    \item Each stratum $Y$ is endowed with a set of `control data' $\mathcal{T}_Y, \pi_{Y}$ and $\rho_{Y}$; here $\mathcal{T}_Y$ is a neighbourhood of $Y$ in $X$ which retracts onto $Y, \pi_{Y}: \mathcal{T}_Y \longrightarrow Y$ is a fixed continuous retraction and $\rho_{Y}: \mathcal{T}_Y \rightarrow[0,2)$ is a `radial function' in this tubular neighbourhood such that $\rho_{Y}^{-1}(0)=Y$. Furthermore, we require that if $\widetilde{Y} \in \mathcal{S}(X)$ and $\widetilde{Y} \cap \mathcal{T}_Y \neq \emptyset$, then
\begin{equation}    
    \left(\pi_{Y}, \rho_{Y}\right): \mathcal{T}_Y \cap \widetilde{Y}  \setminus Y \longrightarrow Y \times (0,2)
\end{equation}
is a differentiable submersion.

\item If $W, Y, \widetilde{Y} \in \mathcal{S}(X)$, and if $p \in \mathcal{T}_Y \cap \mathcal{T}_{\widetilde{Y}} \cap W$ and $\pi_{\widetilde{Y}}(p) \in \mathcal{T}_Y \cap \widetilde{Y}$, then $\pi_{Y}\left(\pi_{\widetilde{Y}}(p)\right)=\pi_{Y}(p)$ and $\rho_{Y}\left(\pi_{\widetilde{Y}}(p)\right)=\rho_{Y}(p)$.
\end{enumerate}
\end{definition}

From this definition, and due to Thom's first isotopy lemma, we have structure of neighborhoods as locally trivial fibrations over strata $Y$, and the following are some direct consequences.

\begin{enumerate}
\item If $Y, \widetilde{Y} \in \mathcal{S}(X)$, then
$$
\begin{aligned}
Y \cap \overline{\widetilde{Y}} \neq \emptyset & \Leftrightarrow \mathcal{T}_Y \cap \widetilde{Y} \neq \emptyset, \\
\mathcal{T}_Y \cap \mathcal{T}_{\widetilde{Y}} \neq \emptyset & \Leftrightarrow Y \subset \overline{\widetilde{Y}}, Y=\widetilde{Y} \quad \text {or } \widetilde{Y} \subset \overline{Y} .
\end{aligned}
$$

\item For each $Y \in \mathcal{S}(X)$, the restriction $\pi_{Y}: \mathcal{T}_Y \rightarrow Y$ is a locally trivial fibration with fibre the cone $C\left(Z_{Y}\right)$ over some other stratified space $Z_{Y}$ (called the \textbf{\textit{link}} over $Y)$, with atlas $\mathcal{U}_{Y}=\{(\phi, \mathcal{U})\}$ where each $\phi$ is a trivialization 
\begin{equation}
\label{equation_chartlike_map}
\pi_{Y}^{-1}(\mathcal{U}) \rightarrow \mathcal{U} \times C\left(Z_{Y}\right),    
\end{equation}
and the transition functions are stratified isomorphisms of $C\left(Z_{Y}\right)$ which preserve the rays of each conic fibre as well as the radial variable $\rho_{Y}$ itself, hence are suspensions of isomorphisms of each link $Z_{Y}$ which vary smoothly with the variable $y \in \mathcal{U}$.
\end{enumerate}

This structure shows that there is a filtration of $\widehat{X}$ of the form
\begin{equation}
    \widehat{X}_n \supseteq  \widehat{X}_{n-1} \supseteq  \widehat{X}_{n-2}... \supseteq \widehat{X}_{0}
\end{equation}
where $\widehat{X}_{j}$ is the union of strata $Y_{\alpha} \subset S(\widehat{X})$ of dimension less than or equal to $j$. It is clear that $\widehat{X}_j \backslash \widehat{X}_{j-1}$ is a smooth manifold of dimension $j$. We define the \textit{\textbf{regular part}} of $\widehat{X}$ to be $X^{reg}:=\widehat{X}_n \setminus \widehat{X}_{n-1}$. 

If in addition if $\widehat{X}_{n-1}=\widehat{X}_{n-2}$ and $\widehat{X} \backslash \widehat{X}_{n-2}$ is dense in $\widehat{X}$, 
then we say that $\widehat{X}$ is a \textbf{stratified pseudomanifold}.

There is a functorial equivalence between Thom-Mather stratified spaces and manifolds with corners and iterated fibration structures (see Proposition 2.5 of \cite{Albin_signature}, Theorem 6.3 of \cite{Albin_hodge_theory_cheeger_spaces}).

\subsubsection{Manifolds with corners with iterated fibration structures}

In section 1 of \cite{Albin_2017_index}, there is a detailed description of iterated fibration structures on manifolds with corners. We explain some of these structures from said article which we will use and refer the reader to the source for more details. In \cite{kottke2022products}, these are referred to as manifolds with fibered corners that are interior maximal.

An $n$-dimensional manifold with corners $X$ is an $n$-dimensional topological manifold with boundary, with a smooth atlas modeled on $(\mathbb{R}^+)^n$ whose boundary hypersurfaces are embedded. We denote the set of boundary hypersurfaces of $X$ by $\mathcal{M}_1(X)$. A collective boundary hypersurface refers to a finite union of non-intersecting boundary hypersurfaces.

\begin{definition}
\label{iterated_fibration_structure}
An \textit{\textbf{iterated fibration structure}} on a manifold with corners $X$ consists of a collection of fiber bundles
\begin{center}
    $Z_Y -\mathfrak{B}_Y \xrightarrow{\phi_Y} Y$
\end{center}
where $\mathfrak{B}_Y$ is a collective boundary hypersurface of $X$ with base and fiber manifolds with corners such that:

\begin{enumerate}
    \item Each boundary hypersurface of $X$ occurs in exactly one collective boundary hypersurface $\mathfrak{B}_Y$.
    \item If $\mathfrak{B}_Y$ and $\mathfrak{B}_{\widetilde{Y}}$ intersect, then dim $Y \neq$ dim 
    $\widetilde{Y}$, and we write $Y < \widetilde{Y}$ if dim $Y < $ dim $ \widetilde{Y}$.
    \item If $Y < \widetilde{Y}$, then $\widetilde{Y}$ has a collective boundary hypersurface $\mathfrak{B}_{Y\widetilde{Y}}$ participating in a fiber bundle $\phi_{Y\widetilde{Y}} : \mathfrak{B}_{Y\widetilde{Y}} \rightarrow Y$ such that the diagram 
    
\[\begin{tikzcd}
	{\mathfrak{B}_Y \cap \mathfrak{B}_{\widetilde{Y}}} && {\mathfrak{B}_{Y\widetilde{Y}} \subseteq \widetilde{Y}} \\
	& Y
	\arrow["{\phi_{\widetilde{Y}}}", from=1-1, to=1-3]
	\arrow["{\phi_Y}"', from=1-1, to=2-2]
	\arrow["{\phi_{Y\widetilde{Y}}}", from=1-3, to=2-2]
\end{tikzcd}\]
commutes.
\end{enumerate}
\end{definition}

The base can be assumed to be connected but the fibers are in general disconnected. As we mentioned above, there is an equivalence between Thom-Mather stratified spaces and manifolds with corners with iterated fibration structures. 

If we view a cone over a link $Z$ as the quotient space of $[0,1]_x \times Z$ where the points of the link at $\{x=0\} \times Z$ are identified, then the quotient map is a \textit{blow-down map}.
More generally, there is a desingularization procedure which replaces a Thom-Mather stratified space $\widehat{X}$ with a manifold with corners with an iterated fibration structure, $X$.
There is a corresponding {blow-down} map 
\[ \beta : X \rightarrow \widehat{X}, \] 
which satisfies properties given in Proposition 2.5 of \cite{Albin_signature}.
See Remark 3.3 of \cite{Albin_2017_index} for an instructive toy example.

Under this equivalence, the bases of the boundary fibrations correspond to the different strata, which we shall denote by
\begin{equation}
    \mathcal{S}(X) = \{Y: Y \textit{ is the base of a boundary fibration of X} \}.
\end{equation}
The bases and fibers of the boundary fiber bundles are manifolds with corners with iterated fibration structures (see for instance Lemma 3.4 of \cite{albin2010resolution}). The condition dim $Z_Y > 0$ for all $Y$ corresponds to the category of pseudomanifolds within the larger category of stratified spaces. The partial order on $\mathcal{S}(X)$ gives us a notion of depth
\begin{center}
$ \text{depth}_X(Y) = \max \{ n \in \mathbb{N}_0 : \exists Y_i \in \mathcal{S}(X)$ s.t. $Y=Y_0 <Y_1 < ... <Y_n \}$.
\end{center}
The depth of $X$ is then the maximum of the integers $\text{depth}_X(Y)$ over $Y \in \mathcal{S}(X)$. 

We now introduce some auxiliary structures associated to manifolds with corners with iterated fibration structures.
If $H$ is a boundary hypersurface of $X$, then because it is assumed to be embedded, there is a smooth non-negative function $\rho_H$ such that 
$\rho_{H}^{-1}(0)=H$ and $d\rho_H$  does not vanish at any point on $H$. We call any such function a \textbf{\textit{boundary defining function for $H$}}.
For each $Y \in \mathcal{S}(X)$, we denote a \textbf{\textit{collective boundary defining function}} by
\begin{center}
    $\rho_Y = \prod_{H \in \mathfrak{B}_Y} \rho_H$,
    and by
    $\rho_X = \prod_{H \in \mathcal{M}_1(X)} \rho_H$
\end{center}
 a \textbf{\textit{total boundary defining function}}, where $\mathcal{M}_1(X)$ denotes the set of boundary hypersurfaces of $X$.

When describing the natural analogues of objects in differential geometry on singular spaces, the iterated fibration structure comes into play. For example,
\begin{equation}
\label{equation_smooth_functions_on_stratified_spaces}
\mathcal{C}_{\Phi}^{\infty}(X)= \{ f \in \mathcal{C}^{\infty}(X) : f \big|_{\mathfrak{B}_Y} \in 
\phi_{Y}^*\mathcal{C}^{\infty}(Y) \text{ for all } Y \in \mathcal{S}(X) \}
\end{equation}
corresponds to the smooth functions on $X$ that are continuous on the underlying
stratified space $\widehat{X}$, compatible with the boundary fibration structure on $X$, and restricts to a smooth function on each $Y\in \mathcal{S}(X)$. 

\subsubsection{Wedge metrics and related structures}
\label{subsubsection_wedge_metrics_related_structures}

On Thom-Mather stratified pseudomanifolds we can define metrics which are \textit{locally conic}. For instance, consider the model space $\widehat{X}=\mathbb{R}^k \times C({Z})$, and the resolved manifold with corners that corresponds to the blowup of this space is $X=\mathbb{R}^k \times [0,1)_x \times Z$. By a model wedge metric we mean a metric of the form
\begin{equation}
\label{equation_definition_homogeneous_metric_cone}
    g_{w} = g_{\mathbb{R}^k} + dx^2 + x^2 g_Z,
\end{equation}
where $g_Z$ is a wedge metric on $Z$. This is a product metric on the product space $X^{reg}$ and which we call a  \textbf{\textit{rigid/ product type}} wedge metric, degenerating as it approaches the stratum at $x=0$. These metrics are degenerate as bundle metrics on the tangent bundle but we can introduce a rescaled bundle on which they are non-degenerate. 

Formally we proceed as follows. Let $X$ be a manifold with corners and iterated fibration structure. Consider the ‘wedge one-forms’
\begin{center}
    $\mathcal{V}^{*}_{w}= \{ \omega \in \mathcal{C}^{\infty}(X; T^*X) : \text{ for each } Y \in \mathcal{S}(X), \text{  }i^*_{\mathfrak{B}_Y} \omega (V) =0 \text{ for all } V \in \text{ker } D\phi_Y \}$.
\end{center}
Using the Serre-Swan theorem or proceeding as in \cite[\S 8.1]{melroseAPS}
we can identify $\mathcal{V}^{*}_{w}$ with the space of sections of a vector bundle $\prescript{{w}}{}{T}^*X$ which we call the wedge cotangent bundle of X, together with a map
\begin{equation}
    i_{w} : \prescript{{w}}{}{T}^*X \rightarrow T^*X
\end{equation}
that is an isomorphism over $X^{reg}$ such that,
\begin{center}
    $(i_{w})_*\mathcal{C}^{\infty}(X;\prescript{{w}}{}{T}^*X) =  \mathcal{V}^{*}_{w} \subseteq \mathcal{C}^{\infty}(X; T^*X)$.
\end{center}
In local coordinates near a collective boundary hypersurface, the wedge cotangent bundle is spanned by
\begin{center}
    $ dx$, $xdz$, $dy$
\end{center}
where $x$ is a boundary defining function for $\mathfrak{B}_Y$, $dz$ represents covectors along the fibers and $dy$ covectors along the base.
The dual bundle to the wedge cotangent bundle is known as the wedge tangent bundle, $\prescript{{w}}{}{T}X$. It is locally spanned by
\begin{center}
    $\partial_x$, $\frac{1}{x}\partial_z$, $\partial_y$
\end{center}
A \textit{\textbf{wedge metric}} is simply a bundle metric on the wedge tangent bundle.
Totally geodesic wedge metrics and exact wedge metrics are defined inductively. On depth zero spaces, which are just smooth manifolds, a wedge metric is a Riemannian metric. Assuming we have defined totally geodesic wedge metrics at spaces of depth less than $k$, let us assume $X$ has depth $k$.
A wedge metric $g_w$ on $X$ is a \textit{\textbf{totally geodesic wedge metric}} if, for every $Y \in \mathcal{S}(X)$ of depth $k$ there is a collar neighbourhood $\mathscr{C}(\mathfrak{B}_Y) \cong [0,1)_x \times \mathfrak{B}_Y$ of $\mathfrak{B}_Y$ in $X$, a metric $g_{w, pt}$ of the form
\begin{equation}
    g_{w, pt} = dx^2 +x^2 g_{\mathfrak{B}_Y/Y} +\phi^*_Yg_Y
\end{equation}
where $g_Y$ is a totally geodesic wedge metric on Y, $g_{\mathfrak{B}_Y/Y}+\phi^*_Yg_Y$ is a submersion metric for $\mathfrak{B}_Y \xrightarrow{\phi_Y} Y$ and $g_{\mathfrak{B}_Y/Y}$ restricts to each fiber of $\phi_Y$ to be a totally geodesic wedge metric on $Z_Y$ and
\begin{equation}
\label{equation_structure of the metric}
    g_{w} - g_{w, pt} \in x^2 \mathcal{C}^{\infty} \big(\mathscr{C}(\mathfrak{B}_Y); S^2(^{w}T^*X))\big).
\end{equation}
If at each step  $g_{w} = g_{w, pt}$, we say $g_w$ is a \textit{\textbf{rigid}} or \textit{\textbf{product-type wedge metric}}. If at every step, $g_{w} - g_{w, pt} = \mathcal{O}(x)$ as a symmetric two-tensor on the wedge tangent bundle, we say $g_w$ is an \textit{\textbf{exact wedge metric}}. We work with exact wedge metrics in this article.

The notion of a \textit{\textbf{wedge differential operator}} 
$P$ of order $k$ acting on sections of a vector bundle $E$, taking them to sections of a vector bundle $F$ is described on page 11 of \cite{Albin_2017_index}. We first define the \textit{\textbf{edge vector fields on X}} by
\begin{center}
    $\mathcal{V}_e = \{ V \in  \mathcal{C}^{\infty}(X;TX): V \big|_{\mathfrak{B}_Y}$ is tangent to the fibers of $\phi_Y$ for all $Y \in \mathcal{S}(X) \}$.
\end{center}
There is a rescaled vector bundle that is called the \textit{\textbf{edge tangent bundle}} $\prescript{e}{}TX$ together with a natural vector bundle map $i_e : \prescript{e}{}TX \rightarrow TX$ that is an isomorphism over the interior and satisfies
\begin{center}
    $(i_e)_* \mathcal{C}^{\infty}(X;\prescript{e}{}TX)=\mathcal{V}_e$.
\end{center}
In local coordinates near a point in $\mathfrak{B}_Y, (x, y, z)$,
a local frame for $\prescript{e}{}TX$ is given by 
\begin{center}
    $ x\partial_x, x\partial_y, \partial_z$
\end{center}
Note that the vector fields $x\partial_x$ and $x\partial_y$ are degenerate as sections of $TX$, but not as sections of $\prescript{e}{}TX$. 

The universal enveloping algebra of $\mathcal{V}_e$ is the ring $\text{Diff}^*_e(X)$ of edge differential operators. That is, these are the differential operators on $X$ that can be expressed locally as finite sums of products of elements of $\mathcal{V}_e$. They have the usual notion of degree and extension to sections of vector bundles, as well as an edge symbol map defined on the edge cotangent bundle (see \cite{Mazzeo_Edge_Elliptic_1,Albin_signature,Albin_hodge_theory_cheeger_spaces}).
Similarly, $\text{Diff}^*_e(X; E, F)$ denotes the edge differential operators acting on sections of a vector bundle $E$ and taking them to a sections on a vector bundle $F$.

We define \textit{\textbf{wedge differential operators}} by 
\begin{equation*}
    \text{Diff}^k_w(X;E,F) = \rho^{-k}_X \text{Diff}^k_e(X;E,F)
\end{equation*}
following, e.g., \cite{Albin_2017_index}, where $\rho_X$ is a total boundary defining function for $X$.

\subsubsection{Hodge-de Rham Dirac operator.}
\label{subsection_de_Rham}

Given a wedge metric, we can construct the corresponding Hodge de Rham Dirac operator similar to the smooth setting.
We follow \cite[\S 4.3]{Zhanglectures} and introduce the Clifford operators 
\begin{equation}
    cl(u)=u \wedge - \iota_{u^\#}, \quad \widehat{cl}(u)=u \wedge + \iota_{u^\#}
\end{equation}
where $u$ is a wedge one form, and $cl(u)$ is the Clifford multiplication for the de Rham complex. Zhang's notation is for the Clifford algebra on the tangent space, while we use that on the wedge cotangent bundle by duality, as in \cite{jayasinghe2023l2}. This action extends to forms with coefficients in a flat bundle over $X$ similar to the smooth setting (see \cite[\S 2]{albin2013novikov}).

Given wedge co-vector fields $u, v \in \Gamma(^{w} T^*X)$, we have the relation 
\begin{equation}
    \{cl(u), cl(v)\}=cl(u)cl(v)+cl(v)cl(u)=-2 g(u, v).
\end{equation}
Given a flat connection on a bundle $E$, there is an induced connection on forms valued in $E$, which we denote by $\nabla^E$. Taking the anti-symmetrization after acting on forms by $\nabla^E$ defines an operator $d_E$ and this satisfies $d_E^2=0$ (due to the flatness of the connection) similar to the smooth setting. 

Composing $\nabla^E$ with the Clifford action we obtain the Hodge de Rham Dirac operator $D$ acting on sections supported on $\mathring{X}$. Given a (local) orthonormal frame of the wedge cotangent bundle $\{e^1, e^2, \ldots, e^{n} \}$, we can write
\begin{equation}
    D=\sum_{i=1}^{n} cl(e^{i}) \nabla^E_{{e^{i}}^\#}.
\end{equation}
If for some $Y\in \mathcal{S}(X)$, we restrict to a collar neighborhood of $\mathfrak{B}_Y$ where we have an exact wedge metric, the Dirac operator can be written as
\begin{equation}
\label{temp_lah_di_dah}
cl(\partial_x)\nabla^E_{\partial_x} + {cl}\left(\frac{1}{x}\partial_{z_i}\right)\nabla^E_{\frac{1}{x}\partial_{z_i}} +{cl}(\partial_{y_j})\nabla^E_{\partial_{y_j}}
\end{equation}
up to a differential operator in $x\text{Diff}^1_w(X;F)$.   
Here $x$ is a boundary defining function for $\mathfrak{B}_Y$, and we recognize \eqref{temp_lah_di_dah} as a wedge differential operator of order one. It is well known that restricted to $X^{reg}$ we can write the operator as $D_E=d_E+d^*_E$ where
$d_E$ is the twisted de Rham operator and $d_E^*$ its formal adjoint.

\subsection{Stratified Morse-Bott functions}
\label{subsection_stratified_Morse_Bott_functions}

A Morse-Bott function on a smooth manifold is a function whose critical points consist of non-degenerate critical submanifolds, and we refer to \cite{AtiyahBottYangMillsRiemann83} for an exposition.
As opposed to the simpler case of isolated critical points, one motivation to study such critical point \emph{sets} is because they occur when the Morse-Bott function arises from a Hamiltonian group actions; in general such critical point sets can be different than those studied by Bott, and in \cite{Kirwancohomofquot84} what is now known as the \textit{Morse-Bott-Kirwan} condition was introduced. 

The study of Morse-Bott-Kirwan critical point sets on smooth spaces has been generalized in various ways on certain singular spaces such as singular projective varieties equipped with algebraic torus actions; the work of (\cite{kirwan1988intersection,feehan2022bialynicki}) study Morse-Bott functions and corresponding inequalities on singular spaces using different methods, including using Bialynicki-Birula decompositions of algebraic varieties corresponding to group actions.

A Morse-Bott function on a smooth manifold $M$ is a function $h$ such that the set where $dh=0$ is non-degenerate in the following manner, where we follow \cite[\S 1]{AtiyahBottYangMillsRiemann83}. If $N \subset M$ is a connected submanifold, it will be a \textit{non-degenerate critical point set of $h$} if and only if 
\begin{enumerate}
    \item $dh=0$ along $N$, 
    \item the Hessian of $h$ is non-degenerate on the normal bundle to $N$ in $M$.
\end{enumerate}
This is equivalent to the statement that there is a tubular neighbourhood of $N$ in $M$ which is a fibration over $N$, where the fibers are diffeomorphic to a product of discs $\mathbb{D}^{k_1}_x \times \mathbb{D}^{k_2}_r$ where $x, r$ are the radial distance functions on each disc factor, such that the Morse-Bott function can be written as $x^2-r^2$ with respect to some choice of smooth metric on the neighbourhood. It is the latter notion which we generalize to the stratified setting since the standard non-degeneracy condition on the Hessian is more opaque in this singular setting. We begin by first defining the notion of two types of \emph{stratified} non-degenerate critical point set, with respect to some wedge structure.

\begin{remark}
    \textbf{In the following three definitions}, we shall assume $\widehat{X}$ to be a stratified pseudomanifold with a wedge metric $g_w$ and a continuous function $h' \in C^0(\widehat{X})$ which lifts to a map $h \in C_{\Phi}^{\infty}(X)$ (see \eqref{equation_smooth_functions_on_stratified_spaces}) on the resolved manifold with corners, i.e., $h=h' \circ \beta$. We write $\text{crit}(h):=\beta(|dh|_{g_w}^{-1}\{0\})$.

    We note that despite the appearance of the metric $g_w$ in the following definitions, the condition of being critical, i.e. $\widehat{F_a}\subset \beta(|dh|_{g_w}^{-1}(0))$, is independent of the choice of wedge metric $g_w$. In fact we only need that the differential of $h$ vanishes along $\overline{\beta^{-1}\left(\widehat{F_a}\right)}$.
\end{remark}
    
\begin{definition}[\textbf{Smooth non-degenerate critical point sets}]
    If $F_a\subset \widehat{X}$ is a connected component of $\text{crit}(h)$ which is wholly a subset of $X^{reg}$, then we say it is a smooth non-degenerate critical point set if $F_a$ is a smooth submanifold of $X^{reg}$ and the Hessian of $h$ is non-degenerate on the normal bundle of $F_a$ in $X^{reg}$.
\end{definition}
This is the usual definition of a non-degenerate critical point set for a Morse-Bott function.

\begin{definition}[\textbf{Type I stratified critical point sets}]
    If $F_a\subset \widehat{X}$ is a connected component of $\text{crit}(h)$ which is disjoint from $\widehat{X}^{reg}$, then we say that it is a Type I stratified critical point set if $F_a$ is a smooth submanifold of $Y_\alpha\in \mathcal{S}(X)$ where $F_a \hookrightarrow Y_\alpha$ is a smooth embedding, and where the restriction of the function $h$ to $F_a$ is a smooth Morse-Bott function
    
    Since $\pi_\alpha: \mathcal{T}_{Y_\alpha}\rightarrow Y_\alpha$ is a locally trivial fibration with fiber $C(\widehat{Z_a})$, (the cone over a Thom-Mather stratified space $\widehat{Z_a}$), we conclude that $F_a$ has an open neighborhood $V({F_a})$ in $\widehat{X}$ which arises as a fiber bundle 
   \begin{equation}
       \pi_{F_a}: V({F_a}) \to F_a,
   \end{equation}
    with fibers $C(\mathbb{S}^{\ell})\times C(\widehat{Z_a})$, where $\ell+1>0$ is the codimension of the smooth embedding $F_a \hookrightarrow Y_\alpha$. On each fiber we define $\rho_N$ to be the distance from the set $\{0\} \times C(\widehat{Z_a})$, where ($\{0\}$ is the origin of the smooth disc $C(\mathbb{S}^{\ell})$) to a point on the fiber.
    We equip this critical point set with the \textbf{critical control data} $(\mathcal{T}_{F_a}, \pi_{F_a}, \rho_{F_a})$
    where $(\mathcal{T}_{F_a}, \pi_{F_a}):=(V(F_a), \pi_{F_a})$ and $\rho_{F_a}:=\rho_Y^2+\rho_N^2$ where $\rho_Y$ is the boundary defining function of $Y$ corresponding to the radial function given by the control data.
\end{definition}

Since $F_a$ is embedded in $Y_{\alpha}$, the normal bundle of the submanifold within $Y_{\alpha}$ has $\mathbb{R}^l$ fibers, of which we take a ball neighbourhood $C(\mathbb{S}^{\ell})$ and take the product of these fibers at each point with the fibers $C(\widehat{Z_a})$ of the fibration $\pi_\alpha$ to form the fibers of the new fibration $\pi_{F_a}$.

\begin{remark}
\label{Remark_clarify_Type_I}
We observe that by construction, the \textit{critical control data} admits a splitting corresponding to the product of fibers $C(\mathbb{S}^{\ell})\times C(\widehat{Z_a})$. Since the function restricted to $F_a$ is smooth Morse-Bott, the Morse-Bott lemma can be used to factor the disc factor into attracting and expanding disc factors. If in addition the factor $C(\widehat{Z_a})$ admits a similar splitting, overall the fiber can split into 4 factors.
For the purposes of analyzing the local Morse cohomology, it is easier to take products of the attracting and expanding factors (of the disc and cone factors) to form a two factor decomposition of each fiber, which can then be expressed as a whitney sum of fiber bundles. We will do this in Definition \ref{definition_non_degeneracy_critical_point_sets}.

We also note that the critical control data is not equal to the control data of the stratum $Y_{\alpha}$ appearing in the above definition.
\end{remark}

In the following we will assume that the components $Y_{\alpha}$ of the strata $\mathcal{S}(\widehat{X})$ are connected, as can be arranged by refining the stratification by simply separating the connected components.

\begin{definition}[\textbf{Type II stratified critical point sets}]
Now let $\widehat{F_a}\subset \widehat{X}$ be a connected component of $\text{crit}(h)$ which is disjoint from $\widehat{X}^{reg}$.
We say $\widehat{F_a}$ is a Type II stratified critical point set if $\widehat{F_a}$ is a Thom-Mather stratified space which is equal to some $\widehat{Y_{\alpha}}$, by which we denote the metric closure (with the wedge metric obtained by restriction) of a connected stratum $Y_{\alpha}$ of $\widehat{X}$.
 
We denote by $F_a^{reg}=\widehat{F_a}^{reg}:= \widehat{(F_a)}_m \setminus \widehat{(F_a)}_{m-1}$ the regular part of $\widehat{F_a}$.
In this case we define the \textbf{critical control data} $(\mathcal{T}_{F_a^{reg}}, \pi_{F_a}, \rho_{F_a})$ for $F_a^{reg}$ to be the restriction of the control data $(\mathcal{T}_Y,\pi_Y, \rho_Y)$ of $Y$. 
\end{definition}

In what follows, we will always define and study objects related to $F_a^{reg}$, which in the case of Type I stratified critical points corresponds to objects on $F_a$ since in that case $F_a$ is a smooth manifold (a pseudomanifold of depth $0$).

\begin{definition}[\textbf{Non-degeneracy of critical point sets}]
\label{definition_non_degeneracy_critical_point_sets}
If $F_a\subset \text{crit}(h)$ is a critical point set which is of Type I or Type II, we say it is non-degenerate if there is a certain compatibility between the function $h'$ and the critical control data that we defined in each case.

Namely, given the critical control data $(\mathcal{T}_{F_a^{reg}}, \pi_{F_a}, \rho_{F_a})$, we \textit{further} require that the locally trivial fibration 
\begin{equation}
\label{equation_fibration_critical_components}
    \pi_{F_a}:\mathcal{T}_{F_a^{reg}}\to F_a^{reg}
\end{equation}
decomposes as a Whitney sum of fiber bundles
\begin{equation*}
    \mathcal{T}_{F_a^{reg}} = E_u\oplus E_s,
\end{equation*}
where $E_{u/s}:=C(\widehat{Z_{a,u/s}})$, and with respect to this splitting each fiber
\begin{equation*}
\pi_{F_a}^{-1}(q) := C(\widehat{Z_{a,u}})\times C(\widehat{Z_{a,s}}) = C(\widehat{Z_{a,u}\star Z_{a,s}}),
\end{equation*}
is a product of cones 
\begin{equation*}
C(\widehat{Z_{\bullet}}) =  \widehat{Z_{\bullet}} \times [0,2) \Big/ \widehat{Z_{\bullet}}\times\{0\}
\end{equation*}
over other Thom-Mather stratified spaces $\widehat{Z_{a,u}}, \widehat{Z_{a,s}}$. Denoting the restriction of $\rho_{F_a}$ to each subbundle by 
\begin{equation*}
\rho_{F_a}\big|_{E_u} = \rho_{F_a,u}, \quad \rho_{F_a}\big|_{E_s}=\rho_{F_a,s},
\end{equation*}
we require that they satisfy
\begin{equation*}
\rho_{F_a}\big|_{\mathcal{T}_{F_a^{reg}}} = \sqrt{\rho_{F_a,u}^2 + \rho_{F_a,s}^2} ,
\end{equation*}
and with respect to these `radial functions' we demand that $h'$ satisfies

\begin{equation*}
\label{Morse_Bott_Kirwan_modify}
h'\big|_{\mathcal{T}_{F_a^{reg}}} =  f(F_a^{reg})+\rho_{F_a,s}^2 - \rho_{F_a,u}^2. 
\end{equation*}
\end{definition}

As discussed in Remark \ref{Remark_clarify_Type_I}, and following the definition given above, we observe that the factors $E_{u/s}$ of the fibers given above are products of attracting and expanding factors.

\begin{definition}[Stratified Morse-Bott functions]
    \label{Definition_stratified_Morse_Bott_function}
Let $\widehat{X}$ be a stratified pseudomanifold with a wedge metric $g_w$ and a continuous function $h' \in C^0(\widehat{X})$ which lifts to a map $h \in C_{\Phi}^{\infty}(X)$ (see \eqref{equation_smooth_functions_on_stratified_spaces}) on the resolved manifold with corners, i.e., $h=h' \circ \beta$ and which is a Morse-Bott function when restricted to $\widehat{X}^{reg}$. Finally, away from the regular part $\widehat{X}^{reg}$ we demand that the image crit$(h)$ of the set $|dh|_{g_w}^{-1}(0)$ under the blow-down map $\beta$ is a union of non-degenerate critical point sets which are all smooth, or Type I, or Type II stratified critical point sets. Then we say that the function $h'$ is a \textbf{stratified Morse-Bott function} (we use the same terminology for $h$ by abuse of notation, distinguishing them only where it is technically important).
\end{definition}

\begin{definition}[singular Fundamental neighborhood]
Let $\widehat{F_a}$ be a stratified non-degenerate critical point set of a stratified Morse-Bott function $h$ on a stratified pseudomanifold $\widehat{X}$ as above. On the tubular neighborhood $\mathcal{T}_{F_a^{reg}} = \{ \zeta\in \widehat{X} : \rho_{F_a}(\zeta)< 4 \}$ 
we can define the closed subset 
\[ \left\{ \zeta\in \mathcal{T}_{F_a^{reg}}: \max\{\rho_{F_a,u}(\zeta),\rho_{F_a,s}(\zeta)\} \leq 1 \right\} , \]
which is a fiber bundle whose fibers $\widehat{U_a}$ are a product of truncated cones 
\[ \widehat{U_{a}} :=  \widehat{U_{a,s}}\times \widehat{U_{a,u}}  \]
where each factor $\widehat{U_{a,s/u}}$ is homeomorphic to the one point compactification at $\rho_{a,s/u}=0$ of  $\widehat{Z_{a,s/u}}\times [0,1]_{\rho_{a,s/u}}$, where $\rho_{a,s/u}$ is the restriction of 
$\rho_{{F_a},s/u}$ to each fiber.
Taking the metric completion in $\widehat{X}$ of this fiber bundle over $F_a^{reg}$, we denote this completion by $U(\widehat{F_a})$ and refer to it as the \textbf{\textit{singular fundamental neighbourhood of }$\widehat{F_a}$}.
\end{definition}

\begin{remark}
\label{Remark_not_a_neighbourhood}
We emphasize that we do not demand any fiber bundle structure on $U(\widehat{F_a})$, since any such notion would carry with it additional complications since the candidate base is stratified. We follow the philosophy that for the purposes of studying $L^p$ cohomology on spaces, one only has to understand the space of sections on a dense set, and find domains for the operators by taking appropriate (ideal) boundary data. In studying local invariants corresponding to global domains, we will consider restrictions of sections in the global domains to the dense sets and define local domains using appropriate graph closures. This is done on a set which corresponds to a resolution of the above singular fundamental neighbourhoods that we define next.

We emphasize here that the nomenclature of (resolved) fundamental neighbourhood is an abuse of language since the set we define is not strictly a neighbourhood of the fixed point set $F_a$ in the case of stratified critical point sets of Type II.
\end{remark}

\begin{definition}[resolved fundamental neighbourhood]
\label{definition_resolved_fundamental_neighbourhood}
    Given a critical point set $\widehat{F_a}$ with a singular fundamental neighbourhood $U(\widehat{F_a})$, we define the \textbf{(resolved) fundamental neighbourhood} $U(F_a)$ as the fiber bundle 
\begin{equation}
\label{equation_fibration_critical_components_resolved}
    {U_a} - {U(F_a)} \xrightarrow{{\phi_{F_a}}} F_a^{reg}
\end{equation}    
    where the fibers $U_a$ are the products of resolutions  of the fibers $C(\widehat{Z_{a,u}}), C(\widehat{Z_{a,s}})$ i.e. 
\begin{equation}
    U_a =U_{a,s}\times U_{a,u} := \left( [0,1]_x\times Z_{a,s} \right) \times  \left([0,1]_r\times Z_{a,u}\right),
\end{equation}   
    where $r_a=\rho_{a,u}$ and $x_a=\rho_{a,s}$ are new variables on the resolution. We equip $U(F_a)$ with the restriction of the wedge metric $g_w$.
    We will refer to these resolved fundamental neighbourhoods as \textbf{fundamental neighbourhoods} when it is clear by context and notation.
\end{definition}

\begin{remark}
    We note that the neighbourhood $U(\widehat{F_a})$ arises as the metric completion in $\widehat{X}$ of a fiber bundle, and thus arises as a subset of $\widehat{X}$. On the other hand, while $U(F_a)$ is still a bundle over $\widehat{F_a^{reg}}$, we do not consider this resolution as a subset of $X$ (the resolution of $\widehat{X}$), and indeed it may not embed into $X$ in a natural way. For our later analysis, it will suffice to consider the space $U(F_a)$ separately from $X$, where the association is via the spaces of sections defined on each space.
\end{remark}

Finally we study the \textit{metric} boundary of the resolved fundamental neighborhood $U(F_a)$. We refer to \cite[\S 2.1]{Jesus2018Wittensgeneral} for a discussion of the product of cone metrics being a cone metric (c.f., \cite[\S 3.2.1]{Jesus2018Wittens})
where a refinement of the original stratification of $\widehat{U_a}$ and a refinement of the corresponding manifold with fibered boundary structure to the resolved space $U_p$ is presented for products of cones.
For the general fundamental neighbourhoods of critical points $F_a$ that we consider, we have the following structure.

\begin{definition}
\label{definition_smooth_boundary_and_singularities}
Given a fundamental neighbourhood $U({F_a})$, we define the fibers of the metric boundary to be 
to be 
\begin{equation}
    \partial U_a = [\{x=1\} \times {Z_{a,s}} \times {C_r(Z_{a,u})}] \cup [{C_x({Z_{a,s}})} \times \{r=1\} \times Z_{a,u} ] 
\end{equation}
which are the fibers of the fibration 
\begin{equation}
    { \partial U_a} - {\partial U(F_a)} \xrightarrow{{\partial \phi_{F_a}}} F_a^{reg}
\end{equation}
where this is the fibration obtained by restricting the fibration ${\phi_{F_a}}$ in \eqref{equation_fibration_critical_components_resolved} to the metric boundaries of the fibers.
\end{definition}

It is clear that under the blow-down map, the image of the metric boundary are the boundary components of co-dimension 1 (as opposed to the singular components of higher codimension) of the fundamental neighbourhood. We will see that this is where we need to apply boundary conditions (supplementing the ideal boundary conditions chosen for the sections on the entire space $X$) to fix a domain for the local Hilbert complexes.

Finally we add a simplifying hypothesis for our later analysis.
\begin{definition}[\textbf{Flatness assumption}]
\label{Definition_flatness_assumption}
     Given a stratified Morse-Bott function $h$ on $\widehat{X}$, we require every connected component of $\widehat{F_a}\subset \text{crit}(h) = \beta(|dh|_{g_w}^{-1}(0))$ satisfies that the `tubular neighborhood' $\mathcal{T}_{F_a^{reg}}$ of $F_a^{reg}$, is a flat bundle i.e. admits a flat connection that extends to a flat wedge connection on $F_a$. Notice that this is a non-trivial hypothesis even in the case of smooth (i.e. non-stratified) $X$.
\end{definition}

Flatness conditions of this sort have been used previously in the literature to simply techniques as well as explicit formulas, for instance in \cite{BanaglFlatfiber2013}, and we refer the reader to that article for interesting examples of spaces which also admit actions of Lie groups, including certain flag manifolds.
The following proposition shows that there are model metrics near critical point sets of \textit{flat} stratified Morse-Bott functions.

\begin{proposition}
\label{Proposition_model_metric_existence}
    Let $X$ be a resolved stratified pseudomanifold with a wedge metric $g$ on $X^{reg}$ and a flat stratified Morse-Bott function $h$. Then there exists another wedge metric $g^\prime$ which on the fibration over 
    connected components $F_a^{reg}$ of the critical point set is of the form 
    \begin{equation}
    g^\prime \big|_{U(F_a)} = \phi_{F_a}^*g \big|_{F_a} + g \big|_{U_a}
\end{equation}
where $U_a$ is the fiber over the regular part of $F_a$.
\end{proposition}

\begin{proof}
We will modify the metric $g$ near the critical point sets to obtain a new metric. It suffices to describe the modification near a single connected component $F_a$ of the critical point set.

It is easy to find a cutoff function $\chi_a \in C^{\infty}_{\Phi}(X)$ which is identically $0$ outside a fundamental neighbourhood of $F_a$ and is identically $1$ inside a smaller fundamental neighbourhood. If we can find a model metric $g_m$ on the support of $\chi$, then the new metric is given by
\begin{equation}
    g^\prime = \chi g_m + (1-\chi)g 
\end{equation}
    
We can construct such a metric $g_m$ as follows. Given the metric $g$, we can restrict it to the base of the fibration over the regular part $F_a^{reg}$ and freeze coefficients along the normal fibers and extend it to a wedge metric on the fundamental neighbourhood by homogenously extending on the fibers to obtain a local metric $g_1$.
Observe that both the original metric $g$ and $g_1$ coincide on $F_a^{reg}$, and thus the latter admits an extension to $F_a$. We can now modify the metric $g_1$ restricted to $F_a^{reg}$ so that the metric on the links $Z_1$ and $Z_2$ corresponding to the attracting and expanding conic factors are isometric over each link. 
If the metrics on each fiber over $F_a^{reg}$ are not isometric, we pick the metric over one fiber and parallel transport it along $F_a^{reg}$ using the flat connection (see Definition \ref{Definition_flatness_assumption}).
This yields the \textit{model} metric
\begin{equation}
    g_2=g\big|_{F_a^{reg}} + g\big|_{U_a}
\end{equation}
where $g\big|_{F_a^{reg}}$ is the metric restricted to tangent vectors of $F_a^{reg}$, and extends to a wedge metric at the resolved singularities of $F_a$, while the metrics on the fibers for $g,g_2$ remain non-degenerate wedge metrics as well. Since the fiber metrics of $g_2$ are all isometric, we 
have the desired model metric at $F_a$. Similarly one can do the local modifications at every connected component of the critical point set to obtain the desired metric on $X$.
\end{proof}

\section{Hilbert complexes and cohomology}
\label{section_Hilbert_complexes_and_cohomology}

In this section, we discuss twisted de Rham complexes on stratified pseudomanifolds and their restrictions to fundamental neighbourhoods of critical points. First, we review some facts about abstract Hilbert complexes. Then, we review twisted de Rham complexes on stratified pseudomanifolds before describing local complexes on fundamental neighbourhoods and the corresponding cohomology groups. We then briefly study Poincar\'e dual complexes before reviewing some results on polynomial Lefschetz supertraces that are key to proving the Morse and Lefschetz-Morse inequalities.

\subsection{Abstract Hilbert complexes}
\label{subsection_abstract_hilbert_complexes}
We define Hilbert complexes following \cite{bru1992hilbert}. 

\begin{definition}
A \textbf{\textit{Hilbert complex}} is a complex $\mathcal{P}=(H_*,\mathcal{D}(P_*),P_*)$ of the form:
\begin{equation}
    0 \rightarrow \mathcal{D}(P_0) \xrightarrow{P_0} \mathcal{D}(P_1) \xrightarrow{P_1} \mathcal{D}(P_2) \xrightarrow{P_2} ... \xrightarrow{P_{n-1}} \mathcal{D}(P_n) \rightarrow 0,
\end{equation}
where each map $P_k$ is a closed 
operator which is called the differential, such that:
\begin{itemize}
    \item the domain $\mathcal{D}(P_k)$ of $P_k$ is dense in the separable Hilbert space $H_k$,
    \item the range of $P_k$ satisfies $\operatorname{ran}(P_k) \subset \mathcal{D}(P_{k+1})$,
    \item $P_{k+1} \circ P_k = 0$ for all $k$.
\end{itemize}
\end{definition}

We will often denote such a complex by $\mathcal{P}=(H,\mathcal{D}(P),P)$ without explicitly denoting the grading. 
We shall sometimes denote the complex as $\mathcal{P}(X)$ when the Hilbert spaces are sections of a vector bundle on the resolution of a stratified pseudomanifold $\widehat{X}$, and we say that the \textit{\textbf{Hilbert complex $\mathcal{P}(X)$ is associated to the space $X$}}.
The \textit{\textbf{cohomology groups}} of a Hilbert complex are defined to be $\mathcal{H}^k(\mathcal{P}):= \ker(P_k)/\operatorname{ran}(P_{k-1})$. We shall often use the notation $\mathcal{H}^k$, where the complex used is clear from the context and $\mathcal{H}^k(\mathcal{P}(X))$ when the space needs to be specified (including spaces with boundary when they come up later on). If these groups are finite dimensional in each degree, we say that it is a \textit{\textbf{Fredholm complex}}.

For every Hilbert complex $\mathcal{P}$ there is an \textit{\textbf{adjoint Hilbert complex}} $\mathcal{P}^*$, given by
\begin{equation}
\label{adjoint_complex}
    0 \rightarrow \mathcal{D}((P_{n-1})^*) \xrightarrow{(P_{n-1})^*} \mathcal{D}((P_{n-2})^*) \xrightarrow{(P_{n-2})^*} \mathcal{D}((P_{n-3})^*) \xrightarrow{(P_{n-3})^*} ... \xrightarrow{(P_{1})^*} H_0 \rightarrow 0
\end{equation}
where the differentials are $P_k^*: \text{Dom}(P^*_k) \subset H_{k+1} \rightarrow H_k$, the Hilbert space adjoints of the differentials of $\mathcal{P}$. That is, the Hilbert space in degree $k$ of the adjoint complex $\mathcal{P}^*$ is the Hilbert space in degree $n-k$ of the complex $\mathcal{P}$, and the operator in degree $k$ of $\mathcal{P}^*$ is the adjoint of the operator in degree $(n-1-k)$ of $\mathcal{P}$. The corresponding cohomology groups of $\mathcal{P}^*(H,P^*)$ are $\mathcal{H}^k(H, (P)^*) := \ker(P^*_{n-k-1})/\operatorname{ran}(P^*_{n-k})$.
For the case of the de Rham complex $(L^2\Omega^k(X),d_{\max})$, the adjoint complex $\mathcal{P}^*$ is the complex $\mathcal{Q}=(L^2\Omega^{n-k}(X),\delta_{\min})$, since the operators $d$ and $\delta$ are formal adjoints of each other. 

\begin{remark}
    In \cite{jayasinghe2023l2}, we referred to the adjoint Hilbert complexes introduced in this article as dual Hilbert complexes. In light of differences between certain Hilbert complexes associated with \textit{dualizing operators} such as the Hodge star operator for certain choices of domains on wedge and other singular metrics, we will distinguish between \textit{dual complexes} and \textit{adjoint complexes}.
\end{remark}

\subsubsection{Dirac complexes, Laplace type operators and domains}

We can form a two-step complex where the Hilbert spaces are $H^+ =\bigoplus_{q=even} H_q$, and $H^- = \bigoplus_{q=odd} H_q$.
This leads to a \textit{\textbf{wedge Dirac complex}} as introduced in Definition 3.3 \cite{jayasinghe2023l2}, 
\begin{equation}
    0 \rightarrow \mathcal{D}(D^{+}) \xrightarrow{D^+} \mathcal{D}(D^{-}) \rightarrow 0
\end{equation}
where $D^{\pm}$ is the spin$^\mathbb{C}$-Dirac operator restricted to the spaces, together with the domain for the operator $D= \sqrt{2} (P+P^*)$ given by
\begin{equation}
    \label{Domain_Dirac_first}
    \mathcal{D}(D)=\mathcal{D}(P) \cap \mathcal{D}(P^*).
\end{equation}
There is an associated \textbf{\textit{Laplace-type operator}} $\Delta_k = P_{k}^*P_k+P_{k-1}P_{k-1}^*$ in each degree, which is a self adjoint operator with domain
\begin{equation}
\label{Laplacian_P_type}
\mathcal{D}(\Delta_k) = \{ v \in \mathcal{D}(P_k) \cap \mathcal{D}(P_{k-1}^*) : P_k v \in \mathcal{D}(P_k^*), P^*_{k-1} v \in \mathcal{D}(P_{k-1}) \},
\end{equation}
and with nullspace
\begin{equation}
   \widehat{\mathcal{H}}^k(\mathcal{P}):= \ker(\Delta_k) = \ker(P_k) \cap \ker(P_{k-1}^*).
\end{equation}
The Kodaira decomposition which we present below in Proposition \ref{Kodaira_decomposition} identifies this with the cohomology of the complex $\mathcal{H}^k(\mathcal{P})$.
We observe that this Laplace-type operator can be written as the square of the associated Dirac-type operator $D=(P+P^*)$, restricted to each degree to obtain $\Delta_k$, and that the domain can be written equivalently as
\begin{equation}
\label{Laplacian_D_type}
\mathcal{D}(\Delta_k) = \{ v \in \mathcal{D}(D) : D v \in \mathcal{D}(D) \}.
\end{equation}
The null space is isomorphic to the cohomology for Fredholm complexes.

\begin{proposition}
\label{Kodaira_decomposition}
For any Hilbert complex $\mathcal{P}=(H, P)$ we have the \textbf{\textit{weak Kodaira decomposition}}
\[H_k = \widehat{\mathcal{H}}^k(H,P) \oplus \overline{\operatorname{ran}(P_{k-1})} \oplus \overline{\operatorname{ran}(P_k^*)}.\]
\end{proposition}
This is Lemma 2.1 of \cite{bru1992hilbert}. 

\begin{proposition}
\label{Fredholm_is_closed_range}
If the cohomology of a Hilbert complex $\mathcal{P}=(H_*, P_*)$ is finite dimensional then, for all $k$, $\operatorname{ran}(P_{k-1})$ is closed and therefore $\mathcal{H}^k(\mathcal{P}) \cong \widehat{\mathcal{H}}^k(\mathcal{P})$. 
\end{proposition}
This is corollary 2.5 of \cite{bru1992hilbert}. The next result justifies the use of the term \textit{Fredholm complex}.
\begin{proposition}
A Hilbert complex $(H_k, P_k)$, $k=0,...,n$ is a Fredholm complex if and only if, for each $k$, the Laplace-type operator $\Delta_k$ with the domain defined in \eqref{Laplacian_P_type} is a Fredholm operator.
\end{proposition}
This is Lemma 1 on page 203 of \cite{schulze1986elliptic}. 
Due to these results, we can identify the space of \textit{harmonic elements}, or the elements of the Hilbert space which are in the null space of the Laplace-type operator, with the cohomology of the complex in the corresponding degree. We shall use the same terminology for non-Fredholm complexes which we study as well. 

For Fredholm complexes, the null space of the Laplacian is isomorphic to the cohomology of the complex since the operator has closed range.

\begin{proposition}
\label{Kernel_equals_cohomology}
A Hilbert complex $\mathcal{P}=(H, P)$, is a Fredholm complex if and only if its adjoint complex, $(\mathcal{P}^*)$ is Fredholm.
If it is Fredholm, then 
\begin{equation}
    \mathcal{H}^k(\mathcal{P}) \cong \mathcal{H}^{n-k}(\mathcal{P}^*).
\end{equation}
\end{proposition}

In particular, for operators with closed range, the reduced cohomology groups are the same as the cohomology groups and are isomorphic to the null space of the Laplace-type operator, in which case the decomposition in Proposition \ref{Kodaira_decomposition} is called the \textbf{\textit{(strong) Kodaira decomposition}}.

\subsection{Twisted de Rham complexes on stratified pseudomanifolds}
\label{subsection_Hilbert_complexes_stratified_pseudomanifolds}

Dirac operators and the Hilbert complexes (in particular the de Rham complex) are not necessarily essentially self-adjoint on singular spaces. All possible domains for the de Rham complex on stratified spaces with wedge metrics are studied in \cite{Albin_hodge_theory_cheeger_spaces}. In this subsection we review these choices of domains, corresponding to ideal boundary conditions which are also called mezzo-perversities. We refer the reader to \cite{BanaglnonWitt2002,Albin_hodge_theory_on_stratified_spaces,albin2017novikov,RefinedintersectionAlbinBanagl2015} for more details.

Consider a wedge differential operator $P_X \in \operatorname{Diff}_w^1(X;E,F)$ acting between sections of two bundles $E$ and $F$.
For example, the de Rham operator $d$ and the Dirac operator $D=d+\delta$ are of wedge type. 

There are two canonical domains which are the \textbf{\textit{minimal domain}},
\begin{equation}
    \mathcal{D}_{\min}(P_X)= \{ u \in L^2(X;E) : \exists (u_n) \subseteq {C}^{\infty}_c(\mathring{X};E) \text{ s.t. }
    u_n \rightarrow u \text{ and } (P_Xu_n) \text{ is } L^2\text{-Cauchy} \},
\end{equation}
and the \textbf{\textit{maximal domain}},
\begin{equation*}
    \mathcal{D}_{\max}(P_X)= \{ u \in L^2(X;E) : (P_Xu) \in L^2(X;F) \},
\end{equation*}
wherein $P_Xu$ is computed distributionally.
Any other closed extension of $P_X$ corresponds to a choice of domain $\mathcal{D}_W(P_X)$ containing the minimal domain and contained in the maximal domain. This determines a choice of domain for the adjoint $P_X^*$ and a self-adjoint domain for $D=P_X+P^*_X$.

\begin{remark}
If $X$ has a wedge metric such that the Dirac operator satisfies the Witt condition, then we can find a wedge metric so that there is a domain satisfying the geometric Witt condition by rescaling the fibers on the links (see \cite{Albin_signature}). Even in the non-Witt case, a similar rescaling can be used to reduce the choices of self-adjoint extensions to cohomological data on the links that we study below, and this is used in \cite{Albin_hodge_theory_on_stratified_spaces,Albin_hodge_theory_cheeger_spaces}.
We follow the same convention.

Taking a metric without rescaling gives operators on domains with more sections, however there are no extra cohomology classes appearing in any of these domains.
\end{remark}

For wedge elliptic operators such as the Hodge de Rham Dirac operator $D_X$,
these domains satisfy the inclusions
\begin{equation}\label{equation_Dmin_Dmax_inclusion}
    \rho_XH^1_e(X;E) \subseteq \mathcal{D}_{\min}(D_X) \subseteq \mathcal{D}_{\max}(D_X) \subseteq H^1_e(X;E),
\end{equation}
where $H^1_e(X;E) = \{ u \in L^2(X;E) : Vu \in L^2(X;E) \text{ for all } V \in  \mathcal{C}^{\infty}(X;\prescript{e}{}TX) \}$ is the edge Sobolev space introduced in \cite{Mazzeo_Edge_Elliptic_1}. 
We now focus on the non-Witt case and discuss all possible \textit{topological} self adjoint extensions, between the maximal and minimal domains introduced above following \cite{Albin_hodge_theory_cheeger_spaces}.

\subsubsection{Case of Depth 1}
We begin by defining mezzo-perversities on a depth-1 stratum $Y$ and then use induction to generalize the definition to stratified spaces of arbitrary depth.
The inclusion \eqref{equation_Dmin_Dmax_inclusion} suggests that an element $u \in \mathcal{D}_{\max}(D)$ has a partial asymptotic expansion (as established in \cite{Albin_hodge_theory_cheeger_spaces}):
\begin{equation}
    u \sim \alpha(u) +dx \wedge \beta(u) + \tilde{u}, \quad \tilde{u} \in x^{1^-}H^{-1}_e(X;E),
\end{equation}
where $x$ is a boundary defining function for the boundary hypersurface that resolves $Y$ in $X$, and $\alpha(u), \beta(u)$ are the orthogonal projections of the forms $u$ and $\iota_{\partial_x}u$, respectively, onto the space
\begin{equation}
    H^{-1/2}(Y; \Lambda^\bullet(T^*Y) \otimes \mathcal{H}^{l/2} (\partial X/Y)),
\end{equation}
where $l$ is the dimension of the link $Z$ at $Y$, and where $x^{1-}H^{-1}_e(X;E) =  \bigcap_{\epsilon > 0} x^{1-\epsilon}H^{-1}_e(X;E)$ consists of elements in the minimal domain of $D$, as is clear by  \eqref{equation_Dmin_Dmax_inclusion}.
This asymptotic expansion allows us to define a \textbf{\textit{Cauchy data map}} on the resolved manifold with fibered boundaries
\begin{equation}
    \mathcal{C}_Y(D)(u) = \big(\alpha(u),\beta(u)\big), \text{ for }u \in \mathcal{D}_{\max}(D), 
\end{equation}
which can be identified with the projection
\begin{equation}
   \mathcal{D}_{\max}(D) \longrightarrow \mathcal{D}_{\max}(D)\big/\mathcal{D}_{\min}(D).
\end{equation}
An extension of $P_X=d_X$ is given by a choice of a closed domain $\mathcal{D}_{W}(P_X)$ such that 
\begin{equation}
    \mathcal{D}_{\min}(P_X) \subseteq \mathcal{D}_{W}(P_X) \subseteq \mathcal{D}_{\max}(P_X),
\end{equation}
where the subscript $W$ denotes a choice of flat sub-bundle of $\mathcal{H}^{l/2} (\partial X/Y)$. Equivalently this can be realized as 
\begin{align*}
    \mathcal{D}_{W}(P_X) &= \{ u \in \mathcal{D}_{\max}(P_X): (W^{\perp},\{0\}) \circ \mathcal{C}_Y(u) = 0\},\\
    &=\{u \in \mathcal{D}_{\max}(P_X): \alpha(u) \in H^{-1/2}(Y; \Lambda^\bullet (T^*Y) \otimes W)\}
\end{align*}
where by abuse of notation $W$ is an closed linear operator acting on the Cauchy data $\mathcal{C}_Y(u)$, which is given by
\[\mathcal{C}_{Y}(D)(u) = (\alpha(u),\beta(u)), u \in \mathcal{D}_{\max,W}(D_X), \](See \cite[p.11]{Albin_hodge_theory_cheeger_spaces}).

Recall that for an operator $P \in \text{Diff}^1(U;E,F)$ on a manifold with boundary $(U, \partial U)$ we have the Green-Stokes formula (see, e.g., Proposition 9.1 of \cite{taylor1996partial})
\begin{equation}
\label{equation_Greens_identity}
    \langle Ps, \tau \rangle_F - \langle s, P^* \tau \rangle_E = \int_{\partial U} g_F( i\sigma_1(P)(dr) s, \tau) \operatorname{dVol}_{\partial U}
\end{equation}
where $\sigma_1(P)(dr)$ is the principal symbol of the operator $P$ evaluated at the differential of a boundary defining function $r$, where we have chosen $r$ such that the outward pointing unit normal vector to the boundary is given by $\partial_r$, and the corresponding covector is $dr$.
We can use this to understand the adjoint domains through the boundary pairing at the resolved boundary $\partial X$, which fibers over the stratum $Y$ with links $(\partial X/ Y)$.
Integration by parts shows us
\begin{align}
    G_{d}: \mathcal{D}_{W}(d) \times \mathcal{D}_{W'}(\delta) &\longrightarrow \mathbb{R}\\
    (u,v) &\longmapsto \langle d s, \tau \rangle_{L^2} - \langle s, \delta \tau \rangle_{L^2}=\int_{\partial X} (\iota_{\partial_x}(dx) s, \tau) \operatorname{dVol}_{\partial X},
\end{align}
and the adjoint domain $\mathcal{D}_{W^{\perp}}(\delta)$ 
is the largest domain for which the boundary pairing vanishes for all elements in $\mathcal{D}_{W}(d)$. 
Thus, we study the boundary pairing for the operator $D=d+\delta$ given by
\begin{align}
    G_{D}: \mathcal{D}_{\max}(D) \times \mathcal{D}_{\max}(D) &\longrightarrow \mathbb{R}\\
    (u,v) &\longmapsto \langle D u, v \rangle_{L^2} - \langle u, D v \rangle_{L^2}.
\end{align}
This can be written using the Green-Stokes theorem as 
\begin{equation}
    G_{D}(u,v) = \langle \alpha(u), \beta(v) \rangle_{\partial X} - \langle \beta(u), \alpha(v) \rangle_{\partial X},
\end{equation}
where $\alpha,\beta$ belong to the middle degree cohomology $\mathcal{H}^{l/2}(\partial X)$ and are given by the Cauchy data $u_{l/2} = \alpha(u) + dx \wedge \beta(u)$. If we set 
\begin{equation}
    \mathcal{D}_W(d) = \{u \in \mathcal{D}_{\max}(d): \alpha(u) \in W\},
\end{equation}
then the adjoint domain is given by picking the \textbf{complement of $W$ with respect to the boundary pairing} and define 
\begin{equation}
    \mathcal{D}_{W^{\perp}}(\delta) = \{u \in \mathcal{D}_{\max}(\delta): \beta(u) \in W^{\perp} \}.
\end{equation}
Indeed if $\alpha(u) \in W$ and $\beta(u) \in W^{\perp}$, it is easy to check that $\langle \alpha(u), \beta(v) \rangle_{\partial X}=0$.
Then the domain for the Dirac operator is given by 
\begin{equation}
    \mathcal{D}_W(D=d+\delta) = \{u \in \mathcal{D}_{\max}(D): \alpha(u) \in W, \beta(u) \in W^{\perp} \}
\end{equation}
and it is shown in \cite{Albin_hodge_theory_cheeger_spaces,Albin_hodge_theory_on_stratified_spaces} that this corresponds to the domain for the Dirac operator defined in Subsection \ref{subsection_abstract_hilbert_complexes}.

In the depth one case, a \textbf{\textit{mezzo-perversity}} is a choice of sub-bundle $W$
\[
\begin{tikzcd}
    W \arrow[rr] \arrow[rd] & &\mathcal{H}^{l/2}(\partial X/Y)\arrow[dl] \\
    & Y &
\end{tikzcd}
\]
where $(\partial X/ Y)$ is the boundary fibration and $\mathcal{H}^{l/2}(\partial X/Y)$ is the fibration over $Y$ with typical fiber $\mathcal{H}^{l/2}(Z)$, where $W$ is parallel with respect to a particular flat connection (which is independent of the metric). We refer to \cite[\S 3.1]{albin2017novikov} for details of this Gauss-Manin type flat connection on the vertical cohomology. 

This determines the choice of adjoint domain $\mathcal{D}_{W^{\perp}}(\delta=P_X^*)$, realized as 
\begin{equation}
    \mathcal{D}_{W^{\perp}}(P^*_X) = \{ u \in \mathcal{D}_{\max}(P^*_X): (\{0\},W) \circ \mathcal{C}_Y(u) = 0\},
\end{equation}

\begin{remark}
In \cite{Albin_hodge_theory_cheeger_spaces,Albin_hodge_theory_on_stratified_spaces} the domains for the Signature complex are defined in terms of Cheeger ideal boundary conditions as follows.
Let $\pi_W, \pi_{{W}^\perp} \in C^\infty \big(Y^1; \operatorname{End}(\mathcal{H}^{l_1/2}(\partial X/Y^1)) \big)$ be the orthogonal projection onto $W$ and $W^\perp$ respectively. Then a \textit{\textbf{Cheeger ideal boundary condition}} $W$ associated to the Dirac operator $D$ is given by
\begin{equation}
    V = (\pi_W, \pi_{W^\perp} ) \in C^\infty \big(Y;\operatorname{End}(\mathcal{H}^{l/2}(\partial X/Y) \oplus \mathcal{H}^{l/2}(\partial X/Y))\big),
\end{equation}
which corresponds to the domain
\begin{equation}
\label{equation_different_domain_Dirac_notation}
    \mathcal{D}_{V}(D_X) = \{ u \in \mathcal{D}_{\max}(D_X): V^{\perp} \circ \mathcal{C}_Y(u) = 0\}
\end{equation}
\end{remark}

\begin{remark}[Convention]
\textbf{Unless specified otherwise}, we use the convention that we denote the domain for the Dirac operator corresponding to the mezzo-perversity $W$ as \[\mathcal{D}_W(D):=\mathcal{D}_W(d) \cap \mathcal{D}_{W^{\perp}}(d^*),\] and similarly for $\Delta=D^2$ (this is as opposed to notations such as $\mathcal{D}_{W^{\perp}}(D)$ prioritizing the choice of domain for $\delta$).  
\end{remark}

We discuss an example of a suspension of a torus (see Figure \ref{fig_1_sus_torus}) here, which we continue to build on in Examples \ref{Example_suspension_torus_2} and \ref{Example_suspension_torus_Witten_deformation} with different choices of mezzo-perversities.

\begin{figure}[h]
    \centering
    \includegraphics[scale=.4]{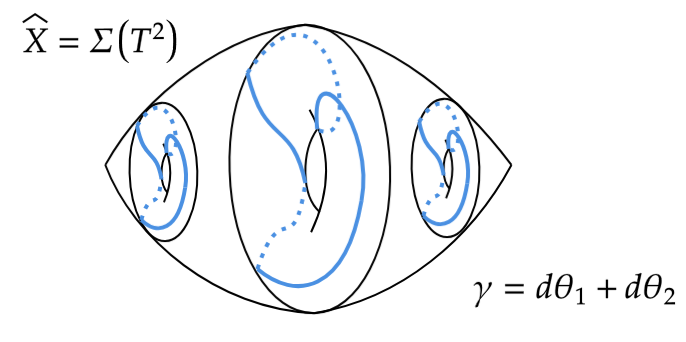}
    \caption{A stratified pseudomanifold arising as the suspension of a two torus. The blue curve depicts a homology class dual to the 1-form $\gamma=d\theta_1+d\theta_2$}
    \label{fig_1_sus_torus}
\end{figure}

\begin{example}
\label{Example_suspension_over_torus_first}
Consider the suspension over the two torus $\widehat{X}=\Sigma(\mathbb{T}^2)$, the resolution of which is $X=[0,\pi]_\phi \times \mathbb{T}^2_{\theta_1,\theta_2}$ with the wedge metric $d\phi^2+\sin^2(\phi) (d\theta_1^2 + d\theta_2^2)$, on which we consider the stratified Morse function $h=\cos(\phi)$. 

We define the form $\gamma:=d\theta_1+d\theta_2$ and observe that $\star_{\mathbb{T}^2} \gamma= -d\theta_1+d\theta_2=:\gamma'$. We pick the ideal boundary condition $W$ on $X$ corresponding to the sub-bundle consisting of the restriction of $\gamma$ to each of the two singular points. 
The de Rham operator is
\begin{equation}
    d=d\phi \wedge \nabla_{\partial_\phi} + \sin(\phi) d\theta_1 \wedge \nabla_{\frac{1}{\sin(\phi)}\partial_{\partial_{\theta_1}}} +\sin(\phi) d\theta_2 \wedge \nabla_{\frac{1}{\sin(\phi)}\partial_{\partial_{\theta_2}}}.
\end{equation}
It is easy to check that the sections $d\theta_1, d\theta_2$ as well as $d\phi \wedge d\theta_1, d\phi \wedge d\theta_2$ are in the null space of $d$.
The de Rham operator has domain 
\begin{equation}
    \mathcal{D}_W(d)=\{ \omega \in \mathcal{D}_{\max}(d) : \alpha (\omega_{\delta}) \in H^{-1/2}(Y ; W) \},
\end{equation}
where $Y=\{ \phi=0\} \cup \{ \phi=\pi\}$.
The domain of $\delta$ is
\begin{equation}
    \mathcal{D}_{W^{\perp}}(\delta)=\{ \omega \in \mathcal{D}_{\max}(\delta) : \beta (\omega_{d}) \in H^{-1/2}(Y ; W^{\perp}), \}
\end{equation}
and one can check that 
\begin{equation}
    \mathcal{D}_{W^{\perp}}(\delta_{\varepsilon})=\{ \omega \in \mathcal{D}_{\max}(\delta_\varepsilon) : e^{+\varepsilon \cos(\phi)}\beta (\omega_{d}) \in H^{-1/2}(Y ; W^{\perp}) \}.
\end{equation}
The cohomology of both of the complexes $\mathcal{D}_{W}(d)$ and $\mathcal{D}_{W^{\perp}}(\delta)$ are spanned by $1$, $\gamma$, the sections $d\phi \wedge \gamma'$ and the volume form, all of which are in the domain $\mathcal{D}_{W}(d) \cap \mathcal{D}_{W^{\perp}}(\delta)$.
\end{example}

\begin{remark}   
It is easy to check that $d\phi \wedge d\theta_1, d\phi \wedge d\theta_2$ are in the null space of $d$ and these elements are in the domain $\mathcal{D}_{\max}(d)$.
The sections $d\theta_1$ and $d\theta_2$ are co-exact (their duals $d(\phi d\theta_1), d(\phi d\theta_2)$ under the Hodge star operator are exact). These elements are in the domain $\mathcal{D}_{\max}(\delta)$.
\end{remark}

\subsubsection{Case of depth 2}
\label{subsection_depth_2}
We describe the case of depth 2 strata.
Consider a depth two stratification:
\begin{equation}
    Y^2 \subset Y^1 \subset \widehat{X},
\end{equation}
and let $Z^i - H^i \overset{\phi_i}{\longrightarrow} Y^i$ denote the boundary fibration where we label the boundary hypersurfaces according to the strata.

At depth one, we pick the mezzo-perversity $W \subseteq \mathcal{H}^{\dim Z^1/2}(H^1/Y^1)$ and define 
\begin{multline}
    \mathcal{D}_{\max,W^1}(P_X) = \{ u \in L^2(X;E) : P_Xu \in L^2(X;F),\\
    \; \chi u \in \mathcal{D}_{W^1}(D_{\widehat{X} \setminus Y^2}), \; \forall \chi \in C^\infty_c(\widehat{X} \setminus Y^2) \}
\end{multline}
where the domain $\mathcal{D}_{W^1}(P_{\widehat{X} \setminus Y^2})$ is well defined for all sections supported away from $Y^2$.
This determines a domain for the Dirac operator $\mathcal{D}_{W^1}(D_{\widehat{X} \setminus Y^2})$ as described in Subsection \ref{subsection_abstract_hilbert_complexes}, each element $u$ of which has 
a partial asymptotic expansion
\begin{equation}
    u \sim \alpha(u) +d\rho_{H^2} \wedge \beta(u) + \tilde{u}, \quad \tilde{u} \in \rho_{H^2}^{1-}H^{-1}_e(X;E),
\end{equation}
where $\rho_{H^2}$ is a boundary defining function for $H^2$.
There is a Cauchy data map 
\begin{equation}
    \mathcal{C}_{Y^2}(D)(u) = (\alpha(u),\beta(u)), u \in \mathcal{D}_{\max,W^1}(D_X) 
\end{equation}
similar to the case of depth 1, where $\alpha(u), \beta(u)$ are the orthogonal projections of the forms $u$ and $\iota_{({d\rho_{H^2}}^\#)}u$ respectively onto the space
\begin{equation}
    H^{-1/2}(Y^2; \Lambda^\bullet(T^*Y^2) \otimes \mathcal{H}^{l_2/2} (H^2/Y^2)),
\end{equation}
where $l_2$ is the dimension of the link $Z^2$ at $Y^2$, and where $\rho_{H^2}^{1-}H^{-1}_e(X;E) =  \bigcap_{\epsilon > 0} \rho_{H^2}^{1-\epsilon}H^{-1}_e(X;E)$ consists of elements in the minimal domain of $D$.
We denote by $W^2$ a flat subbundle of $\mathcal{H}^{l_2/2} (H^2/Y^2)$ which fibers over $Y^2$, satisfying Cheeger ideal boundary condition at $Y^2$. Then we define the domain
\begin{equation}
    \mathcal{D}_{(W^1, W^2)}(P_X) = \{u \in \mathcal{D}_{\max, W^1}(P_X): ({W^2}^{\perp},\{0\}) \circ \mathcal{C}_{Y^2} (u) = 0\},
\end{equation}
where $W^2,{W^2}^{\perp} \subseteq \mathcal{H}^{l_2/2} (H^2/Y^2)$, which determines domains for $P_X^*$ and the de Rham operator by the considerations in Subsection \ref{subsection_abstract_hilbert_complexes}.

\subsubsection{Arbitrary finite depth case}
For a stratified space 
\begin{equation}
    Y^k \subset Y^{k-1} \subset \cdots \subset Y^1 \subset \widehat{X}
\end{equation}
of depth $k$, we denote $Z^i - H^i \overset{\phi_i}{\longrightarrow} Y^i$ as the boundary fibration. A \textbf{\textit{mezzo-perversity}} is a collection of bundles 
\begin{equation}
    W = \{W^1 \to Y^1, \cdots, W^{k} \to Y^{k}\},
\end{equation}
where each $W^i$ is a subbundle of $\mathcal{H}^{\dim Z^i/2}(H^i \setminus Y^i)$ and has a corresponding Cheeger ideal boundary condition. The domains $\mathcal{D}_W(P_X)$ are defined by induction and we refer to \cite{Albin_hodge_theory_cheeger_spaces,albin2017novikov} for the details.

Given a global domain for $P_X$ by $\mathcal{D}_{W}(P_X)$, we denote the corresponding complex by 
\[\mathcal{P}_W(X):=\big(L^2\Omega^*(X;E),\mathcal{D}_{W}(P_X), P_X\big).\] 
In the case where $P_X = d^E$ is a twisted de Rham operator, the Cheeger ideal boundary condition defines a self-adjoint Fredholm complex $\mathcal{P}_W(X)$.  We call 
\begin{equation}
    \mathcal{Q}_{W^{\perp}}(X):=\big(\mathcal{P}_W(X))^* = (L^2\Omega^*(X;F),(\mathcal{D}_{W}(P_X))^*, P_X^*\big)
\end{equation}
the \textbf{\textit{adjoint complex}} of $\mathcal{P}_W(X)$.
A procedure similar to the case of the stratum of depth one when extended by induction as above shows that
\begin{equation}
    \big(\mathcal{D}_{W}(P_X)\big)^*=\mathcal{D}_{W^{\perp}}(P_X^*).
\end{equation}

\begin{remark}
\label{Witt_assumption}
It is easy to see that no choices are involved when the cohomology of the links vanishes in the middle degree over each stratum, which corresponds to the \textit{\textbf{Witt condition}}, in which case the Dirac operator is essentially self-adjoint.

This is of course after a suitable rescaling of the metric in order to ensure that there are no \textit{small eigenvalues} other than the zero eigenvalues for certain \textit{vertical Dirac operators} at the strata. This can always be achieved by rescaling the links (making them smaller) close to the strata so that the non-zero eigenvalues of Dirac-type operators on the links are large, and we refer the reader to a discussion in \cite[\S 2]{albin2017novikov}, and to  \cite[\S 2]{Albin_2017_index} for an extension to more general Dirac-type operators.
\end{remark}

\subsection{Domains and boundary conditions for complexes on subspaces.}
\label{subsubsubsection_Neumann_boundary_condition}

Several important local domains and local complexes for the de Rham and Dolbeault complexes at isolated critical points (and isolated fixed points of self-maps) were studied in \cite[\S 5.2.3]{jayasinghe2023l2}, \cite[\S 3.2.1]{jayasinghe2024holomorphic}.
Here we extend this to the case of global domains studied above for the de Rham complex, corresponding to fundamental neighbourhoods of critical points.

\subsubsection{Case of isolated critical points:}
We first study the case of fundamental neighbourhoods of isolated critical points, which also give a model for the fibers $U_a$ of the fundamental neighbourhoods $U(F_a)$ of non-isolated critical points that we study in this article. We note that these are Type II critical points iff they are isolated singularities. If not, and if they are also not smooth points, they are of Type I.
\begin{remark} 
    We follow the convention in \cite[\S 5.2.3]{jayasinghe2023l2} that $\mathcal{P}$ and $\mathcal{Q}$ represent complexes associated with the operators $d$ and $\delta$, respectively. The subscripts $N$ and $D$ indicate the maximal and minimal domains, respectively.
\end{remark}

Given a global twisted de Rham complex $\mathcal{P}_W(X)=\big(L^2\Omega^k(X;E),\mathcal{D}_W(P), P\big)$ where $X$ is a smooth manifold, when we restrict it to a smooth neighbourhood $U$ with smooth \textit{metric boundary} $\partial U$ (see Definition \ref{definition_smooth_boundary_and_singularities}), there are two canonical ways of defining a Hilbert complex in a neighbourhood that are of interest in Morse theory. We will denote the operator restricted to $U$ by $P_U$, denoting it by $P$ when it is understood by context that we are studying the operator of the complex on $U$.

We define the complex $\mathcal{P}_{N}(U)$ to be the Hilbert complex with the maximal domain 
\begin{equation}
    \mathcal{D}_{N}(P_U):=\mathcal{D}_{\max}(P_U)=\{ u \in L^2(U;E) : P_U u \in L^2(U;E) \},
\end{equation}
where $P_U u$ is defined in the distributional sense.
This fixes the domain for operators in the \textit{\textbf{adjoint Hilbert complex of $\mathcal{P}_N(U)$}}, and we will \textit{\textbf{denote it as $(\mathcal{P}_N(U))^*=\mathcal{Q}_N(U)$}},
where $\mathcal{Q}(X)=(L^2\Omega^{n-k}(X;E),\mathcal{D}(\delta),P^*=\delta)$. Here the boundary conditions for $\mathcal{Q}_N(U)$ match the boundary conditions for $\mathcal{P}_D(U)$ where the choice of domain indicated by the subscript $D$ corresponds to the minimal domain when restricted to $U$ which is smooth.

This induces a domain for the Dirac-type operator $D=P+P^*$.
We refer the reader to \cite[\S 5.2.3]{jayasinghe2023l2} for a more detailed exposition of the choices of domains in the smooth setting in the case of isolated critical points, and the boundary conditions at $\partial U$ satisfied by the sections of these domains.

Given a boundary defining function $\rho$ for $\partial U (=\{ \rho=0\})$ for an attracting critical point, the boundary conditions acting on smooth sections of the de Rham complex reduces to 
\begin{equation}
    \sigma(d^*)(d\rho)u|_{\rho=0} =0, \quad \sigma(d^*)(d\rho)du|_{\rho=0} =0,
\end{equation}
and for an expanding critical point we get
\begin{equation}
    \sigma(d)(d\rho)u|_{\rho=0} =0, \quad \sigma(d)(d\rho)d^*u|_{\rho=0} =0.
\end{equation}

We now consider the case when $\widehat{U} \subset \widehat{X}$ is a singular neighborhood of an isolated critical point. We first observe that we can apply the boundary conditions $N,D$ above at the metric boundary to extend the definitions of $\mathcal{D}_{N}(P_U),\mathcal{D}_{D}(P_U)$ for both $P=d_E$ and $P=\delta_E$, corresponding to the maximal domains on the local neighbourhoods. However, we need to pick local domains that respect the global choices of domains.

Given a complex $\mathcal{P}_{W}(X)$, we define the local complex 
\begin{equation}
    \mathcal{P}_{W,N}\big(U(F_a)\big)=\big(L^2\Omega(U(F_a);E),\mathcal{D}_{W,N}(P_{U(F_a)}), P_{U(F_a)}\big)
\end{equation}
with the domain
\begin{equation}
\label{equation_local_complex_first_factor}
    \mathcal{D}_{W, N}\big(P_{U(F_a)}\big) := \text{graph closure of } \big\{ \mathcal{D}_{N}(P_{U(F_a)}) \cap_{k} ({W^k}^{\perp},\{0\}) \circ \mathcal{C}_{Y^k} = 0 \big\},
\end{equation}
for $P_X=d_X$, 
and 
\begin{equation}
    \mathcal{D}_{W^{\perp}, N}\big(P^*_{U(F_a)}\big) := \text{graph closure of } \big\{ \mathcal{D}_{N}(P^*_{U(F_a)}) \cap_{k} (\{0\},{W^k}) \circ \mathcal{C}_{Y^k} = 0 \big\},
\end{equation}
for the adjoint $P^*=\delta_E$, where $W^k \subseteq \mathcal{H}^{l_2/2} (H^k/Y^k)$ are the linear operators corresponding to the ideal boundary condition at each stratum, as defined in the previous section on global domains. The Cauchy data $\mathcal{C}_Y$ vanishes only for the restriction of the sections to the neighbourhood $U(F_a)$.

If we have an isolated critical point $U(F_a)=U_{a,s} \times U_{a,u}$, then we define
\begin{equation}
    \mathcal{P}_{W,B}(U(F_a)):=\mathcal{P}_{W,N}(U_{a,s}) \times \mathcal{Q}^*_{W^{\perp},N}(U_{a,u}),
\end{equation}
where $\mathcal{Q}^*_{W^{\perp},N}(U_{a,u})=\mathcal{P}_{W,D}(U_{a,u})$ is the local adjoint complex.
In this case, the local cohomology in degree $k$ at the critical point is simply given by
\begin{equation}
    \mathcal{H}^{k}(\mathcal{P}_{W,B}(U(F_a))) =\sum_{k=k_1+k_2} \mathcal{H}^{k_1}(\mathcal{P}_{W,N}(U_{a,s})) \otimes \mathcal{H}^{k_2}(\mathcal{Q}^*_{W^{\perp},N}(U_{a,u})).
\end{equation}

\begin{remark}[Domains defined with refined stratification]
Technically we use the refined stratification that corresponds to the resolved manifold with boundary with iterated fibration structure introduced in Subsection \ref{subsection_stratified_Morse_Bott_functions}, where the boundary conditions and ideal boundary conditions described above can be understood using the Green-Stokes theorem.
\end{remark}

\subsubsection{Case of general critical points:}

We now consider the case of a fundamental neighbourhood of a general critical point $U(F_a)$. Recall that $U(F_a)$ is the metric closure of the fibered neighbourhood $U(F_a^{reg})$ over $F_a^{reg}$, which is dense in $U(F_a)$. We also have that a fiber $U_a$ over a point in $F^{reg}_a$ is a product of the stable and unstable fibers $U_{a,s} \times U_{a,r}$. The metric boundary $\partial U(F_a^{reg})$ of $U(F_a^{reg})$ is dense in $\partial U(F_a)$ of $U(F_a)$.
We can define the boundary conditions corresponding to the generalized Neumann conditions for $P_E$ and $P^*_E$ on the stable and unstable fibers of each fiber over $F^{reg}$. It is easy to see by our assumptions on the metric on $U(F_a)$ that this corresponds to boundary conditions on the metric boundary $\partial U(F_a^{reg})$. By the density of $\partial U(F_a^{reg})$ in $\partial U(F_a)$, these boundary conditions completely describe a domain $\mathcal{D}_{B}(P_U(F_a))$ of the complex on $U(F_a)$. If the fibers are completely stable (gradient flow of the stratified Morse-Bott function is attracting on every fiber), then this domain is $\mathcal{D}_{N}(P_U(F_a))$, the maximal domain on $U(F_a)$.
We define the domain,
\begin{equation}
    \mathcal{D}_{W,B}(P_{U(F_a)}):=\text{graph closure of } \{ \mathcal{D}_{B}(P_{U(F_a)}) \cap_{k} ({W^k}^{\perp},\{0\}) \circ \mathcal{C}_{Y^k} = 0 \},
\end{equation}
for $P_X=d_X$.

\subsection{Local cohomology groups and complexes for general fundamental neighborhoods}
\label{subsection_local_cohomology_complexes}

The local cohomology of this domain is
\begin{equation}
    \mathcal{H}^k(\mathcal{P}_{W,B}(U(F_a))
\end{equation}
which, in the case where the fibration over $F^{reg}_a$ is trivial can be expanded as
\begin{equation}    \mathcal{H}^k(\mathcal{P}_{W,B}(U(F_a))=\sum_{k_0+k_1+k_2=k} 
    \mathcal{H}^{k_0}(\mathcal{P}_{W}(F_a)) \otimes\mathcal{H}^{k_1}(\mathcal{P}_{W,N}(U_{a,s})) \otimes \mathcal{H}^{k_2}(\mathcal{Q}^*_{W^{\perp},N}(U_{a,u})),
\end{equation}
where $U_{a,s} \times U_{a,u}$ are the fibers over $F_a$ in $U(F_a)$ and where $\mathcal{P}_{W}(F_a)$ is the de Rham complex on the stratified space $F_a$ with what we shall call the \textit{\textbf{induced mezzo-perversity}} $W(F_a)$ (we shall introduce this momentarily) which we denote by $W$ with abuse of notation.
Equivalently, the cohomology of the complex can be expressed as
\begin{equation}
    \mathcal{H}^k(\mathcal{P}_{W,B}(U(F_a))=\sum_{k_0+k_3=k} 
    \mathcal{H}^{k_0}(\mathcal{P}_{W}(F_a)) \otimes\mathcal{H}^{k_3}(\mathcal{P}_{W,B}(U_{a}))=\mathcal{H}^k(\mathcal{R}_{W,B}(U(F_a))),
\end{equation}
where 
\begin{equation}
    \mathcal{R}_{W,B}(U(F)):=(L^2\Omega^{\cdot}(F;\mathcal{H}^{\cdot}(\mathcal{P}_{W,B}(U_{a})) \otimes E), \mathcal{D}_{W(F)}(d^H_E), d^H_E),
\end{equation}
for any $F=F_a$, where the operator $d^H_E$ is the \textbf{horizontal de Rham operator} at $F$. Given a local frame of vector fields $\{e^H_i\}$ on $F_a^{reg}$ 
we can construct the horizontal de Rham operator
\begin{equation}
    d_E^H:=\sum {(e^H_i)}^{\flat} \wedge \nabla_{e^H_i}
\end{equation}
acting on forms with coefficients in the bundle $\mathcal{H}^{\cdot}(\mathcal{P}_{W,B}(U_{a})) \otimes E$ over $F_a$. We can define the \textbf{vertical de Rham operator} by $d^V_E:=d-d^H_E$.
Since the sections in $L^2\Omega^{\cdot}(F;\mathcal{H}^{\cdot}(\mathcal{P}_{W,B}(U_{a})) \otimes E)$ are a subset of $L^2\Omega^{\cdot}(U(F);E)$, given a mezzo-perversity $W$ on $X$, it determines an ideal boundary condition for the sections in $L^2\Omega^{\cdot}(F;\mathcal{H}^{\cdot}(\mathcal{P}_{W,B}(U_{a})) \otimes E)$ which we call the \textit{\textbf{induced mezzo-perversity}}, crucially when extending the sections from $F^{reg}$ to $F$.

\textbf{The last description extends to the general case where the fibration of $U(F_a)$ over $F_a$ is not trivial}. This is due to our assumption that the metric is of model type near critical point sets, and hence the de Rham and Laplace type operators split into horizontal and vertical versions at the critical points (in particular the cohomology groups on the vertical fibers are well defined).

The spaces $\mathcal{R}_{W,B}(U(F_a))$ are called the \textbf{\textit{normal cohomology complexes}} near critical point sets in the general stratified Morse-Bott setting of this article. This is an analog of twisting by local coefficients given by the cohomology of the fibers that are used in topological versions of Morse inequalities and equivariant index theorems.

We also define the \textit{\textbf{horizontal/vertical}} Laplacians $\Delta_E^H/\Delta_E^V$ by 
\begin{equation}
    \label{horizontal_vertical_Laplace_operators}
    \Delta_E^H:=d_E^H (d_E^H)^*+(d_E^H)^* d_E^H, \quad \Delta_E^V:=d_E^V (d_E^V)^*+(d_E^V)^* d_E^V
\end{equation}
where $(d_E^H)^*$ and $(d_E^V)^*$ are the adjoints of $d_E^H$ and $d_E^V$ respectively on the fundamental neighbourhoods. The flatness condition we use is equivalent to the fact that the operators $\Delta_E^H$ and $\Delta_E^V$ commute, where the corresponding curvature term is computed explicitly in the smooth case in \cite{bismut1986witten}.

\begin{remark}
The bundle $W$ is flat with respect to the Gauss-Manin connection.
We see that the bundle of cohomology is a flat bundle since the cohomology is in the null space of the vertical de Rham operator. Over $F^{reg}$, the vertical de Rham operator corresponding to the Gauss-Manin connection is independent of the choice of wedge metric (see Lemma 3.1 \cite{albin2017novikov}).
\end{remark}

\subsection{Poincar\'e dual complexes and self-dual mezzo-perversities}
\label{Subsection_Poincare_Dual_complexes}

Given a twisted de Rham complex 
\begin{equation}
    \mathcal{P}_W(X)=(L^2\Omega^k(X;E),\mathcal{D}_W(P_X),P_X),
\end{equation} 
we define \textit{\textbf{the Poincar\'e dual complex}} to be 
\begin{equation}
 (\mathcal{P}_W(X))_{PD}:=\mathcal{P}_{\star W^{\perp}}(X)=
(L^2\Omega^{\cdot}(X;E),\mathcal{D}_{\star W^{\perp}}(P_X),P_X),
\end{equation}
which is the adjoint of the complex 
\begin{equation}
\mathcal{Q}_{\star W(X)}=(L^2\Omega^{n-\bullet}(X;E),\mathcal{D}_{\star W}(P^*_X),P^*_X),
\end{equation}
which we call the \textit{\textbf{the adjoint Poincar\'e dual complex}} where
\begin{equation}
    \mathcal{D}_{\star W^{\perp}}(P):=\{ s \in L^2\Omega^k(X;E) \mid \star s \in \mathcal{D}_{W}(P^*) \subseteq L^2\Omega^{n-k}(X;E) \}.
\end{equation}

Given a normal cohomology complex 
\begin{equation}
    \mathcal{R}_{W,B}(U(F_a))=(L^2\Omega^{\bullet}(F;\mathcal{H}^{\bullet}(\mathcal{P}_{W,B}(U_{a})) \otimes E), \mathcal{D}_{W(F)}(d^H_E), d^H_E),
\end{equation}
the adjoint normal cohomology complex is 
\begin{equation}
    \mathcal{R}_{W^{\perp}, B^{\perp}}(U(F_a))=(L^2\Omega^{\bullet}(F;\mathcal{H}^{\bullet}(\mathcal{P}_{W^{\perp},B^{\perp}}(U_{a})) \otimes E), \mathcal{D}_{W(F)^{\perp}}(d^H_E), d^H_E),
\end{equation}
where $\star N=D$ and $\star D= N$, which determines $\star B$ which we also denote as $B^{\perp}$. The local complexes obtained by restricting to a fundamental neighbourhood of a connected component of the critical point set have an analogous duality given by the Hodge star operator which plays the role of a \textbf{Dirichlet to Neumann operator}, intertwining the boundary conditions.

If there is a perversity $W$ such that $W=\star W^{\perp}$, then we say that $W$ is a \textbf{self-dual mezzo-perversity}, and we can define a Signature operator following \cite{Albin_hodge_theory_cheeger_spaces,Albin_hodge_theory_on_stratified_spaces,albin2017novikov}.
In this case since the $\star$ operator preserves the forms in the domain of $d+\delta$ corresponding to the de Rham complex, the middle dimensional forms decompose into a direct sum of vector spaces, where the forms in each are either self-dual ($\omega=\star \omega$) or anti-self-dual ($\omega=-\star \omega$). This descends to a $\mathbb{Z}_2$-grading of the middle dimensional harmonic forms and the signature invariant is then the difference of the dimensions of these spaces, and corresponds to the Fredholm index of the signature operator.

Spaces which admit self-dual mezzo-perversities are called $L$-spaces or Cheeger spaces, and the signature is independent of the choice of self-dual mezzo-perversity.

\begin{remark}
The fact that the signature is independent of the choice of mezzo-perversity was noticed used in \cite{Albin_hodge_theory_cheeger_spaces,RefinedintersectionAlbinBanagl2015} to show that the Signature is a topological invariant. The key idea can be easily seen locally in the case of an isolated singularity, where given a self-dual form which is not in the minimal domain of the operator $d+\delta$, it can be written as $\alpha \pm dx \wedge \star \alpha$ on a neighbourhood of the conic point, where $x$ is the radial distance from the conic point. Then $\alpha \mp dx \wedge \star \alpha$ is an anti-self-dual form in a local neighbourhood. 
\end{remark}

Given a four dimensional smooth manifold equipped with a flat vector bundle $E$, the twisted self-dual and anti-self dual complexes are
\begin{equation}
    0 \rightarrow \Omega^0(X;E) \xrightarrow{d_E} \Omega^1(X;E) \xrightarrow{d^{\pm}_E} \Omega^{2\pm}(X;E) \rightarrow 0, 
\end{equation}
where $d^{\pm}_E:=\Pi^{\pm} \circ d_E$ where $\Pi^{\pm}$ are the projections onto the self-dual ($+$) and anti-self-dual $(-)$ complexes. It is easy to see that these definitions extend to stratified pseudomanifolds given a choice of self-dual mezzo-perversity.
However, the indices of the self-dual and anti-self-dual complexes depend on the choice of self-dual mezzo-perversity on such spaces.

\begin{remark}
    We observe that $\star (W^{\perp})= (\star W)^{\perp}$, and that the Hodge star operator intertwines the operators $\mathcal{D}_W(d_E)$, $\mathcal{D}_{\star W^{\perp}}(\delta_E)$, corresponding to the complexes $\mathcal{P}_W(X)$ and $\mathcal{Q}_{\star W}(X)$.
\end{remark}
In Section \ref{section_Witten_deformation}, we will see that Poincar\'e duality extends to deformed complexes.

\subsection{Polynomial supertraces}
\label{subsection_polynomial_Lefschetz_supertraces}

The categorification of the Morse inequalities given by the Witten instanton complexes can be used to easily derive the Morse inequalities by using trace formulae on them. Here we collect some trace formulas and dualities that we use to prove such results.
We refer the reader to \cite[\S 4]{jayasinghe2023l2} and \cite[\S 3.3]{jayasinghe2024holomorphic} for  generalizations of these.

\begin{definition}
\label{L2_Lefschetz}
Let $X$ be a pseudomanifold with a twisted de Rham complex $\mathcal{P}=(H,\mathcal{D}(P), P)$, then we define the associated \textit{\textbf{Poincar\'e polynomial}} to be
\begin{equation} 
L(\mathcal{P})(b):= \sum_{k=0}^n b^k \dim {\mathcal{H}^k(\mathcal{P})} \in \mathbb{Z}_0^+[b],
\end{equation}
and the associated \textit{\textbf{Euler characteristic}} to be
\begin{equation} 
L(\mathcal{P}) = L(\mathcal{P})(-1).
\end{equation}
\end{definition}

\begin{definition}
\label{definition_polynomial_Lefschetz_supertrace}
Let $\mathcal{P}=(H,\mathcal{D}(P),P)$ be a finite dimensional Hilbert complex.
For all $t \in \mathbb{R}^+$, we define the \textbf{\textit{polynomial Lefschetz heat supertrace}} as
\begin{equation}
    \mathcal{L}(\mathcal{P})(b,t)=\mathcal{L}(\mathcal{P}(X))(b,t):=\sum_{k=0}^n  Tr(b^k e^{-t \Delta_k}),
\end{equation}
and we call $\mathcal{L}(\mathcal{P})(-1,t)$ the \textit{\textbf{Lefschetz heat supertrace}} associated to the complex. Here, we use the notation $\mathcal{L}(\mathcal{P}(X))(b,t)$ when the complex $\mathcal{P}=\mathcal{P}(X)$ is associated to a pseudomanifold $X$.
\end{definition}
The following is a simplified version of Theorem 3.14 of \cite{jayasinghe2023l2}.

\begin{theorem}
\label{Lefschetz_supertrace}
Let $\mathcal{P}=(H,\mathcal{D}(P), P)$ be a finite dimensional Hilbert complex. 
For all $t \in \mathbb{R}^+$,
\begin{equation} 
\label{equation_with_the_b}
    \mathcal{L}(\mathcal{P})(b,t)=L(\mathcal{P})(b)+ (1+b) \sum_{k=0}^{n-1} b^k S_k(t),
\end{equation}
where
\begin{equation} 
\label{equation_error_in_Lefschetz_supertrace}
    S_k(t)=\sum_{\lambda_i \in Spec(\Delta_k)} e^{-t \lambda_i}  \langle v_{\lambda_{i}}, v_{\lambda_{i}} \rangle,
\end{equation}
where $\{v_{\lambda_i}\}_{i \in \mathbb{N}}$ are an orthonormal basis of co-exact eigensections of $\Delta_k$. In particular,
\begin{equation}
\label{Heat_formula_all_t_P}
    L(\mathcal{P})= \mathcal{L}(\mathcal{P})(-1,t),
\end{equation}
and the Lefschetz heat supertrace is independent of $t$.
\end{theorem}

\begin{remark}
In \cite{jayasinghe2024holomorphic} where we studied Dolbeault complexes, we showed that there are isomorphisms between the local and global adjoint complexes and the Serr\'e dual complexes when the complexes satisfy the Witt condition. Similarly for de Rham complexes, the Poincar\'e dual and adjoint complexes are isomorphic for Witt spaces but are different for non-Witt spaces in general, which is why we use the adjoint complexes in the result below. Moreover the geometric endomorphism is simply the identity morphism in this case.
\end{remark}

\begin{proposition}[Duality]
\label{proposition_Lefschetz_on_adjoint}
Let $\mathcal{P}=(H,\mathcal{D}(P), P)$ be an elliptic complex of maximal non-trivial degree $n$, and let $\mathcal{Q}$ be the adjoint complex. Then
\begin{equation}
\label{Kalman_filter}
   b^n \mathcal{L}(\mathcal{Q})(b^{-1},t)=\mathcal{L}(\mathcal{P})(b,t).
\end{equation}
In particular, we have the equality
\begin{equation}
    L(\mathcal{P})= (-1)^n L(\mathcal{Q}).
\end{equation}
\end{proposition}

The following result summarizes the main dualities we use in this article, and is a generalization of Proposition 7.5 of \cite{jayasinghe2023l2} where it was proven for the VAPS domain on Witt spaces for general geometric endomorphisms.

\begin{proposition}
\label{Proposition_duality_complex_conjugation_for_local_Lefschetz_numbers}
Let $\widehat{X}$ be a stratified pseudomanifold of dimension $n$ with a wedge metric and let $E$ be a flat bundle. Let $\mathcal{P}_{W}(X)=\big(L^2\Omega^{\bullet}_{W}(X;E),\mathcal{D}_{W}(P), P=d_E \big)$ with the adjoint complex $(\mathcal{Q}_{W^{\perp}})(X)$.
Then we have that
\begin{equation}
    \mathcal{L}(\mathcal{P}_{W}(X))(b,t)=b^n \mathcal{L}\big((\mathcal{Q}_{W^{\perp}})(X)\big)(b^{-1},t). 
\end{equation}
Similarly for local complexes, we have
\begin{equation}
       \mathcal{L}\big(\mathcal{P}_{W,B}(U(F_a))\big)(b,t) = b^n \mathcal{L}\big((\mathcal{P}_{W^{\perp},B^{\perp}})(U(F_a))_{PD}\big)(b^{-1},t),
\end{equation}
and for the normal cohomology complexes
\begin{equation}
    \mathcal{L}\big(\mathcal{R}_{W,B}(U(F_a))\big)(b,t) = b^n \mathcal{L}\big((\mathcal{R}_{W^{\perp},B^{\perp}})(U(F_a))_{PD}\big)(b^{-1},t).
\end{equation}
\end{proposition}

\begin{proof}
The proofs are similar for both the global and local complexes and follow from the isomorphisms between complexes, their adjoint complexes.  
The first equality in each follows from Proposition \ref{proposition_Lefschetz_on_adjoint}.

\end{proof}

\section{Witten deformation}
\label{section_Witten_deformation}

The use of Witten deformation to obtain Morse inequalities for degenerate critical points in the sense of Morse-Bott in the smooth setting was already discussed in Witten's original article \cite{witten1982supersymmetry}. A modified deformation method was used by Bismut in \cite{bismut1986witten}, but it was shown by Wu and Zhang in \cite{wu1998equivariant} that much simpler localization techniques in \cite{bismut1991complex} could be used, similar in spirit to Witten's original arguments, and \cite{Zhanglectures} has a brilliant exposition which we follow, with adjustments including some generalizations of those done in \cite{jayasinghe2023l2,jayasinghe2024holomorphic}.

In the first subsection we study Witten deformed Hilbert complexes, both local and global, then study a Bochner type formula for the deformed Laplace-type operator which is key to understanding the localization of global cohomology near critical points, as well as establishing a spectral gap result for the model operator at connected components of critical point sets of a flat stratified Morse-Bott function $h$.

\subsection{Witten deformed Hilbert complexes}
\label{subsection_Witten_deformed_elliptic}

Given a stratified pseudomanifold $\widehat{X}$ equipped with a wedge metric, a flat vector bundle $E$ on $X$, we denote the twisted de Rham complex by $\mathcal{P}_W(X)=\big(H_k=L^2\Omega^k(X;E),\mathcal{D}_W(P),P\big)$. Given a flat stratified Morse-Bott function which resolves to $h$, we define the Witten deformed differential $P_\varepsilon:=e^{-\varepsilon h}Pe^{\varepsilon h}$, where $\varepsilon \geq 0 $ is a parameter. Since $h$ is smooth (i.e., lifts to $C_{\Phi}^{\infty}(X)$), the map $P_\varepsilon-P=\varepsilon (dh) \wedge$ extends to a bounded map $H_k \rightarrow H_{k+1}$ for any $k$, as does $\varepsilon \iota_{dh^{\#}}$.
We follow \cite[\S 4.3]{Zhanglectures} and introduce the left and right Clifford operators 
\begin{equation}
    cl(u)=u \wedge - \iota_{u^\#}, \quad \widehat{cl}(u)=u \wedge + \iota_{u^\#}
\end{equation}
where $cl(u)$ is the Clifford multiplication for the de Rham complex. Zhang's notation is for the Clifford algebra on the tangent space, while we use that on the wedge cotangent bundle by duality. 
Zhang shows that $D_\varepsilon= D+ \varepsilon \widehat{cl}(dh)$ and we refer the reader to \cite{Zhanglectures} for more details.  
We consider forms with coefficients in a bundle $E$, equivalently sections of $F=\Lambda^*(\prescript{w}{}{T^*X}) \otimes E$.
The adjoint of $P_\varepsilon=e^{-\varepsilon h}Pe^{\varepsilon h}$ is $P^*_\varepsilon:=e^{\varepsilon h}P^*e^{-\varepsilon h}$.

Let us consider the domains for these deformed operators on $X$.
Since $h \in C_{\Phi}^{\infty}(X)$, given $u \in L^2(X;F)$ we have that $\widehat{cl}(dh)u \in L^2(X;F)$. Moreover given a sequence $(u_n) \subset C^{\infty}_c(\mathring{X},F)$, we have that $(\widehat{cl}(dh)u_n) \subset C^{\infty}_c(\mathring{X},F)$. Thus if $D(u_n)$ is $L^2$-Cauchy, then so is $D_\varepsilon(u_n)$.
It follows that $\mathcal{D}_{\min}(D)=\mathcal{D}_{\min}(D_\varepsilon)$. 
An analogous argument shows that $\mathcal{D}_{\min}(P)=\mathcal{D}_{\min}(P_\varepsilon)$ and that  $\mathcal{D}_{\min}(P^*)=\mathcal{D}_{\min}(P^*_\varepsilon)$.
Since the maximal domain of the deformed de Rham operator is defined as $\mathcal{D}_{\max}(P_\varepsilon)=\{ v \in L^2\Omega^*(X) \mid P_\varepsilon v \in L^2\Omega^*(X) \}$ where $P_\varepsilon=e^{-\varepsilon h}Pe^{\varepsilon h}$ we see that the domain $\mathcal{D}_{\max}(P_\varepsilon)$ is isomorphic to $\mathcal{D}_{\max}(P_0)$, given by multiplication by the function $e^{-\varepsilon h}$.

We define the domains for the deformed de Rham operators
\begin{equation}
    \mathcal{D}_{W}(d_\varepsilon)= \{ v \in \mathcal{D}_{\max}(d_\varepsilon) \mid  e^{+\varepsilon h}v \in \mathcal{D}_{W}(d)  \}
\end{equation}
which generalizes the domains corresponding to the mezzo-perversity $W$ to the deformed case.

\begin{remark}   
As we discussed in the previous section, the sections in the maximal domain of $D=d+\delta$ have expansions
\begin{equation}
    v \sim \alpha(v)+dx \wedge \beta(v)+ \widetilde{v}
\end{equation}
near the boundary hypersurfaces $Y^k$ where $x$ is a boundary defining function for $Y^k$, where $\alpha(v), \beta(v) \in H^{-1/2}(Y^k; \Lambda^*(T^*Y^k) \otimes \mathcal{H}^{l_k/2}(\mathcal{P}(\partial X / Y))$
and $\widetilde{v}$ are sections in $x^{1-}H^{-1}(X; \Lambda^* \prescript{w}{}{TX})$, in particular in the minimal domain of $D$.
The sections in the maximal domain for the deformed complex have expansions of the form
\begin{equation}
    v \sim \alpha_\varepsilon(v)+dx \wedge \beta_\varepsilon(v)+ \widetilde{v}
\end{equation}
where $\alpha_\varepsilon(v), \beta_\varepsilon(v)$ can be defined locally near each singular stratum as $\alpha_\varepsilon(v)|_{Y}:=\alpha(v)e^{-\varepsilon h}$ where $x$ is the radial defining function near the stratum $Y$ for a wedge metric which is locally product type, and  $\beta_\varepsilon(v)|_{Y}:=\beta(v)e^{-\varepsilon h}$.

This is because the $\{0\}$ eigenvalues of the Witten deformed Laplace-type operator on a model fundamental neighbourhood corresponds to the terms of the form $\alpha(u_\varepsilon), \beta(u_\varepsilon)$ upto terms in the minimal domain of the Laplace-type operator. Here $\alpha(u_\varepsilon)$ are in fact in the null space of the normal operator $d$ at the boundary hypersurface, while $dx \wedge \beta(u_\varepsilon)$ are in the null space of the normal operator of $\delta$. This generalizes the Cauchy data we studied above in the case for $\varepsilon=0$ to the singular case.

Then it is easy to see that 
\begin{equation}
    \mathcal{D}_{W}(d_\varepsilon)= \{ v \in \mathcal{D}_{\max}(d_\varepsilon) \mid  \alpha_\varepsilon(v)  \in H^{-1/2}(Y; \mathcal{H}^{l/2}(\mathcal{P}(\partial X / Y)) \}.
\end{equation}
\end{remark}

Observe that multiplication by $e^{\varepsilon h}$ is a homeomorphism $H_k \rightarrow H_k$ that takes a section in the domain of any closed extension $\mathcal{D}_{W}(d)$ to a section in $\mathcal{D}_{W}(d_\varepsilon)$
and conjugates $d$ and $d_\varepsilon$.
We define the \textbf{Witten deformed complex} $\mathcal{P}_{W,\varepsilon}(X)=\big(L^2\Omega^k(X;E),\mathcal{D}_{W,\varepsilon}(d_\varepsilon),d_\varepsilon\big)$ by changing the differential in $\mathcal{P}(X)$ to $d_\varepsilon$. Multiplication by $e^{-\varepsilon h}$ induces an isomorphism between the cohomology of $\mathcal{P}_{W,\varepsilon}(X)$ and $\mathcal{P}_{W,0}(X):=\mathcal{P}_{W}(X)$. 
This can be summarized using the commutative diagram 
\[\begin{tikzcd}
	{...} && {H_k} && {H_{k+1}} && {...} \\
	\\
	{...} && {H_{k}} && {H_{k+1}} && {...}
	\arrow["P", from=1-1, to=1-3]
	\arrow["P", from=1-3, to=1-5]
	\arrow["P", from=1-5, to=1-7]
	\arrow["{P_\varepsilon}", from=3-3, to=3-5]
	\arrow["{P_\varepsilon}", from=3-1, to=3-3]
	\arrow["{P_\varepsilon}", from=3-5, to=3-7]
	\arrow["{e^{-\varepsilon h}}", from=1-3, to=3-3]
	\arrow["{e^{-\varepsilon h}}", from=1-5, to=3-5]
\end{tikzcd}\]
where the domains $\mathcal{D}(P_{k})$ and $\mathcal{D}(P_{\varepsilon,k})$ can be canonically identified.
Similarly we define
\begin{equation}
    \mathcal{D}_{W^{\perp}}(d^*_\varepsilon)=\{ v \in \mathcal{D}_{\max}(d^*_\varepsilon) \mid  e^{-\varepsilon h}v \in \mathcal{D}_{W^{\perp}}(d^*)  \}.
\end{equation}

\begin{remark}
It is clear that these domains can be characterized as 
\begin{equation}
    \mathcal{D}_{W^{\perp}}(d^*_\varepsilon)= \{ v \in \mathcal{D}_{\max}(d^*_\varepsilon) \mid  \beta_\varepsilon(v) \in H^{-1/2}(Y; \mathcal{H}^{l/2}(\mathcal{P}(\partial X / Y)) \}
\end{equation}
in light of the work in \cite{Albin_hodge_theory_cheeger_spaces,Albin_hodge_theory_on_stratified_spaces}.
\end{remark}

Similar to the complexes on $X$, when we restrict to a fundamental neighbourhood $U(F_a)$ of a critical point $F_a$ of $h$, we define the \textbf{\textit{local} Witten deformed complexes} $\mathcal{P}_{W,N,\varepsilon}(U(F_a))$. 
There are isomorphisms from the Hilbert complexes $\mathcal{P}_{W,N}(U(F_a))$ to $\mathcal{P}_{W,N,\varepsilon}(U(F_a))$, and we have that $\big(\mathcal{Q}_{W,N}(U(F_a))\big)^*=\mathcal{P}_{W^{\perp},D}(U(F_a))$ to $(\mathcal{Q}^*_{W,N,\varepsilon}(U(F_a)))^*=\mathcal{P}_{W^{\perp},D,\varepsilon}(U(F_a))$ given by multiplication by the function $e^{-\varepsilon h}$ when $P=d$ (and by $e^{+\varepsilon h}$ when $P=d^*$). 
It is easy to check that the \textbf{Poincar\'e dual complex of a Witten deformed complex is the Witten deformed complex of a Poincar\'e dual complex}.

Similarly, multiplication by
$e^{+\varepsilon h}$ gives an isomorphism from $\mathcal{Q}_{W,N}(U(F_a))$ to $\mathcal{Q}_{W,N,\varepsilon}(U(F_a))$, and  $\big(\mathcal{P}_{W,N}(U(F_a))\big)^*$ to $\big(\mathcal{P}_{W,N,\varepsilon}(U(F_a))\big)^*$ for $P=d$.
We also have \textbf{\textit{local} Witten deformed cohomology complexes} 
$\mathcal{R}_{W,B,\varepsilon}(U(F_a))$, which satisfy similar properties with the adjoint and Poincar\'e duals as well.

The results in \cite{Albin_hodge_theory_cheeger_spaces,Albin_hodge_theory_on_stratified_spaces} together with explicit isomorphisms of Hilbert complexes given by the conjugation by $\exp(\pm \varepsilon h)$ suffices to extract the information about the complexes that we will require for various proofs in this article.
Lemma 2.17 of \cite{bru1992hilbert} shows that discreteness is invariant under Hilbert complex isomorphisms and Corollary 2.19 of \cite{bru1992hilbert} shows that the cohomology of the complexes are isomorphic for all values of $\varepsilon \geq 0$.

We study the Witten deformed operators and cohomology for the space in Example \ref{Example_suspension_over_torus_first}.

\begin{example}
\label{Example_suspension_torus_Witten_deformation}
Consider the suspension over the two torus $\widehat{X}=\Sigma(\mathbb{T}^2)$, the resolution of which is $X=[0,\pi]_\phi \times \mathbb{T}^2_{\theta_1,\theta_2}$ with the wedge metric $d\phi^2+\sin^2(\phi) (d\theta_1^2 + d\theta_2^2)$, on which we consider the stratified Morse function $h=\cos(\phi)$. We define the form $\gamma:=d\theta_1+d\theta_2$ and observe that $\star_{\mathbb{T}^2} \gamma= -d\theta_1+d\theta_2=:\gamma'$. We pick the ideal boundary condition $W$ on $X$ corresponding to the sub-bundle consisting of the restriction of $\gamma$ to each of the two singular points. Since $\star W^{\perp}= \star \{d\theta_1-d\theta_2\}=\{ \gamma \}$ we see that this mezzo perversity is self-dual (in particular $X$ is a Cheeger space).
The de Rham operator is
\begin{equation}
    d=d\phi \wedge \nabla_{\partial_\phi} + \sin(\phi) d\theta_1 \wedge \nabla_{\frac{1}{\sin(\phi)}\partial_{\theta_1}} +\sin(\phi) d\theta_2 \wedge \nabla_{\frac{1}{\sin(\phi)}\partial_{\theta_2}}
\end{equation}
and it is easy to check that the sections $d\theta_1, d\theta_2$ as well as $d\phi \wedge d\theta_1, d\phi \wedge d\theta_2$ are in the null space of $d$.
The de Rham operator has domain 
\begin{equation}
    \mathcal{D}_W(d)=\{ \omega \in \mathcal{D}_{\max}(d) : \alpha (\omega_{\delta}) \in H^{-1/2}(Y ; W) \}
\end{equation}
where $Y=\{ \phi=0\} \cup \{ \phi=\pi\}$.
The Witten deformed operator is 
\begin{equation}
    d_\varepsilon=(d\phi \wedge \nabla_{\partial_\phi} + \varepsilon d\cos(\phi)\wedge ) + \sin(\phi) d\theta_1 \wedge \nabla_{\frac{1}{\sin(\phi)}\partial_{\theta_1}} +\sin(\phi) d\theta_2 \wedge \nabla_{\frac{1}{\sin(\phi)}\partial_{\theta_2}}
\end{equation}
and the domain is 
\begin{equation}
    \mathcal{D}_W(d_\varepsilon)=\{ \omega \in \mathcal{D}_{\max}(d_\varepsilon) : e^{-\varepsilon \cos(\phi)}\alpha (\omega_{\delta}) \in H^{-1/2}(Y ; W) \}
\end{equation}
where $Y=\{ \phi=0\} \cup \{ \phi=\pi\}$.
The domain of $\delta$ is
\begin{equation}
    \mathcal{D}_{W^{\perp}}(\delta)=\{ \omega \in \mathcal{D}_{\max}(\delta) : \beta (\omega_{d}) \in H^{-1/2}(Y ; W^{\perp}) \}
\end{equation}
and one can check that 
\begin{equation}
    \mathcal{D}_{W^{\perp}}(\delta_{\varepsilon})=\{ \omega \in \mathcal{D}_{\max}(\delta_\varepsilon) : e^{+\varepsilon \cos(\phi)}\beta (\omega_{d}) \in H^{-1/2}(Y ; W^{\perp}) \}.
\end{equation}
The cohomology of both of the complexes $\mathcal{D}_{W}(d_\varepsilon)$, and $\mathcal{D}_{W}(\delta_\varepsilon)$ are spanned by $e^{-\varepsilon h}$, $e^{-\varepsilon h} \gamma$, and the sections $e^{+\varepsilon h}d\phi \wedge \gamma'$ and the volume form multiplied by $e^{+\varepsilon h}$, all of which are in the domain $\mathcal{D}_{W}(d_{\varepsilon}) \cap \mathcal{D}_{W^{\perp}}(\delta_{\varepsilon})$. Since $h=\cos(\phi)$ we see that the first two elements of the above spanning set converge in $L^2$ mass near $\phi=\pi$ (the function $e^{-\varepsilon \cos(\phi)}$ on the interval $[0,\pi]_{\phi}$ converges in $L^2$ mass near $\phi=\pi$ as $\varepsilon$ goes to $\infty$) and the latter two elements converge in $L^2$ mas near $\phi=0$.

The local domains can be studied similarly, and we consider them at $\phi=0$ where the Morse function is expanding.
The complex $\mathcal{P}_{W,N,\varepsilon}(d_\varepsilon)(U_{\phi=0})$ has cohomology spanned by $e^{-\varepsilon x^2}, e^{-\varepsilon x^2} \gamma$ at the tangent cone, and the complex $\mathcal{P}_{W^{\perp},N,\varepsilon}(\delta_\varepsilon)(U_{\phi=0})$ has cohomology spanned by the volume form on $(U_{\phi=0})$ multiplied by $e^{-\varepsilon x^2}$ and the form $dx \wedge \gamma' e^{-\varepsilon x^2}$ where $x$ is the linearization of $\phi$ at $\phi=0$, which is the radial function on the tangent cone at the critical point at $\phi=0$. It is easy to check that these last two sections are good approximations of the spanning set of the global cohomology that we described above, that converge in $L^2$ mass near $\phi=0$. 
\end{example}

\begin{remark}   
\label{Remark_Crucial_domain_deformed_arg}
It is easy to check that $e^{-\varepsilon_2 \cos(\phi)} d\phi \wedge d\theta_1, e^{-\varepsilon_2 \cos(\phi)} d\phi \wedge d\theta_2$ are in the null space of $d_{\varepsilon_1}$ for all values of $\varepsilon_1 \geq 0$ and $\varepsilon_2 \in \mathbb{R}$, and we emphasize this holds even for negative values of $\varepsilon_2$, and these elements are in the domain $\mathcal{D}_{\max,\varepsilon}(P_{\varepsilon_1})$.
The sections $d\theta_1$ and $d\theta_2$ are co-exact (their duals $d(\phi d\theta_1), d(\phi d\theta_2)$ under the Hodge star operator are exact). The sections $e^{+\varepsilon_2 \cos(\phi)} d\theta_1$ and $e^{+\varepsilon_2 \cos(\phi)} d\theta_2$ are in the null space of $\delta_{\varepsilon_1}$ for all values of $\varepsilon_1 \geq 0$ and $\varepsilon_2 \in \mathbb{R}$ and these elements are in the domain $\mathcal{D}_{\max,\varepsilon}(\delta_{\varepsilon_1})$.
\end{remark}

\subsection{Bochner identities, localization and model operators}
\label{subsection_Bochner_Schrodinger}

Consider the Laplacian $D_\varepsilon^2 = (P_\varepsilon+P_\varepsilon^*)^2 = \Delta_{\varepsilon}=P_{\varepsilon}P^*_{\varepsilon}+P^*_{\varepsilon}P_{\varepsilon} $. We show that we have a Bochner type formula
\begin{equation}    
\label{equation_Witten_deformed_Laplace_type_expansion}
\Delta_{\varepsilon}=\Delta+\varepsilon^2\Vert dh\Vert^2+\varepsilon K
\end{equation}
generalizing that of Proposition 4.6 of \cite{Zhanglectures} for Morse functions on smooth manifolds.
As in \eqref{horizontal_vertical_Laplace_operators}, the undeformed operator on a fundamental neighbourhood $U(F_a)$ with a model metric can be written as $\Delta=\Delta^{H}+\Delta^{V}$. 
Since $h$ is constant on $F_a$, we see that $|dh|^2$ is identically $0$ on $F_a$, which is where the global harmonic sections on $X$ concentrate as we show in Section \ref{section_technical}. Proposition \ref{Propostion_growth_estimate_witten_deformed} is key for this argument and is a straightforward corollary of the following Proposition.

\begin{proposition}[Bochner type formula]
\label{Proposition_stratified_Morse_Laplace_structure}
Let $\widehat{X}$ be a stratified pseudomanifold with a wedge metric and a stratified Morse-Bott function $h$, and let $\mathcal{P}_W(X)$ be a de Rham complex. For any $\varepsilon \in \mathbb{R}$, given $s, s' \in \mathcal{D}_W(D_{\varepsilon}^2)$, we have
\begin{equation}
\label{Zhangs_Bochner_formula_1}
    \langle D_\varepsilon s, D_\varepsilon s' \rangle_{L^2(X;F)}=\langle s, [D^2+\varepsilon K  +\varepsilon^2 |dh|^2] s' \rangle_{L^2(X;F)}.
\end{equation} 
where $F=\Lambda^*(\prescript{w}{}{T^*X})$ and $K=(D \widehat{cl}(dh) + \widehat{cl}(dh) D)$ is a $0$-th order operator.
\end{proposition}

\begin{proof}
We begin by expanding the left hand side of the expression to get 
\begin{multline}
\langle Ds, Ds' \rangle_{L^2(X;F)} + \varepsilon \langle Ds, \widehat{cl}(dh)s' \rangle_{L^2(X;F)}+\varepsilon \langle  \widehat{cl}(dh)s, Ds' \rangle_{L^2(X;F)}+ \varepsilon^2 \langle \widehat{cl}(dh)s, \widehat{cl}(dh)s' \rangle_{L^2(X;F)}\\
= \langle s, [D^2+\varepsilon^2 |dh|^2]s' \rangle_{L^2(X;F)}  +\int_{\partial X} \langle i cl(d\rho_X)s,D s' \rangle_F \operatorname{dvol}_{\partial X}\\+\varepsilon [\langle Ds, \widehat{cl}(dh)s' \rangle + \langle \widehat{cl}(dh)s, Ds' \rangle_{L^2(X;F)} ],
\end{multline}
and we can expand the terms corresponding to the factor $\varepsilon^1$ as
\begin{multline}
\label{equation_with_boundary_term_Dirac_deformed}
    \langle Ds, \widehat{cl}(dh)s' \rangle_{L^2(X;F)}+ \langle \widehat{cl}(dh)s, Ds' \rangle_{L^2(X;F)} \\
    =\langle s, [D \widehat{cl}(dh) + \widehat{cl}(dh) D]s' \rangle_{L^2(X;F)}+ \int_{\partial X} \langle i cl(d\rho_X)s,\widehat{cl}(dh)s' \rangle_F \operatorname{dvol}_{\partial X}
\end{multline}
up to the factor of $\varepsilon$.
Thus the boundary contribution on $\partial X$ is 
\begin{equation}
    \int_{\partial X} \langle i cl(d\rho_X)s, (D + \varepsilon \widehat{cl}(dh))s' \rangle_F \operatorname{dvol}_{\partial X}=    \int_{\partial X} \langle i \sigma_1(D_\varepsilon)(d\rho_X) s, s'' \rangle_F \operatorname{dvol}_{\partial X}=0
\end{equation}
where $s''=D_\varepsilon s' \in \mathcal{D}_W(D)$ for $s' \in \mathcal{D}_W(D_\varepsilon^2)$ and the boundary integral vanishes any sections $s,s''$ in any self-adjoint domain $\mathcal{D}(D_\varepsilon)$.

It is easy to verify that $K$ is a zeroth order operator and we refer to \cite{Zhanglectures} (c.f., \cite{witten1982supersymmetry}) for details in the smooth setting ($X^{reg}$) in the case of isolated fixed points, and it is similar in the non-isolated case.
\end{proof}

Now we can prove the following estimate away from the critical points of a flat stratified Morse-Bott function. 

\begin{proposition}
\label{Propostion_growth_estimate_witten_deformed}
In the same setting as Proposition \ref{Proposition_stratified_Morse_Laplace_structure}, there exist constants $C>0$, $\varepsilon_0>0$ such that for any section $s \in \mathcal{D}_W(D_\varepsilon)$ with $\text{supp}(s) \subset (X \setminus \bigcup_{a \in \mathcal{I}} U({F_a}))$ and $\varepsilon \geq \varepsilon_0$, one has 
\begin{equation}
    \Vert D_{\varepsilon}s\Vert_{L^2(X;F)} \geq (C \sqrt{\varepsilon}) \Vert s\Vert_{L^2(X;F)}.
\end{equation}
\end{proposition}

\begin{proof}
This is similar to Proposition 4.7 of \cite{Zhanglectures} (c.f. Proposition 6.24 of \cite{jayasinghe2023l2}).
Let $C_1$ be the minimum value of $\Vert dh\Vert$ on $X \setminus \bigcup_{a \in \mathcal{I}} U({F_a})$.
Using the Bochner type formula \eqref{equation_Witten_deformed_Laplace_type_expansion},  it is easy to see that there exists a finite constant $K$ such that
\begin{equation}
\Vert D_{\varepsilon} s\Vert_{L^2(X;F)}^2  = \langle D_{\varepsilon} s, D_{\varepsilon} s \rangle_{L^2(X;F)} \geq (\varepsilon^2 C_1^2-\varepsilon |K|)\Vert s\Vert_{L^2(X;F)}
\end{equation}
since $\langle D s, D s \rangle_{L^2(X;F)}$ is positive.
Finally as the inequality holds for all $s \in \mathcal{D}_{W}(D_\varepsilon^2)$ it holds by density on the closure of $\mathcal{D}_{W}(D_\varepsilon^2)$ in the graph norm of $\mathcal{D}_{W}(D_\varepsilon)$.
\end{proof}

The Laplace-type operator on a fiber $U_a$ of $U(F_a)$, the (resolved) fundamental neighborhood of a critical point set $\widehat{F_a}$ endowed with a product type metric, determines a model operator on the tangent cone (a model harmonic oscillator) for which the local cohomology space can be understood well enough to construct an ansatz for the global and local harmonic sections in Section \ref{section_technical}. The elements in the null space of the model operator on the tangent cone, restrict to elements on the truncated tangent cone satisfying the boundary conditions corresponding to the choices of complexes on $U_s$ and $U_u$ described above. For the de Rham complex on the infinite cone with the conic metric (the tangent cone at isolated conic singularities), the null space of the model operator has been computed explicitly in \cite{ludwig2017index}, and was further worked on in \cite{jayasinghe2023l2}. We can generalize this to the flat stratified Morse-Bott setting where the model neighbourhood can be described as in Definition \ref{definition_resolved_fundamental_neighbourhood}.

\begin{definition}
\label{definition_tangent_cone_truncations}
Recall the notation for the resolved fundamental neighborhood in Definition \ref{definition_resolved_fundamental_neighbourhood}.
For any $z > 0$, we define $U_{a,s,z}:=\{x_a \leq z\} \subseteq U_{a,s}$, $U_{a,u,z}:=\{r_a \leq z\} \subseteq U_{a,u}$, and $U_{a,z}:=U_{a,s,z} \times U_{a,u,z}$.
We can extend the fibration $U(F_a)$ to $U_z(F_a)$ by extending the fibers from $U_{a}$ to $U_{a,z}$ (note that $U(F_a)=U_{z=1}(F_a)$).
Given a complex $\mathcal{P}_{W,B}(U(F_a))$ (respectively $\mathcal{R}_{W,B}(U(F_a))$), we define the complex $\mathcal{P}_{W,B,\varepsilon}(U_z(F_a))$ (respectively $\mathcal{R}_{W,B,\varepsilon}(U_z(F_a))$) similarly to how we defined it for the case where $z=1$.
We denote the non-compact resolved cone by $U_{a,\infty}$, and the corresponding fibration by $U_\infty(F_a)$, which we call the \textbf{infinite fundamental neighbourhood}.
\end{definition}

We observe that while the local complexes $\mathcal{P}_{W,B,\varepsilon}(U_z(F_a))$, $\mathcal{Q}_{W^{\perp},B^{\perp},\varepsilon}(U_z(F_a))$ and  the normal cohomology groups $\mathcal{R}_{W,B,\varepsilon}(U_z(F_a))$ are only defined for $z \in (0,\infty)$ for $\varepsilon \geq 0$, when we restrict to $\varepsilon>0$ they are defined for $z=\infty$ as well, and the cohomology groups of all these complexes are isomorphic as discussed above. In the case of $z=\infty$, the key is that factors of $e^{-\varepsilon (x_a^2+r_a^2)}$ appear in the local cohomology, giving sufficient decay at the infinite end of the cones for the harmonic forms to be $L^2$ bounded.

In subsection \ref{subsection_local_cohomology_complexes}, we introduced horizontal and vertical de Rham operators $d_E^H, d_E^V$ on fundamental neighbourhoods of critical point sets $U(F_a)$ as well as the horizontal and vertical Laplace-type operators $\Delta^H, \Delta^V$, and the latter two commute. Moreover since the derivatives of the vectors $\partial_x$ occur only in the vertical directions, the multiplier function defined by multiplication by the Stratified Morse-Bott function commutes with the horizontal Laplacian on the fundamental neighbourhood with the model metric.
Thus we one can write the Witten deformed Laplacian as
\begin{equation}
\label{equation_model_Witten_deformed_Laplacian}
    \Delta_\varepsilon:= \Delta^H+ \Delta^V_{\varepsilon}
\end{equation}
where $\Delta^V_{\varepsilon}$ is the Witten deformed Laplacian on the normal fibers over the regular part $F_a^{reg}$.

\begin{proposition}[Model equations and spectral gap]
\label{Proposition_model_spectral_gap_modified_general}
In the setting of this subsection, consider the complex $\mathcal{P}_{W, B, \varepsilon}(U_{z/\sqrt{\varepsilon}}(F_a))$ for some fixed $z \in [0,\infty]$. Let $\Delta^V_{\varepsilon}=(D^V_\varepsilon)^2$ be the Laplace-type operator on the vertical fibers over the regular part of $F_a$ for $\varepsilon>0$. 
Then
\begin{equation}
    \text{Spec}(\Delta_\varepsilon) \subseteq \{0\} \cup [C_1, \infty) \cup [C_2\varepsilon, \infty).
\end{equation}
\end{proposition}

In the case of $z=\infty$, we note that $\mathcal{P}_{W, B, \varepsilon}(U_{z/\sqrt{\varepsilon}}(F_a))=\mathcal{P}_{W, B, \varepsilon}(U_{z}(F_a))$.

\begin{proof}
Since the Witten deformed Laplacian splits as in equation \eqref{equation_model_Witten_deformed_Laplacian}, we can proceed by a separation of variables argument, and see that the eigensections of the Laplacian of the complex $\mathcal{P}_{W, B, \varepsilon}(U_{z}(F_a))$ can be described as the products of the eigensections of $\Delta^H$ and $\Delta^V_{\varepsilon}$.
Thus the eigensections of the operator $\Delta_\varepsilon$ split into three.
The first are the zero eigenvalue eigensections which correspond to the harmonic representatives of the local cohomology of the complex $\mathcal{P}_{W, B, \varepsilon}(U_{z}(F_a))$ (alternately $\mathcal{R}_{W, B, \varepsilon}(U_{z}(F_a))$), which we call the \textbf{small eigenvalue eigensections}.
The second are non-zero eigenvalue eigensections of the complex $\mathcal{R}_{W, B, \varepsilon}(U_{z}(F_a))$, which we will call the \textbf{middle eigenvalue eigensections}. Since the Laplace-type operator splits, and coefficient bundle over $F_a$ for the normal cohomology complex corresponds to sections which are in the null space of $\Delta^V_{\varepsilon}$, these eigenvalues will not change as $\varepsilon$  changes. The corresponding eigenvalues can be bounded by below by $C_1$. The third set of eigensections correspond to the eigensections in 
$\mathcal{P}_{W, B, \varepsilon}(U_{z}(F_a))$ which are the complement of the small and medium eigensections, which we will call the \textbf{large eigensections}, and the eigenvalues of which are in $[C_2\varepsilon, \infty)$.

To see this, we observe that the eigenvalues of the operator $\Delta^V_\varepsilon$ on the normal fibers over the regular part split into small and large eigenvalue eigensections.
In the case of Morse-Bott functions, on spaces satifying the Witt condition, this was proven in Proposition 6.26 of \cite{jayasinghe2023l2}. For flat stratified Morse-Bott functions, in the non-Witt case, the spectrum of the deformed Laplace-type operators can still be computed explicitly using separation of variables by cones similar to \cite{jayasinghe2023l2} and since we pick a suitably scaled metric only zero eigenvalue eigensections can appear for any ideal boundary condition.

Then using the separation of variables ansatz and using the fact that there is an orthonormal countable basis of eigensections of $\Delta_\varepsilon$ for both the complexes $\mathcal{R}_{W, B, \varepsilon}(U_{z}(F_a))$ and $\mathcal{P}_{W, B, \varepsilon}(U_{z}(F_a))$, we see that the non-zero eigenvalues split into the medium and large eigenvalues as described above.
\end{proof}

\section{Proof of main results}
\label{section_technical}

\subsection{The Witten instanton complex}

Given $z \in (0, \infty]$, choose $\gamma_{a,z}: \mathbb{R}_s \rightarrow[0,z]$ to be a function in $C^{\infty}(\mathbb{R})$ such that $\gamma_a(s)=1$ when $s<z/2$, and $\gamma_{a,z}(s)=0$ when $s>3z/4$. Recall that our critical point set $\text{crit}(h)$ has a decomposition into connected components,
\[ \text{crit}(h) := \bigcup_{a\in \mathcal{I}} \widehat{F_a},     \]
each with corresponding singular fundamental neighbourhoods, $U(\widehat{F_a})$, and their resolutions $U(F_a)$ (see Definition \ref{definition_resolved_fundamental_neighbourhood}). At each such resolved fundamental neighbourhood $U(F_a)$, we define $t_a=x_a^2+r_a^2$ where the functions $x_a,r_a$ are the functions in Definition \ref{definition_resolved_fundamental_neighbourhood}, and we drop the subscripts for fixed critical point sets $\widehat{F_a}$. 
Given a harmonic form $\omega_{a} \in \mathcal{H}(\mathcal{P}_{W, B}(U_z({F_a})))$, we define $\omega_{a,\varepsilon}:=\omega_a e^{-t \varepsilon} \in \mathcal{H}(\mathcal{P}_{W, B, \varepsilon}(U_z({F_{a}})))$ and 
\begin{equation}
\label{equation_modify_forms_cutoff}
    \alpha_{a, z, \varepsilon} :=\left\|\gamma_a(t) \omega_{a,\varepsilon}\right\|_{L^{2}(U({F_{a,z}}))}, \hspace{5mm} \eta_{a, z, \varepsilon}:=\frac{\gamma_{a,z}(t) \omega_{a,\varepsilon}}{\alpha_{a, z,  \varepsilon}},
\end{equation}
and we define 
\begin{equation}
\label{definition_perturbed_basis}
    \mathcal{W}(\mathcal{P}_{W, B, \varepsilon}(U_z({F_{a}}))):=\Bigg\{ \eta_{a, z, \varepsilon}=\frac{\gamma_{a,z}(t) \omega_{a,\varepsilon}}{\alpha_{a, z,  \varepsilon}} : \omega_{a,\varepsilon} \in \mathcal{H}(\mathcal{P}_{W, B, \varepsilon}(U_z({F_{a}}))) \Bigg\}
\end{equation} 
where the forms $\eta_{a, z, \varepsilon}$ each have unit $L^2$ norm and are supported on $U({F_{a,z}})$ (see Definition \ref{definition_tangent_cone_truncations}). In the case where $z=1$, we drop the subscript $z=1$ and denote the  corresponding forms $\alpha_{a, \varepsilon}, \eta_{a, \varepsilon}$, and the cutoff function $\gamma_a$. 
We can extend the forms from a small fundamental neighbourhood $U(F_a)$ to $X$ by $0$ away from their supports.
We have the Witten deformed Dirac type operator $D_{\varepsilon}=P_{\varepsilon}+P^*_{\varepsilon}$, whose square is the Witten deformed Laplace type operator $\Delta_{\varepsilon}$.

Let $E_{\varepsilon}$ be the vector space generated by the set 
\[ E_\varepsilon:=\text{span} \bigg\{\mathcal{W}(\mathcal{P}_{W, B, \varepsilon}(U(F_{a}))) : a \in \mathcal{I} \bigg\} \]
corresponding to all the harmonic sections $\omega_{a,\varepsilon}$ as above; $E_{\varepsilon}$ is a subspace of $\mathcal{D}_{W}(D_\varepsilon)$ since each $\eta_{a, \varepsilon}$ has finite length and compact support. Since $E_{\varepsilon}$ is finite dimensional (in particular closed), there exists an orthogonal splitting 
\begin{equation}
    \mathcal{D}_{W}(D_\varepsilon)=E_{\varepsilon} \oplus E_{\varepsilon}^{\perp}.
\end{equation}
where $E_{\varepsilon}^{\perp}$ is the orthogonal complement of $E_{\varepsilon}$ in $\mathcal{D}_W(D_\varepsilon) \subseteq L^2\Omega(X;E)$ with respect to the norm of $L^2\Omega(X;E)$. Denote by $\Pi_{\varepsilon}, \Pi_{\varepsilon}^{\perp}$ the orthogonal projection maps from $L^2\Omega(X;E)$ to $E_{\varepsilon}, E_{\varepsilon}^{\perp}$, respectively. From here on we shall drop the notation of $W$ denoting the choice of domain, as all proceeding statements and estimates will hold independent of any choice of domain for $D_\varepsilon$.
\begin{remark}[Convention]
\label{Remark_drop_notation_W_convention_1}
    In the proof of the main result in this subsection, we will fix a choice of global mezzo perversity and Morse-Bott function, and since the choices of domains are fixed we will drop the $W$ in the notation in this subsection.
\end{remark}

We split the deformed Witten operator by the projections as follows.
\begin{equation}
\label{equation_fourfold_operators}
D_{\varepsilon, 1} =\Pi_{\varepsilon} D_{\varepsilon} \Pi_{\varepsilon}, \hspace{3mm}
D_{\varepsilon, 2} =\Pi_{\varepsilon} D_{\varepsilon} \Pi_{\varepsilon}^{\perp}, \hspace{3mm} D_{\varepsilon, 3} =\Pi_{\varepsilon}^{\perp} D_{\varepsilon} \Pi_{\varepsilon}, \hspace{3mm} D_{\varepsilon, 4} =\Pi_{\varepsilon}^{\perp} D_{\varepsilon} \Pi_{\varepsilon}^{\perp}.
\end{equation}
In the following proposition and its proof, the norms and inner products are those for the $L^2$ forms (unless otherwise specified) corresponding to the Hilbert space of the complex $\mathcal{P}_{\varepsilon}(X)$, and the inner products are the same for all $\varepsilon \geq 0$. We observe that quantities such as $\eta_{a,\varepsilon}$ and $\gamma_a(t)$ are only supported in the fundamental neighbourhoods corresponding to the critical points $a$, where the $L^2$ inner products for the local complexes $\mathcal{P}_{B,\varepsilon}(U_z(F_{a}))$ match the global inner product. In the proof, we will only need the complexes for $z=1$ and $z=\infty$, the latter being used only in the proof of the third statement of the proposition.

\begin{proposition}
\label{proposition_Zhangs_Morse_inequalities_intermediate_estimates}
In the setting described above, we have the following estimates.
\begin{enumerate}
    \item There exist a constant $\varepsilon_{0}>0$ such that for any $\varepsilon>\varepsilon_{0}$ and for any $s \in \mathcal{D}(D_{\varepsilon})$, $$\Vert D_{\varepsilon, 1}s\Vert  \leq \frac{\|s\|}{2\varepsilon}.$$ 
    \item There exists a constant $\varepsilon_{1}>0$ such that for any $s \in E_{\varepsilon}^{\perp} \cap \mathcal{D}(D_{\varepsilon}), s^{\prime} \in E_{\varepsilon}$, and $\varepsilon>\varepsilon_{1}$,
$$
\begin{aligned}
\left\|D_{\varepsilon, 2} s\right\| & \leq \frac{\|s\|}{2\varepsilon} \\
\left\|D_{\varepsilon, 3} s^{\prime}\right\| & \leq \frac{\left\|s^{\prime}\right\|}{2\varepsilon}
\end{aligned}
$$
    
    \item  There exist constants $\varepsilon_{2}>0$ and $C_0>0$ such that for any $s \in E_{\varepsilon}^{\perp} \cap \mathcal{D}(D_{\varepsilon})$ and $\varepsilon>\varepsilon_{2}$,
$$
\left\|D_{\varepsilon} s\right\| \geq C_0 \|s\|
$$
\end{enumerate}
\end{proposition}

\begin{remark}
We remark that the first Sobolev space appearing in the statement of Proposition 5.6 of \cite{Zhanglectures} in the smooth setting is replaced by the domain $\mathcal{D}(D_\varepsilon)$ in this generalization to our singular setting. We observe that on non-Witt spaces different choices of domains for $P_\varepsilon$ have different local cohomology groups, and the estimates hold only for sections in the domain $\mathcal{D}_W(D_\varepsilon)$ notated as $\mathcal{D}(D_\varepsilon)$ as per our convention (see Remark \ref{Remark_drop_notation_W_convention_1}).
\end{remark}

\begin{proof}
\textbf{\textit{Outline:}} We adapt the proof of Proposition 5.6 of \cite{Zhanglectures}. The notation is slightly different (in particular the Witten deformation parameter is taken to be $T$ in that article) but it is easy to follow that proof, which is structured after that in \cite{wu1998equivariant}, and variants of which were given in \cite{jayasinghe2023l2,jayasinghe2024holomorphic}.
The same proofs in the smooth case more or less go through after replacing the model solutions in the smooth proofs with the forms in $\mathcal{W}(\mathcal{P}_{B, \varepsilon}(U(F_{a})))$. We can pick an orthonormal basis $\widehat{W_{\varepsilon,a}}$ for the vector space generated by those forms. 

\textbf{\textit{proof of 1:}} For any $s \in \mathcal{D}(D_\varepsilon)$ the projection $\Pi_{\varepsilon} s$ can be written
\begin{equation}
\label{Projection_first_witten_deform}    
\Pi_{\varepsilon} s=\sum_{a \in \mathcal{I}} \sum_{\eta \in \widehat{W_{\varepsilon,a}}} \left\langle \eta, s\right\rangle_{L^2} \eta
\end{equation}
where each $\eta$ can be written as a linear combination of forms $\eta_{a,\varepsilon}$ as defined in equation \eqref{definition_perturbed_basis}. We have that
\begin{equation}
    \Vert \Pi_{\varepsilon} D_\varepsilon \eta_{a, \varepsilon}\Vert  \leq e^{-C_0 \varepsilon}
\end{equation}
for some $C_0>0$ and $\varepsilon$ large enough using the following argument.

Since $D_{\varepsilon}=D+ \varepsilon \widehat{cl}(dh)$, observe that for a smooth function $v$ we have by the Leibniz rule
\begin{equation}
\label{equation_Leibniz_for_Witten_deformation}
    D_{\varepsilon}( v \omega_{a,\varepsilon}) = D(v \omega_{a,\varepsilon})+ \varepsilon \widehat{cl}(dh)( v \omega_{a,\varepsilon}) = cl(dv) \omega_{a,\varepsilon} + v \big ( (D \omega_{a,\varepsilon}) + \varepsilon \widehat{cl}(dh)(\omega_{a,\varepsilon}) \big ) = cl(dv) \omega_{a,\varepsilon}
\end{equation}
since $D_{\varepsilon} \omega_{a,\varepsilon}=0$. Since $\text{supp}( d\gamma_a(t)) \subseteq \{1/2 \leq t=x^2+r^2 \leq 3/4\}$ (recall that $\gamma_a=\gamma_{a,z=1}$), 
each $D_{\varepsilon}\eta_{a,\varepsilon}$ is compactly supported in $\{1/2 \leq t=x^2+r^2 \leq 3/4\}$. Then for $\varepsilon$ large enough
\begin{equation} 
\label{inequality_useful_for_Morse_proof_234}
\left\langle D_{\varepsilon} \eta_{a, \varepsilon}, \eta_{a, \varepsilon} \right\rangle_{L^2} =  \left\langle cl(d\gamma_a(t)) \omega_{a, \varepsilon}/ {\alpha_{a,\varepsilon}}, \eta_{a, \varepsilon} \right\rangle_{L^2} \leq e^{-C_0 \varepsilon}
\end{equation}
for some large enough positive constant $C_0$ 
since $cl(d\gamma_a(t)) \omega_{a, \varepsilon}/ {\alpha_{a,\varepsilon}}$ is supported away from $t \leq 1/2$ and $\omega_{a, \varepsilon}=\omega_a e^{-t\varepsilon}$.

Since each $\eta$ has unit norm, the Cauchy Schwartz inequality shows that $|\left\langle \eta, s\right\rangle_{L^2}| \leq \Vert s\Vert $.
Since the forms $\eta$ in the basis $\widehat{W_{\varepsilon,a}}$ for a given critical point set $\widehat{F_a}$ are orthogonal,
and since the supports of the forms in $\widehat{W_{\varepsilon,a}}$ for distinct connected components of $\text{crit}(h)$ have no intersection, using equation \eqref{Projection_first_witten_deform} we see that
\begin{equation}
    \Vert \Pi_{\varepsilon} D_\varepsilon \Pi_{\varepsilon}s\Vert  \leq e^{-C_0 \varepsilon}\Vert s\Vert 
\end{equation}
for large enough $\varepsilon$ in order to compensate for the finite sum of terms as well as ensuring \eqref{inequality_useful_for_Morse_proof_234}. The estimate of the Proposition statement follows since the exponential decays faster than the required decay.

\textbf{\textit{proof of 2:}} In the smooth setting, this follows from Proposition 4.11 of \cite{Zhanglectures} which we will redo in this case, changing the details where necessary.
Since $D_\varepsilon$ is self-adjoint, it is easy to see that $D_{\varepsilon,2}$ is the adjoint of $D_{\varepsilon,3}$, and it suffices to prove the first estimate of the two.

Since each $\eta_{a, \varepsilon}$ has support in $U(F_a)$, one deduces that for any $s \in E_{\varepsilon}^{\perp} \cap \mathcal{D}(D_{\varepsilon})$,

$$
\begin{aligned}
D_{\varepsilon, 2} s & =\Pi_{\varepsilon} D_{\varepsilon} \Pi_{\varepsilon}^{\perp} s=\Pi_{\varepsilon} D_{\varepsilon} s \\
& =\sum_{a \in \mathcal{I}} \sum_{\eta \in \widehat{W_{\varepsilon,a}}} \left\langle\eta, D_{\varepsilon} s\right\rangle_{L^2} \eta \\
& =\sum_{a \in \mathcal{I}} \sum_{\eta \in \widehat{W_{\varepsilon,a}}} \eta \int_{U_a} \left\langle\eta, D_{\varepsilon} s\right\rangle_{\Lambda^\bullet} \operatorname{dvol}_{U_a} \\
& =\sum_{a \in \mathcal{I}} \sum_{ \eta \in \widehat{W_{\varepsilon,a}}
} \eta \int_{U_a} \left\langle D_{\varepsilon} \eta,  s\right\rangle_{\Lambda^\bullet}  \operatorname{dvol}_{U_a}
\end{aligned}
$$
where we have denoted the inner product on the wedge exterior bundle by the angle brackets with subscript $\Lambda^{\bullet}$. Since each $\eta$ can be written as a linear combination of forms $\eta_{a,\varepsilon}$ as defined in equation \eqref{definition_perturbed_basis}, it suffices to estimate the integral over $U_a$ with integrand $\left\langle D_{\varepsilon} \eta_{a,\varepsilon},  s\right\rangle_{\Lambda^{\bullet}}$. This can be expanded as
\begin{equation}
    \left\langle D_{\varepsilon} \frac{\gamma_a(t) \omega_{a} e^{-(x^2+r^2) \varepsilon}}{{\alpha_{a, \varepsilon}}},  s\right\rangle_{L^2} =\left\langle \frac{cl(d\gamma_a(t)) \omega_{a} e^{-(x^2+r^2) \varepsilon}}{{\alpha_{a, \varepsilon}}},  s\right\rangle_{L^2} 
\end{equation}
restricted to each $U_a$, where we have used the argument in \eqref{equation_Leibniz_for_Witten_deformation}, similar to equation (4.40) of \cite{Zhanglectures}. We know that $d\gamma_a$ is only supported on the set $\{1/2 \leq t=(x^2+r^2) \leq 3/4\}$ in each $U(F_a)$, we see that the desired inequality follows from 
Cauchy-Schwarz.

\textbf{\textit{proof of 3:}} 
In the case of stratified Morse functions in the essentially self-adjoint case, the third statement in Proposition 5.6 of in \cite{Zhanglectures} (where the proof is that of Proposition 4.12 of that article) was generalized in Proposition 6.28 of \cite{jayasinghe2023l2}. In the flat stratified Morse-Bott case we need important modifications.

Let $\widetilde{\gamma} \in C_{\Phi}^{\infty}(X)$ be defined such that restricted to each $U(F_{a})$ for critical points $a$, $\widetilde{\gamma}(t)=\gamma_a(t)$, and that $\widetilde{\gamma} \big|_{X \backslash \bigcup_{a \in \mathcal{I}} U(F_{a})}=0$.
For any  $s \in E_{\varepsilon}^{\perp} \cap \mathcal{D}(D_{\varepsilon})$ we see that $\widetilde{\gamma} s \in E_{\varepsilon}^{\perp} \cap \mathcal{D}(D_{\varepsilon})$.
We will show that 
\begin{equation}
\label{equation_to_prove_third_estimate}
    \left\|D_{\varepsilon} s\right\| \geq \frac{1}{2} \big( \left\|(1-\widetilde{\gamma}) D_{\varepsilon} s\right\|+\left\|\widetilde{\gamma} D_{\varepsilon} s\right\| \big) > C_0 \Vert s\Vert 
\end{equation}
and we split the last estimate into two parts.

\textbf{\textit{part 1:}} 
We first estimate the term $\Vert (1-\widetilde{\gamma}) D_{\varepsilon} s\Vert $ from below using global arguments.
Observe that 
$$
\begin{gathered}
\Vert (1-\widetilde{\gamma}) D_{\varepsilon} s\Vert =\Vert D_{\varepsilon} (1-\widetilde{\gamma})s + [D, \widetilde{\gamma}] s\Vert  \\ \geq \frac{\sqrt{\varepsilon}}{2}C_{8}\|(1-\widetilde{\gamma}) s\|-C_{9}\|s\|, 
\end{gathered}
$$
where we use the fact that $\widetilde{\gamma}$ is supported away from the critical point set which allows us to use the growth estimate in Proposition \ref{Propostion_growth_estimate_witten_deformed}, and also the fact that $[D, \widetilde{\gamma}]=cl(d \widetilde{\gamma})$ is a zeroth order operator bounded by a constant $C_9$. Then for large enough $\varepsilon$ we see that we have a much better lower bound that what we need for proving \ref{equation_to_prove_third_estimate} given a suitable bound for the other term there.

\textbf{\textit{part 2:}} 
Now we consider the term $\left\|\widetilde{\gamma} D_{\varepsilon} s\right\|$ and estimate it using results about the eigenvalues on exact cones.
Observe that since $\gamma_{a,\infty}$ is identically $1$ on the non-compact `neighbourhood' $U_{\infty}(F_a)$, the forms in $\mathcal{W}(\mathcal{P}_{B,\varepsilon}(U_\infty(F_{a})))$ are elements of $\mathcal{H}(\mathcal{P}_{B,\varepsilon}(U_\infty(F_{a})))$ with unit norm.
For any section $s$ verifying $\operatorname{supp}(s) \in \bigcup_{a \in \mathcal{I}} U(F_{a})$, the projection $\Pi'_{\varepsilon} s$ is defined by
\begin{equation}
\label{Projection_first_witten_deform_non_compact}    
\Pi'_{\varepsilon} s:=\sum_{a \in \mathcal{I}} \sum_{\omega_{a,\varepsilon} \in \widehat{W'_{\varepsilon,a}}} \left\langle \omega_{a,\varepsilon}, s\right\rangle_{L^2(U_\infty(F_{a}))} \omega_{a,\varepsilon}
\end{equation}
where $\widehat{W'_{\varepsilon,a}}$ is an orthonormal basis for the vector space $\mathcal{H}(\mathcal{P}_{B,\varepsilon}(U_\infty(F_{a})))$.
Then $\Pi'_{\varepsilon}$ is an orthogonal projection
\begin{equation}
    \Pi'_{\varepsilon} :\bigoplus_{a \in \mathcal{I}} L^2\Omega^\cdot(U_\infty(F_{a});E) \rightarrow \mathcal{H}(\mathcal{P}_{B,\varepsilon}(U_\infty(F_{a}))).
\end{equation}

Note that if $\omega_{a,\varepsilon} \in \mathcal{H}(\mathcal{P}_{B,\varepsilon}(U_\infty(F_{a})))$, then $\gamma_a(t) \omega_{a,\varepsilon} \in E_{\varepsilon}$ and so if $s \in E^{\perp}_{\varepsilon}$ we have
\begin{equation}
\label{equation_for_projection_modified_21}
\Pi'_{\varepsilon} s=\sum_{a \in \mathcal{I}} \sum_{\omega_{a,\varepsilon} \in \widehat{W'_{\varepsilon,a}}}
\omega_{a,\varepsilon} 
\left\langle(1-\gamma_a(t)) \omega_{a,\varepsilon}, s \right\rangle_{L^2\Omega^\cdot(U_\infty(F_{a});E)}.
\end{equation}
Using an argument similar to that in the first point of the proposition, we can show that
\begin{equation}
    \Vert (1-\gamma_a(t))\omega_{a,\varepsilon}\Vert ^2_{L^2\Omega(U(F_{a,\infty});E)} \leq \frac{C}{\varepsilon}
\end{equation}
for some constant $C$ since $\gamma_a$ equals to 1 near each $a$. Then equation \eqref{equation_for_projection_modified_21} shows that there exists $C_{5}>0$ such that when $\varepsilon \geq 1$,
\begin{equation}
\label{equation_4.46_of_Zhang}
\left\|\Pi'_{\varepsilon} s \right\|^{2} \leq \frac{C_{5}}{{\varepsilon}}\|s\|^{2}.
\end{equation}

Since $\widetilde{\gamma}(s-\Pi'_{\varepsilon}s)$ is in $E_{\varepsilon}^{\perp}$, by equation \eqref{equation_4.46_of_Zhang}, we can now use the spectral results given by Proposition \ref{Proposition_model_spectral_gap_modified_general} to prove the necessary estimates for $\left\|\widetilde{\gamma} D_{\varepsilon} s\right\|$ for $s \in E_{\varepsilon}^{\perp} \cap \mathcal{D}(D_{\varepsilon})$ (note that $D_\varepsilon \Pi'_{\varepsilon} s=0$).
We first see that 
\begin{equation}
\label{equation_quantity_to_be_estimated_77}
    \Vert \widetilde{\gamma} D_\varepsilon s'\Vert =\Vert D_\varepsilon (\widetilde{\gamma} s') + [\widetilde{\gamma},D]s'\Vert 
\end{equation}
which holds for $s'=(s-\Pi'_{\varepsilon}s)$. We estimate the quantity in \eqref{equation_quantity_to_be_estimated_77} for the medium and the large eigenvalue eigensections given in Proposition \ref{Proposition_model_spectral_gap_modified_general} separately.

For the medium eigenvalue eigensections, we see that $D_\varepsilon (\widetilde{\gamma} s') \geq  C_0 \Vert s'\Vert $ for large enough $\varepsilon$ where we can take $C_0$ to be the minimum of all smallest non-zero eigenvalues of the horizontal ($\Delta^H$) Laplacians on each neighbourhood $U_{\infty}(F_a)$ over all connected components of the critical point set (which are finite by assumption). We also see that $[\widetilde{\gamma},D]=-cl(d\widetilde{\gamma})$, which at every $U(F_a)$ has support which is a bounded distance away from the critical point set and since the medium eigenvalue eigenforms decay exponentially away from the critical point set, we have an upper bound
\begin{equation}
    \|[\widetilde{\gamma},D]s^\prime\| =\| cl(d\widetilde{\gamma}) s^\prime\| < \frac{C_1}{\sqrt{\varepsilon}}\|s^\prime\|.
\end{equation}
Thus we see that for medium eigenvalue eigensections,
\begin{equation}
    \Vert \widetilde{\gamma} D_\varepsilon s'\Vert  \geq C_0\Vert s'\Vert  - \frac{C_1}{\sqrt{\varepsilon}}\Vert s'\Vert .
\end{equation}
We now treat the large eigenvalue eigensections $s$. In this case, again we have that $[\widetilde{\gamma},D]=-cl(d\widetilde{\gamma})$ and since $cl(d\widetilde{\gamma})$ is a bounded zero-th order operator, we have that $\Vert cl(d\widetilde{\gamma}) s'\Vert  \leq C_9 \Vert s'\Vert $ for some positive constant $C_9$ where $s'=(s-\Pi'_{\varepsilon}s)$.
Since these are large eigenvalue eigensections $\Vert D_\varepsilon s'\Vert  \geq C_{11} \sqrt{\varepsilon}\Vert  s'\Vert $, and we have 
\begin{equation}
    \Vert \widetilde{\gamma} D_\varepsilon s'\Vert  \geq C_{11} \sqrt{\varepsilon}\Vert  s'\Vert  - C_9 \Vert s'\Vert .
\end{equation}
Thus for large values of $\varepsilon$, we have the necessary estimate, completing the proof of the Proposition.
\end{proof}

For any $c>0$, denote by {$E_{\varepsilon,c}$} the direct sum of the eigenspaces of $D_{\varepsilon}$ with eigenvalues lying in $[-c, c]$, which is a finite dimensional subspace of $L^2\Omega(X;E)$. Let $\Pi_{\varepsilon,c}$ be the orthogonal projection from $L^2\Omega(X;E)$ to $E_{\varepsilon,c}$. The following is a generalization of Lemma 5.8 of \cite{Zhanglectures}, with modifications for the flat stratified Morse-Bott case.

\begin{lemma}
\label{Lemma_inequality_spectral_for_Witten_deformation}
There exist $C_1>0$, $\varepsilon_3>0$ such that for any $\varepsilon>\varepsilon_3$, and any $\sigma \in E_{\varepsilon}$, 
\begin{equation}
    \left\|\Pi_{\varepsilon,c} \sigma-\sigma\right\| \leq \frac{C_1}{\varepsilon}\|\sigma\|.
\end{equation}
\end{lemma}

\begin{proof}
Let $\delta=\{\lambda \in \mathbf{C}:|\lambda|=c\}$ be the counter-clockwise oriented circle. By Proposition \ref{proposition_Zhangs_Morse_inequalities_intermediate_estimates}, one deduces that for any $\lambda \in \delta, \varepsilon \geq \varepsilon_{0}+\varepsilon_{1}+\varepsilon_{2}$ and $s \in \mathcal{D_\varepsilon}$, there exists constants $C_0 > c_0 >0 $ such that
\begin{equation}   
\begin{gathered}
\left\|\left(\lambda-D_{\varepsilon}\right) s\right\| \geq \frac{1}{2}\left\|\lambda \Pi_{\varepsilon} s -D_{\varepsilon, 1}\Pi_{\varepsilon}s   -D_{\varepsilon, 2} \Pi_{\varepsilon} s\right\| +\frac{1}{2}\left\|\lambda \Pi_{\varepsilon}^{\perp} s-D_{\varepsilon, 3} \Pi_{\varepsilon}^{\perp} s-D_{\varepsilon, 4} \Pi_{\varepsilon}^{\perp} s\right\| \\
\geq \frac{1}{2}\left(\left(c_0-\frac{1}{\varepsilon}\right)\left\|\Pi_{\varepsilon} s\right\|+\left(C_0-c_0-\frac{1}{\varepsilon}\right)\left\|\Pi_{\varepsilon}^{\perp} s\right\|\right) 
\end{gathered}
\end{equation}
where we pick $c_0$ such that $c_0 < C_0$ where $C_0$ is the constant given in the third estimate of Proposition \ref{proposition_Zhangs_Morse_inequalities_intermediate_estimates}. 
This shows that there exist $\varepsilon_{4}>\varepsilon_{0}+\varepsilon_{1}+\varepsilon_{2}$ and $C_{2}>0$ such that for any $\varepsilon \geq \varepsilon_{4}$ and $s \in \mathcal{D_\varepsilon}$,
\begin{equation}  
\label{equation_(5.27)_Zhang}
\left\|\left(\lambda-D_{\varepsilon}\right) s\right\| \geq C_{2}\|s\|.
\end{equation}
Thus, for any $\varepsilon  \geq \varepsilon _{4}$ and $\lambda \in \delta$,
\begin{equation}  
\lambda-D_{\varepsilon }:  \mathcal{D_\varepsilon} \rightarrow L^2(X)
\end{equation}
is invertible and the resolvent $\left(\lambda-D_{\varepsilon}\right)^{-1}$ is well-defined.
By the spectral theorem 
one has
\begin{equation}
\label{equation_spectral_projector_countour_integral}
\Pi_{\varepsilon,c} \sigma-\sigma=\frac{1}{2 \pi \sqrt{-1}} \int_{\delta}\left(\left(\lambda-D_{\varepsilon}\right)^{-1}-\lambda^{-1}\right) \sigma d \lambda.
\end{equation}
Now one verifies directly
that for any $\sigma \in E_\varepsilon$
\begin{equation}
\left(\left(\lambda-D_{\varepsilon}\right)^{-1}-\lambda^{-1}\right) \sigma=\lambda^{-1}\left(\lambda-D_{\varepsilon}\right)^{-1} (D_{\varepsilon, 1}+D_{\varepsilon, 3}) \sigma .
\end{equation}
From Proposition \ref{proposition_Zhangs_Morse_inequalities_intermediate_estimates} and \eqref{equation_(5.27)_Zhang}, one then deduces that for any $\varepsilon \geq \varepsilon_{4}$ and $\sigma \in E_{\varepsilon}$,
\begin{equation}
\label{equation_estimate_finale}
\left\|\left(\lambda-D_{\varepsilon}\right)^{-1} (D_{\varepsilon, 1}+D_{\varepsilon, 3}) \sigma\right\| \leq C_{2}^{-1}\left\|(D_{\varepsilon, 1}+D_{\varepsilon, 3}) \sigma\right\| \leq \frac{1}{C_{2} \varepsilon}\|\sigma\|
\end{equation}
From \eqref{equation_spectral_projector_countour_integral}-\eqref{equation_estimate_finale}, we get the estimate in the statement of the Lemma, finishing the proof.
\end{proof}

In Remark 5.9 of \cite{Zhanglectures}, Zhang explains that one can work out an analog of the proof with real coefficients, whereas the proof above implicitly uses the fact that we work in the category of complex coefficients.
The following is a generalization of Proposition 5.5 of \cite{Zhanglectures}.

\begin{proposition}
\label{Proposition_small_eig_estimate_and_dimension}
There exists a small enough constant $c>0$ and there exists $\varepsilon_0>0$ such that when $\varepsilon>\varepsilon_0$, the number of eigenvalues in $[0,c]$ of $\Delta_{\varepsilon,k}$, the Laplace-type operator acting on forms of degree $k$ of the complex $\mathcal{P}_{\varepsilon}(X)$, is equal to $\sum_{a \in \mathcal{I}} \dim \mathcal{H}^k(\mathcal{P}_{B,\varepsilon}(U(F_{a})))$.
\end{proposition}

\begin{proof}
This follows from the same formal arguments as those in the proof of Proposition 5.5 of \cite{Zhanglectures}. Similar generalizations to the singular setting are given in Proposition 6.31 of \cite{jayasinghe2023l2} and Proposition 5.10 of \cite{jayasinghe2024holomorphic}.
\end{proof}

These results show that the small eigenvalue eigensections form a complex of the Witten deformed Laplacian encodes information about the critical point structures. This is also known as the Witten instanton complex.

\begin{theorem}[Witten instanton complex]
\label{theorem_small_eig_complex}
For any integer $0 \leq k \leq n$, let $
\mathrm{K}_{\varepsilon, k}^{[0, c]} \subset L^2\Omega^k(X;E)$ denote the vector space generated by the eigenspaces of $\Delta_{\varepsilon}$ associated with eigenvalues in $[0, c]$. There exists a small enough $c>0$ and there exists $\varepsilon_0>0$ such that when $\varepsilon>\varepsilon_0$ this has the same dimension as $\oplus_{a \in \mathcal{I}} \mathcal{H}^k(\mathcal{P}_{W,B,\varepsilon}(U(F_a))$, and together form a finite dimensional subcomplex of $\mathcal{P}_{\varepsilon}(X)$ :
\begin{equation}
    \label{small_eigenvalue_complex}
\left(\mathrm{K}_{\varepsilon, k}^{[0, c]}, P_{\varepsilon}\right): 0 \longrightarrow \mathrm{K}_{\varepsilon, 0}^{[0, c]} \stackrel{P_{\varepsilon}}{\longrightarrow} \mathrm{K}_{\varepsilon, 1}^{[0, c]} \stackrel{P_{\varepsilon}}{\longrightarrow} \cdots \stackrel{P_{\varepsilon}}{\longrightarrow} \mathrm{K}_{\varepsilon, n}^{[0, c]} \longrightarrow 0
\end{equation}
which is quasi-isomorphic to $\mathcal{P}_{\varepsilon}(X)$.
\end{theorem}

\begin{proof}    
This follows from Proposition \ref{Proposition_small_eig_estimate_and_dimension} once one shows that the small eigenvalue eigensections form a complex.
Since
$$
P_{\varepsilon} \Delta_{\varepsilon}=\Delta_{\varepsilon} P_{\varepsilon} =P_{\varepsilon} P_{\varepsilon}^* P_{\varepsilon} \text{   and    }
P_{\varepsilon}^* \Delta_{\varepsilon}=\Delta_{\varepsilon} P_{\varepsilon}^* =P_{\varepsilon}^* P_{\varepsilon} P_{\varepsilon}^*
$$
one sees that $P_{\varepsilon}$ (resp. $P_{\varepsilon}^*$ ) maps each $\mathrm{K}_{\varepsilon, k}^{[0, c]}$ to $\mathrm{K}_{\varepsilon, k+1}^{[0, c]}$ (resp. $\mathrm{K}_{\varepsilon, k-1}^{[0, c]}$ ). 
The Kodaira decomposition of $\mathcal{P}_{\varepsilon}(X)$ restricts to this finite dimensional complex $\left(\mathrm{K}_{\varepsilon, k}^{[0, c]}, P_{\varepsilon}\right)$. 
\end{proof}

The strong form of the Morse inequalities is a corollary which we explore in the next subsection.

\subsection{Morse inequalities}

We prove the following polynomial form of the Strong Morse inequalities.

\begin{theorem}
\label{Theorem_strong_Morse_Bott_de_Rham}
[Strong polynomial Morse inequalities for the de Rham complex]
Let $\widehat{X}$ be a stratified pseudomanifold of dimension $n$ with a wedge metric and a flat stratified Morse-Bott function $h$. Let $E$ be a flat vector bundle on $X$. Let $\mathcal{P}_W(X)=(L^2\Omega(X;E), \mathcal{D}_W(d_E),d_E)$ be the de Rham complex. 
Then there exist non-negative integers $Q_0,..., Q_{n-1}$ such that
\begin{multline}
\label{Morse_Bott_inequality_de_Rham_dimension_cohomology}
    M(\mathcal{P}_W(X),h)(b):=\Big( \sum_{a \in \mathcal{I}}  \sum_{k=0}^n b^k \dim(\mathcal{H}^{k}(\mathcal{P}_{W,B}(U(F_a))) \Big) \\= \sum_{k=0}^n b^k \dim(\mathcal{H}^{k}(\mathcal{P}_W(X))) + (1+b) \sum_{k=0}^{n-1} Q_k b^k.
\end{multline}
\end{theorem}

The \textit{inequalities} advertised by the theorem's label are those imposed on the coefficients of the polynomial on the left hand side by the fact that $Q_j \in \mathbb{N}_{0}$.

\begin{proof}[Proof of Theorem \ref{Theorem_strong_Morse_Bott_de_Rham}]
To prove the polynomial Morse inequalities, we apply equation \eqref{equation_with_the_b} of Theorem \ref{Lefschetz_supertrace}
\begin{equation}
\label{interesting_step_small_eig_complex}
    \mathcal{L}(\mathcal{P})(b,t)=L(\mathcal{P})(b)+ (1+b) \sum_{k=0}^{n-1} b^k S_k(t)
\end{equation}
to the complex \eqref{small_eigenvalue_complex}. Since this is a complex of finite dimensional Hilbert spaces, we can take $t$ to $0$ to see that the left hand side is exactly the expression
\begin{equation}
    \Big( \sum_{a \in \mathcal{I}}  \sum_{k=0}^n b^k \dim(\mathcal{H}^{k}(\mathcal{P}_{W,B,\varepsilon}(U_{a,z}))) \Big)
\end{equation}
and the right hand side is of the form
\begin{equation}
    \sum_{k=0}^n b^k \dim(\mathcal{H}^{k}(\mathcal{P}_{W,\varepsilon}(X))) + (1+b) \sum_{k=0}^{n-1} Q_k b^k
\end{equation}
where $Q_k$ are non-negative integers, since they correspond to the dimension of the co-exact small eigenvalue eigensections of the deformed Laplace-type operator. Since the cohomology groups of the Witten deformed and the undeformed complexes have the same dimensions for both the local and global complexes we see that this proves equation \eqref{Morse_Bott_inequality_de_Rham_dimension_cohomology}.
\end{proof}

\begin{remark}
\label{Remark_on_adjoint_inequalities_proof}
We briefly explain in simple terms why the relation $b^nM(\mathcal{P}_W(X),h)(b^{-1})=M(\mathcal{Q}_{W^{\perp}}(X),-h)(b)$ holds. The critical points of both $h$ and $-h$ are the same, but the attracting and expanding directions are reversed. However the local complexes corresponding to such a switch are adjoint complexes (see the constructions in Subsection \ref{subsubsubsection_Neumann_boundary_condition}) and thus the Poincar\'e polynomials for the local complexes (see Definition \ref{Definition_Morse_and_Poincare_polynomials}) which are also the local Morse polynomials satisfy a relation which taking sums over the $a \in \mathcal{I}$ gives the needed relation.
\end{remark}

The following corollary generalizes the well known result of Poincar\'e duality for Morse polynomials in the smooth setting, and is a version of Corollary \ref{corollary_Poincare_dual_inequalities_intro_version}. 

\begin{corollary}[Refined inequalities for self-dual mezzo perversities]
\label{corollary_Poincare_dual_inequalities}
In the same setting as Theorem \ref{Theorem_strong_Morse_Bott_de_Rham} 
if the mezzo-perversity $W$ is self-dual, we have that
\begin{equation}
    \label{equation_Poincare_duality_for_polynomials}
    M(\mathcal{P}_W(X),h)(b)=b^nM(\mathcal{P}_W(X),-h)(b^{-1}), 
\end{equation}
and there is a \textbf{refined polynomial Morse inequality}
\begin{equation}
\label{equation_refined_inequalities}
    M_{re}(\mathcal{P}_W(X),h,b)- \sum_{k=0}^n b^k dim(\mathcal{H}^{k}(\mathcal{P}_W(X))) = \sum_{k=0}^{n} \overline{Q}_{k} b^{k}
\end{equation}
where $\overline{Q}_{k}$ are non-negative integers.
\end{corollary}

\begin{proof}
Since the relation in Remark \ref{Remark_on_adjoint_inequalities_proof} holds for both $\mathcal{P}_W(X)$ and $\mathcal{P}_{\star W^{\perp}}(X)$, since self-dual mezzo perversities are those for which $W=\star W^{\perp}$, we get equation \ref{equation_Poincare_duality_for_polynomials}.

Then taking the minimal polynomials (see Definition \ref{Definition_refined_Morse_polynomials}) of the Morse polynomials for the complex $\mathcal{P}(X)$ with both $h$ and $-h$, together with the fact that the Poincar\'e polynomial is a set of common terms in both inequalities yields equation \ref{equation_refined_inequalities}.
\end{proof}

\section{Guide for computations}
\label{Section_computational_guide}

In Examples \ref{Example_suspension_torus_Witten_deformation} (c.f. \ref{Example_suspension_over_torus_first}) we computed local and global cohomology groups. In this section we give a brief guide on how to compute local Morse polynomials more generally.

\subsection{Recovering the smooth Morse inequalities}
It is well known that the contribution to the Morse polynomial from a critical point $a$ of a Morse function $h$ is $b^{M_a}$ where $M_a$ is the \textit{Morse index} of $a$. This is simply the number of negative eigenvalues of the Hessian of $h$ at $a$, which is well defined, independent of the coordinate chart chosen to compute this number. The Morse lemma can be used to write $h=x^2-r^2$ in a chart near the critical point, where $x$ and $r$ are radial functions on two discs $\mathbb{D}^{k_1}_x$ and $\mathbb{D}^{k_2}_r$, the product of which is a neighbourhood of the critical point, and it is easy to check that $k_2=M_a$. 
Since the fundamental neighbourhood is simply the product of the two discs, the local cohomology is given by the product of the absolute cohomology of the disc $\mathbb{D}^{k_1}_x$ which is the class generated by the constants in degree $0$, and the relative cohomology of the disc $\mathbb{D}^{k_2}_r$ which is the class generated by the volume form in degree $k_2$.
Thus the contribution $b^{M_a}$ is simply the polynomial trace over the local cohomology.
This volume form can be identified with the orientation of the expanding fibers, and this is the perspective in \cite{bismut1986witten}.

If we consider Morse-Bott functions where the fibrations of the fundamental neighbourhoods $U(F_a)$ are trivial fibrations, then each such neighbourhood decomposes into a product of the base of the fibration ($F_a$) and the normal fiber (where the fibers are all homeomorphic), we can use the K\"unneth formula to compute the local cohomology. In this case the Morse-Bott lemma shows that the normal fibers can be understood similar to those in the Morse case (where the base is a point). Thus the computation boils down to figuring out the cohomology of the base.
More generally for any Morse-Bott function, the local cohomology is the cohomology of the base $F_a$ with coefficients twisted by the cohomology of the complex on the normal fibers, which in the smooth setting is simply the orientation of the expanding fibers.

If there are \textit{nice} assumptions on the critical point sets, such as having an isometric action of the fundamental group of the base on the fibers, or if there is a smooth and proper action of a Lie group on the fibers, then the local cohomology groups can be computed more easily (see, e.g., \cite{BanaglFlatfiber2013} for more details).

The following example can be used to see that the refined Morse polynomial can have an error polynomial with no $(1+b)$ factor.

\begin{example}
\label{Example_two_horned_torus}
Consider a height function $h$ on a torus which has Morse polynomial $M(\mathcal{P}(X),h)(b)=1+3b+2b^2$ (this is easy to construct). The Morse polynomial for $-h$ is $M(\mathcal{P}(X),-h)(b)=2+3b+b^2$, and it is easy to check that $b^2 M\mathcal{P}(X),(h)(b^{-1})=M(\mathcal{P}(X),-h)(b)$.
We also see that the refined Morse polynomial is $1+3b^1+b^1$, which is equal to the Poincar\'e polynomial up to the error polynomial which is simply $b^1$.
\end{example}

\subsection{Case of singular critical points}

We can compute the local cohomology groups at isolated critical points and the cohomology of fibers corresponding to the complexes we consider using the following proposition.

\begin{proposition}  
\label{proposition_cohomology_of_cone}
Consider the complex $\mathcal{P}_{W,N}(U(F_a))=(L^2\Omega^{\cdot}(U(F_a);\mathbb{C}), \mathcal{D}_{W,N}(d),d)$ where $U(F_a)=C_x(Z)$, where $Z$ is a stratified space equipped with a mezzo-perversity 9in the notation of Subsection \ref{subsection_depth_2}) $W=(W^1,...W^j)$ where $C(Z)$ has depth $j$.
Then $Z$ can be equipped with the mezzo-perversity $V=(W^1,...,W^{j-1})$,
which determines a complex on the link $\mathcal{P}_{V}(Z)=(L^2\Omega^{\cdot}(Z;\mathbb{C}), \mathcal{D}_{V}(P_Z),P_Z)$, where $P_Z$ is the de Rham operator on $Z$.
Then the cohomology of the cone with generalized Neumann boundary conditions is 
\begin{equation}
\mathcal{H}^k(\mathcal{P}_{W,N}(U(F_a)))=
\begin{cases}
      0, & \text{for}\ k > l/2 \\  
      W^j, & \text{for}\ k = l/2 \\
      \mathcal{H}^k(\mathcal{P}_{V}(Z)), & \text{for}\ k < l/2
\end{cases}
\end{equation}
where $l$ is the dimension of the link $l$ (where $W^j=\{0\}$ is $l$ if odd), and for the generalized Dirichlet boundary conditions we have
\begin{equation}
\mathcal{H}^k(\mathcal{P}_{W,D}(U(F_a)))=
\begin{cases}
      dx \wedge x^{2(k-1)-l}\mathcal{H}^{k-1}(\mathcal{P}_{V}(Z)), & \text{for}\ k > l/2+1 \\  
      dx \wedge {W^j}^{\perp}, & \text{for}\ k = l/2+1 \\
      0, & \text{for}\ k < l/2+1
\end{cases}
\end{equation}
where $dx \wedge W^j$ and represents the forms $dx \wedge \beta$ for $\beta \in W^j$ ((similarly $dx \wedge x^{2(k-1)-l} \mathcal{H}^k(\mathcal{P}_{V}(Z))$)). 
\end{proposition}

This shows how to compute the cohomology for the attracting and expanding factors of a fundamental neighbourhood of an isolated critical point, and thus the local cohomology in order to compute the local Morse polynomial.

\begin{proof}
The proof of a simpler version of this result in the Witt case was given in \cite[\S 6.4]{jayasinghe2023l2} (see Lemma 6.16 of loc. cit. and the subsequent discussion for the cohomology with the adjoint boundary conditions). The proof of this more general result is similar, where the difference is that the additional cohomology due to the mezzo-perversities have to be kept track of. However this is simply a computation following the definitions.
\end{proof}

The following examples, related to Examples \ref{Example_suspension_over_torus_first} and \ref{Example_suspension_torus_Witten_deformation}, serve to build intuition of the above result.

\begin{example}
\label{Example_suspension_torus_2}
Consider a suspension over a torus $\Sigma_r(\mathbb{T}^2_{\theta_1, \theta_2})$ equipped with the wedge metric 
\begin{equation}
    dr^2+\sin^2(r)(d\theta_1^2+d\theta_2^2),
\end{equation}
the resolution of which is $[0,\pi]_r \times \mathbb{T}^2$ where the mezzo-perversity is $W=(W^1)$ where $W^1=d\theta_1$. We consider the Morse function $h=\cos(r)$ on this space. Then we consider the attracting critical point at $r=\pi$. The local cohomology at this point is 
\begin{equation}
\mathcal{H}^k(\mathcal{P}_{W,N}(C(Z)))=
\begin{cases}
      0, & \text{for}\ k > 1 \\  
      d\theta_1  & \text{for}\ k =1\\
      1, & \text{for}\ k =0
\end{cases}
\end{equation}
where $Z=\mathbb{T}^2$.
Similarly the local cohomology at $r=0$ is 
\begin{equation}
\mathcal{H}^k(\mathcal{P}_{W,D}(C_r(Z)))=
\begin{cases}
      dr \wedge r d\theta_1 \wedge r d\theta_2, & \text{for}\ k =3\\  
      dr\wedge d\theta_2  & \text{for}\ k =2\\      
      0, & \text{for}\ k \leq 1.
\end{cases}
\end{equation}
Thus we can compute the Morse polynomial, which is $b^0+b^1+b^2+b^3$. In fact this is a perfect Morse function, and the global cohomology is given by the forms $1,d\theta_1, dr \wedge \sin(r) d\theta_2$ and the volume form.
\end{example}

We can build on the previous example to make the following one.

\begin{example}
\label{example_suspension_suspension_torus_2}
Consider a suspension over the suspension of a torus $\Sigma_x(\Sigma_r(\mathbb{T}^2_{\theta_1, \theta_2}))$ equipped with the wedge metric 
\begin{equation}
    dx^2+\sin^2(x)(dr^2+\sin^2(r)(d\theta_1^2+d\theta_2^2)),
\end{equation}
the resolution of which is $[0,\pi]_x \times [0,\pi]_r \times \mathbb{T}^2$ where the mezzo-perversity is $W=(W^1,W^2)$ where $W^1=d\theta_1,W^2=\{0\}$. We consider the Morse function $h=\cos(x)$ on this space. Then we consider the attracting critical point at $x=\pi$. The local cohomology at this point is 
\begin{equation}
\mathcal{H}^k(\mathcal{P}_{W,N}(C(Z)))=
\begin{cases}
      0, & \text{for}\ k > 3/2 \\  
      d\theta_1  & \text{for}\ k =1\\
      1, & \text{for}\ k =0
\end{cases}
\end{equation}
where $Z=\Sigma_r(\mathbb{T}^2)$ where if we take the conic metric, we must replace 
Similarly the local cohomology at $x=0$ is 
\begin{equation}
\mathcal{H}^k(\mathcal{P}_{W,D}(C_x(Z)))=
\begin{cases}
      dx \wedge x^3 dr \wedge \sin^2(r) d\theta_1 \wedge d\theta_2, & \text{for}\ k =4\\  
      dx\wedge x^2 dr \wedge \sin(r) d\theta_2  & \text{for}\ k =3\\
      0, & \text{for}\ k \leq 2.5.
\end{cases}
\end{equation}
The Morse polynomial is $b^0+b^1+b^3+b^4$ and it is in fact a perfect Morse function.
The global cohomology of the space is generated by $1, d\theta_1, dx \wedge \sin^2(x)dr \wedge \sin(r) d\theta_2$ and the volume form $dx \wedge \sin^3(x) dr \wedge \sin^2(r) d\theta_1 \wedge d\theta_2$
\end{example}

\begin{example}
\label{example_breakdown_refined_inequalities}
We build on Example \ref{Example_suspension_over_torus_first}. With the mezzo-perversity given in that example, it is easy to check that the local and global complexes are self-dual. The global Morse polynomial is $b^3 M(\mathcal{P}_{W},h)(b^{-1})=M(\mathcal{P}_{W^{\perp}},-h)(b)=1+b^1+b^2+b^3$ and the Morse-function is perfect.

Now we consider the same space and the height function as the Morse function, but we consider a different mezzo-perversity corresponding to picking $V=\mathcal{H}^1(Z)$ at the two singular strata.
In this case, the local Morse polynomial at the critical point at the \textit{south pole} is $1+2b$ while at the \textit{north pole} we get $b^3$. The resulting Morse polynomial is $M(\mathcal{P}_{V},h)(b)=1+2b+b^3$.
If we consider the Morse polynomial for the function $-h$ with the mezzo-perversity $\{0\}$, we get the Morse polynomial $M(\mathcal{P}_{\{0\}},h)(b)=1+2b^2+b^3$. It is easy to check that $b^3M(\mathcal{P}_{V},h)(b^{-1})=M(\mathcal{P}_{0},-h)(b)$.
\end{example}

We can easily construct examples of flat stratified Morse-Bott functions by taking products of the spaces $X$ above with $S^1$, and extending the Morse functions to the product space by pullback along the projection from $X \times S^1$ to $X$, which then is a flat stratified Morse-Bott function.
It is easy to compute the local and global cohomology groups using the K\"unneth formula in this setting.
While the following example has a stratified Morse-Bott function, it does not satisfy the flatness assumption in our setting. However we can verify that the Morse inequalities as stated in Theorem \ref{Theorem_strong_Morse_Bott_de_Rham_intro_version} extends to this example where it also follows from \cite{kirwan1988intersection}.

\begin{figure}[h]
    \centering
    \includegraphics[scale=.35]{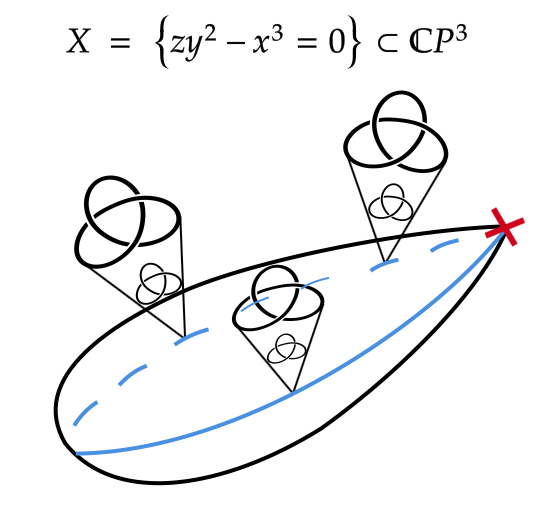}
    \caption{A singular projective variety with stratification $X_4\supset X_2\supset X_0$. The two-dimensional piece $X_2=\{ x=y= 0 \}$ is cut out by the blue curve, and a neighborhood of $X_2$ fibers over $X_2$, with fibers given by cones over a trefoil knot.}
    \label{fig:enter-label}
\end{figure}

\begin{example}
\label{example_singular_Fano}
Consider the projective variety $zy^2-x^3=0$ in $\mathbb{CP}^3$ with coordinates $[w:x:y:z]$, which admits the algebraic torus action $(\lambda,\mu)\cdot [w:x:y:z]=[w:\lambda^3x:\mu^3y:\lambda^2\mu z]$. The circle action corresponding to $\mu=1$, $\lambda=e^{i\theta}$ is a K\"ahler Hamiltonian action corresponding to a function $h$ with a non-isolated critical point set corresponding to the set $\{x=z=0\}$ which is a singular subvariety of the projective variety which has a depth 2 singularity at $[1:0:0:0]$. The Fubini-Study metric on this space is a wedge metric, but it doesn't satisfy the geometric Witt condition. However we can put a different wedge metric on this space and find a perturbation of $h$ which is a stratified Morse-Bott function.
In this case it is possible to check that there are two forms of degrees $0$ and $2$ in the local cohomology group at the non-isolated critical point set, and a form of degree $4$ in the local cohomology at the isolated critical point. Thus the Morse polynomial is $M(b)=b^0+b^2+b^4$, which is a perfect stratified Morse-Bott function.
\end{example}

In the same example, there are wedge K\"ahler Hamiltonian Morse functions corresponding to the torus action when $\mu$ is not fixed to be $1$, and one can check that there are three critical points, one on each stratum. It is easy to check that the one on the stratum of depth $1$ is a stratified non-degenerate critical point set of type I (in fact this is the case for any isolated critical points which are neither on the regular stratum stratum of dimension $0$). Similar examples appear in \cite{jayasinghe2023l2,jayasinghe2024holomorphic} to which we refer the reader for more examples.
Taking a product of such spaces with stratified Morse functions (with isolated critical points) with compact manifolds we can construct examples where the critical points are not isolated, where the flatness condition we study here is met, and in fact where the fibrations appearing in the fundamental neighbourhoods are trivial.

\bibliographystyle{alpha}
\bibliography{references}
\end{document}